\documentclass[12pt,reqno]{amsart}

\usepackage{color}
\usepackage{a4wide}
\usepackage{dsfont}
\usepackage{prettyref}
\usepackage{amssymb}
\usepackage{mathrsfs}

\usepackage[bookmarks=false,breaklinks=false,pdfborder={0 0 1},backref=page,colorlinks=true]{hyperref}
\hypersetup{pdfstartview={XYZ null null 1.25},citecolor=OliveGreen,linkcolor=RoyalBlue,urlcolor=blue}

\usepackage[active]{srcltx} 

\usepackage[all]{xy}
\usepackage{graphicx}
\usepackage{mathtools} 
\usepackage[usenames,dvipsnames]{xcolor}
\usepackage{upgreek}
\usepackage{stmaryrd} 

\newcommand{\Vv}{{\tilde{U}}}

\usepackage{enumitem}		

\let\ldash=\l
\let\cedilla=\c
\let\Hung=\H

\renewcommand{\theenumi}{\tu{(\alph{enumi})}}
\renewcommand{\labelenumi}{\theenumi}

\newrefformat{eq}{\eqref{#1}}
\newrefformat{def}{Definition \ref{#1}}
\newrefformat{lem}{Lemma \ref{#1}}
\newrefformat{pro}{Proposition \ref{#1}}
\newrefformat{prop}{Proposition \ref{#1}}
\newrefformat{cor}{Corollary \ref{#1}}
\newrefformat{thm}{Theorem \ref{#1}}
\newrefformat{exa}{Example \ref{#1}}
\newrefformat{rem}{Remark \ref{#1}}
\newrefformat{sec}{Section \ref{#1}}
\newrefformat{sub}{Subsection \ref{#1}}
\newrefformat{conj}{Conjecture \ref{#1}}
\newrefformat{desc}{\ref{#1}}
\newrefformat{ques}{Question \ref{#1}}

\newcommand*{\doi}[1]{\href{http://dx.doi.org/#1}{doi: #1}}

\theoremstyle{plain}
\newtheorem{tw}{Theorem}[section]
\newtheorem{thm}[tw]{Theorem}
\newtheorem {lem} [tw]{Lemma}
\newtheorem {prop}[tw] {Proposition}

\newtheorem{cor}[tw]{Corollary}

\newtheorem*{theorem*}{Theorem}

\theoremstyle{definition}
\newtheorem {deft}[tw]{Definition}
\newtheorem {defn}[tw]{Definition}
\newtheorem {rem}[tw]{Remark}

\newcommand{\WW}{{\mathds{V}\!\!\text{\reflectbox{$\mathds{V}$}}}}
\newcommand{\Ww}{\mathds{W}}
\newcommand{\wW}{\text{\reflectbox{$\Ww$}}\:\!} 

\newcommand{\cst}{\ifmmode\mathrm{C}^*\else{$\mathrm{C}^*$}\fi}

\newcommand{\CC}{\mathbb{C}}

\newcommand{\ot}{\otimes}

\newcommand{\Tinv}{\mathsf{T}}
\newcommand{\la}{\langle}
\newcommand{\ra}{\rangle}

\newcommand{\id}{\mathrm{id}}

\newcommand{\Ind}{\mathcal{I}}
\newcommand{\Jnd}{\mathcal{J}}
\newcommand{\hh}[1]{\widehat{#1}}
\newcommand{\Hil}{\mathsf{H}}
\newcommand{\sH}{\mathsf{H}}

\newcommand{\RepGH}{\tu{Rep}_{\QG} \Hil}

\newcommand{\Com}{\Delta}

\newcommand{\B}{\mc B}

\newcommand{\Kil}{\mathsf{K}}
\newcommand{\sK}{\mathsf{K}}

\newcommand{\br}{\mathbb{R}}

\newcommand{\cA}{\mathscr{A}}

\newcommand{\sA}{\mathsf{A}}
\newcommand{\bn}{\mathbb{N}}
\newcommand{\bc}{\mathbb{C}}
\newcommand{\Pol}{\tu{Pol}}

\newcommand{\ol}{\overline}
\newcommand{\wt}{\widetilde}
\newcommand{\hCou}{\widehat{\epsilon}}

\newcommand{\alg} {\mathsf{A}}
\newcommand{\blg} {\mathsf{B}}

\newcommand{\G}{\mathbb{G}}
\newcommand{\QH}{\mathbb{H}}

\newcommand{\sB}{\mathsf{B}}

\newcommand{\tu}{\textup}

\newcommand{\QG}{\mathbb{G}}
\newcommand{\hQG}{\hh{\mathbb{G}}}

\newcommand{\mc}{\mathcal}
\newcommand{\hQH}{\hh{\mathbb{H}}}

\newcommand{\Cou}{\epsilon}

 \newcommand{\N}{\mathbb{N}}
\DeclareMathOperator{\C}{C}
\DeclareMathOperator{\c0}{c_0}

\DeclareMathOperator{\M}{M}
\DeclareMathOperator{\Mor}{Mor}
\DeclareMathOperator{\Irr}{Irr}

\numberwithin{equation}{section}

\newcommand{\inv}{\operatorname{Inv}}

\newcommand{\ip}[2]{\langle{#1},{#2}\rangle}

\global\long\def\e{\varepsilon}
\global\long\def\RR{\mathbb{R}}
\global\long\def\H{\mathsf{H}}
\global\long\def\J{\mathcal{J}}
\global\long\def\K{\mathsf{K}}
\global\long\def\a{\alpha}
\global\long\def\be{\beta}
\global\long\def\om{\omega}
\global\long\def\z{\zeta}
\global\long\def\gnsmap{\upeta}

\global\long\def\conv{*}

\global\long\def\Ree{\operatorname{Re}}
\global\long\def\Img{\operatorname{Im}}
\global\long\def\linspan{\operatorname{span}}

\global\long\def\presb#1#2{\prescript{}{#1}{#2}}

\global\long\def\tensor{\otimes}
\global\long\def\tensorn{\mathbin{\overline{\otimes}}}

\global\long\def\i{\mathrm{id}}

\global\long\def\one{\mathds{1}}
\global\long\def\tr{\mathrm{tr}}

\global\long\def\op{\mathrm{op}}

\global\long\def\CUcomp#1{\mathrm{C}^{u}(#1)}

\global\long\def\Linfty#1{L^{\infty}(#1)}
\global\long\def\Lone#1{L^{1}(#1)}
\global\long\def\LoneSharp#1{L_{\sharp}^{1}(#1)}
\global\long\def\Ltwo#1{L^{2}(#1)}
\global\long\def\Cz#1{\mathrm{C}_{0}(#1)}
\global\long\def\CzU#1{\mathrm{C}_{0}^{u}(#1)}
\global\long\def\CU#1{\mathrm{C}^{u}(#1)}

\global\long\def\linfty#1{\ell^{\infty}(#1)}
\global\long\def\lone#1{\ell^{1}(#1)}
\global\long\def\ltwo#1{\ell^{2}(#1)}
\global\long\def\cz#1{\mathrm{c}_{0}(#1)}

\global\long\def \Irred#1{\Irr(#1)}

\global\long\def\VN#1{\mathrm{VN}(#1)}

\global\long\def\tp{\mathbin{\xymatrix{*+<.7ex>[o][F-]{\scriptstyle \top}}
 } }
\global\long\def\tpsmall{\mathbin{\xymatrix{*+<.5ex>[o][F-]{\scriptscriptstyle \top}}
 } }
\global\long\def\tpr{\mathbin{\xymatrix{*+<.7ex>[o][F-]{\scriptstyle \bot}}
 } }
\global\long\def\tprsmall{\mathbin{\xymatrix{*+<.5ex>[o][F-]{\scriptscriptstyle \bot}}
 } }

\begin{document}

\title{Around Property (T) for quantum groups}
\author{Matthew Daws}
\address{Leeds, United Kingdom}
\email{matt.daws@cantab.net}

\author{Adam Skalski}
\address{Institute of Mathematics of the Polish Academy of Sciences, ul.~\'Sniadeckich 8, 00-656 Warszawa, Poland}
\email{a.skalski@impan.pl}

\author{Ami Viselter}
\address{Department of Mathematics, University of Haifa, 31905 Haifa, Israel}
\email{aviselter@staff.haifa.ac.il}

\begin{abstract}
We study Property (T) for locally compact quantum groups, providing several new characterisations, especially related to operator algebraic ergodic theory. Quantum Property (T) is described in terms of the existence of various Kazhdan type pairs, and some earlier structural results of Kyed, Chen and Ng are strengthened and generalised. For second countable discrete unimodular quantum groups with low duals Property (T) is shown to be equivalent to Property (T)$^{1,1}$ of Bekka and Valette. This is used to extend to this class of quantum groups classical theorems on `typical' representations (due to Kerr and Pichot), and on connections of Property (T) with spectral gaps (due to Li and Ng) and with strong ergodicity of weakly mixing actions on a particular von Neumann algebra (due to Connes and Weiss). Finally we discuss in the Appendix equivalent characterisations of the notion of a quantum group morphism with dense image.
\end{abstract}

\maketitle

\section*{Introduction}
The discovery of Property (T) by Kazhdan in \cite{Kazhdan__T} was a major advance in group theory. It has numerous applications, in particular in abstract harmonic analysis, ergodic theory and operator algebras, which can be found in the recent book \cite{Bekka_de_la_Harpe_Valette__book} of Bekka, de~la Harpe and Valette dedicated to this property.
Recall that a locally compact group $G$ has \emph{Property (T)} if every unitary representation of $G$ that has almost-invariant vectors actually has a non-zero invariant vector. Property (T) is understood as a very strong type of \emph{rigidity}. It is `antipodal' to \emph{softness} properties such as amenability -- indeed, the only locally compact groups that have both properties are the ones that have them trivially, namely compact groups.
Property (T) for $G$ has many equivalent conditions (some under mild hypotheses), among them are the following: the trivial representation is isolated in $\widehat G$; every net of normalised positive-definite functions on $G$, converging to $1$ uniformly on compact sets, converges uniformly on all of $G$; there exists a compact Kazhdan pair for $G$; there exists a compact Kazhdan pair with `continuity constants'; for every unitary representation $\pi$ of $G$, the first cohomology space $H^1(G,\pi)$ of $1$-cocycles vanishes; and all conditionally negative definite functions on $G$ are bounded.

In the context of von Neumann algebras, Property (T) was employed by Connes in \cite{Connes__fundamental_group} to establish the first rigidity phenomenon of von Neumann algebras, namely the countability of the fundamental group. Property (T) and its relative versions have been introduced for von Neumann algebras in \cite{Connes__classification_des_facteurs,Connes_Jones__prop_T_vN_alg,Popa_correspondences,Anan_Delaroche__on_Connes_prop_T_vN_alg,Peterson_Popa__relative_prop_T_vN_alg,Popa__factors_Betti_num}. The various forms of Property (T) have been a key component of Popa's deformation/rigidity theory (see for example \cite{Popa__factors_Betti_num}  and references therein). Thus now Property (T) appears in operator algebras both as an important ingredient related to the groups and their actions, and as a notion intrinsic to the theory.

It is therefore very natural to consider Property (T) for \emph{locally compact quantum groups} in the sense of Kustermans and Vaes. The definition is the same as for groups, using the quantum notion of a unitary (co-) representation. Property (T) was first introduced for Kac algebras by Petrescu and Joi{\cedilla{t}}a in \cite{Petrescu_Joita__prop_T_Kac_alg} and then for algebraic quantum groups by B{\'e}dos, Conti and Tuset \cite{Bedos_Conti_Tuset__amen_co_amen_alg_QGs_coreps}. Property (T) for discrete quantum groups in particular was studied in depth by Fima \cite{Fima__prop_T}, and later by Kyed \cite{Kyed__cohom_prop_T_QG} and  Kyed with So{\l}tan \cite{Kyed_Soltan__prop_T_exotic_QG_norms}.  They established various characterisations of Property (T), extending the ones mentioned above and others, produced examples, and demonstrated the usefulness of Property (T) in the quantum setting.
Property (T) for general locally compact quantum groups was first formally defined by Daws, Fima, Skalski and White \cite{Daws_Fima_Skalski_White_Haagerup_LCQG} and then studied by Chen and Ng \cite{Chen_Ng__prop_T_LCQGs}.
Recently  Arano \cite{Arano__unit_sph_rep_Drinfeld_dbls} and Fima, Mukherjee and Patri \cite{Fima_Mukherjee_Patri_compact_bicross_prod} developed new methods of constructing examples of Property (T) quantum groups; the first of these papers showed also that for unimodular discrete quantum groups Property (T) is equivalent to \emph{Central} Property (T), which will be of importance for the second half of our paper. It is also worth noting that these results spurred interest in studying Property (T) also in the context of general \cst-categories (\cite{Popa_Vaes__rep_thy_subfactors_latt_tens_cat,Neshveyev_Yamashita__Drinfeld_ctr_rep_mon_cat}).

Most of our results establish further characterisations of Property (T), which can be roughly divided into two classes. The first is based on conditions that do not involve representations directly. These are interesting in their own right, and also provide a tool to approach the conditions in the second class. The latter includes conditions that are a formal weakening of Property (T) obtained either by restricting the class of considered representations or by replacing lack of ergodicity by a weaker requirement. Some of these results extend celebrated theorems about groups that have proven to admit many applications. The most crucial single result along these lines is a generalisation of a theorem of Bekka and Valette \cite{Bekka_Valette__prop_T_amen_rep}, \cite[Theorem 2.12.9]{Bekka_de_la_Harpe_Valette__book}, stating that for $\sigma$-compact locally compact groups Property (T) is equivalent to the condition that every unitary representation with almost-invariant vectors is \emph{not weakly mixing}.

The paper is structured as follows. Its first part, consisting of Sections \ref{sec:prelim}--\ref{sec:actions}, is devoted to background material, and the main results are presented in its second part, consisting of Sections \ref{sec:prop_T}--\ref{sec:Connes_Weiss}.
\prettyref{sec:prelim} discusses preliminaries on locally compact quantum groups, in particular their representations. Sections \ref{sec:inv_alm_inv_vects} and \ref{sec:actions} contain more preliminaries as well as new results on (almost-) invariant vectors and actions, respectively.
\prettyref{sec:prop_T} introduces various notions of what a `Kazhdan pair' for a locally compact quantum group should be, and shows that each of these is equivalent to Property (T). It also proves that for locally compact quantum groups $\QH$ and $\QG$, if $\QH$ has Property (T) and there is a morphism from $\QH$ to $\QG$ `with dense range', then $\QG$ has Property (T) -- this substantially strengthens one of the results of \cite{Chen_Ng__prop_T_LCQGs}.
In \prettyref{sec:Kazhdan_pairs_apps}, Kazhdan pairs are applied to present characterisations of Property (T) in terms of positive-definite functions and generating functionals of particular types, at least for particular classes of locally compact quantum groups.

For a discrete quantum group, Fima showed in \cite{Fima__prop_T} that Property (T) implies that the quantum group is finitely generated, thus second countable, and unimodular.  In \prettyref{sec:TT11} we introduce a further notion for a discrete quantum group -- that of having a `low dual'.  We then provide the above-mentioned generalisation of the Bekka--Valette theorem for second countable unimodular discrete quantum groups with low duals (\prettyref{thm:T11difficult}); its consequences -- Theorems \ref{thm:dense_gdelta}, \ref{thm:property_T_chars} and \ref{thm:Connes_Weiss} -- will also have these hypotheses on the quantum group. Our proof is closer to those of \cite[Theorem 1.2]{Jolissaint__prop_T_pairs} (see also \cite[Lemma 4.4]{Jolissaint__prop_T_discrete_grps_reg_rep}), \cite[Theorem 9]{Bekka__prop_T_cstr_alg} and \cite{Peterson_Popa__relative_prop_T_vN_alg} than to the original one \cite{Bekka_Valette__prop_T_amen_rep}.  The low dual condition is a significant restriction, but nevertheless still allows the construction of non-trivial examples -- see Remark~\ref{Remark:low}. The attempts to drop it led us to believe that in fact it might be necessary for the quantum Bekka--Valette theorem to hold; however we do not have a theorem or a counterexample which would confirm or disprove this intuition.

As a first consequence of our version of the Bekka--Valette theorem we establish that lack of Property (T) is equivalent to the genericity of weak mixing for unitary representations (\prettyref{thm:dense_gdelta}). This extends a remarkable theorem of Kerr and Pichot \cite{Kerr_Pichot__asym_abel_WM_T}, which goes back to a famous result of Halmos about the integers \cite{Halmos__gen_meas_pres_trans_mixing}. Next, in Sections \ref{sec:spectral_gaps} and \ref{sec:Connes_Weiss} we prove several results roughly saying that to deduce Property (T) it is enough to consider certain \emph{actions} of the quantum group.  Since, as is the case for groups, every action of a locally compact quantum group is unitarily implemented on a suitable Hilbert space, the conditions obtained this way are indeed formally weaker than Property (T).

Motivated by Li and Ng \cite{Li_Ng__spect_gap_inv_states}, \prettyref{sec:spectral_gaps} deals with spectral gaps for representations and actions of locally compact quantum groups. After proving a few general results, we show that one can characterise Property (T) by only considering actions of the (discrete) quantum group on $\B(\K)$ with $\K$ being a Hilbert space (\prettyref{thm:property_T_chars}). The purpose of \prettyref{sec:Connes_Weiss} is to present a non-commutative analogue of the influential theorem of Connes and Weiss \cite{Connes_Weiss}, \cite[Theorem 6.3.4]{Bekka_de_la_Harpe_Valette__book}, originally stated only for discrete groups. The latter asserts that a second countable locally compact group $G$ has Property (T) if and only if every ergodic (or even weakly mixing) measure-preserving action of $G$ on a probability space is strongly ergodic; note that strong ergodicity is weaker than the absence of almost-invariant vectors. Probabilistic tools play a central role in the proof of the Connes--Weiss theorem. Naturally, in the more general, highly non-commutative framework of locally compact quantum groups, actions on probability spaces will not do. Instead, our result (\prettyref{thm:Connes_Weiss}) talks about actions on a particular case of Shlyakhtenko's free Araki--Woods factors, namely $\VN{\mathbb{F}_{\infty}}$, that preserve the trace. Thus, probability theory is essentially replaced by free probability theory. What makes it possible is a result of Vaes \cite{Vaes_strict_out_act} that facilitates the construction of such actions.

The Appendix treats our notion of a morphism between two locally compact quantum groups having dense range. It presents several equivalent conditions, one of which is required in \prettyref{sec:prop_T}.

\subsection*{Acknowledgements} The work on this paper was initiated during the visit of AS to the University of Leeds in June 2012, funded by the EPSRC grant EP/I026819/1.
We thank Stuart White and Jan Cameron for valuable input at that time; we also acknowledge several later discussions with Stuart White without which this paper would have never been written.
AV thanks Orr Shalit and Baruch Solel for discussions on issues related to this work.
Finally, we are grateful to the referees for their remarks.
AS  was partially supported by the NCN (National Centre of Science) grant
2014/14/E/ST1/00525.

\section{Preliminaries}\label{sec:prelim}

We begin by establishing some conventions and notation to be used throughout the paper.
Inner products will be linear on the right side. For a Hilbert space $\Hil$ the symbols $\mc B(\Hil), \mc K(\Hil)$ will denote the algebras of bounded operators and compact operators on $\Hil$, respectively, and if $\xi, \eta \in \Hil$, then $\omega_{\xi,\eta}\in \mc B(\Hil)_*$ will be the usual vector functional, $T\mapsto \la\xi, T \eta\ra$. We will also write simply $\omega_\xi$ for $\omega_{\xi, \xi}$.

The symbols $\ot,\tensorn$ will denote the spatial/minimal tensor product of $\cst$-algebras and the normal spatial tensor product of von Neumann algebras, respectively. If $\alg$ is a $\cst$-algebra then $\M(\alg)$ denotes its multiplier algebra. For $\om \in \alg^*$, we denote by $\overline{\om} \in \alg^*$ its adjoint given by $\overline{\om}(x) := \overline{\om(x^*)}$, $x \in \alg$.

A \emph{morphism} between $\cst$-algebras $\alg$ and $\blg$ is a non-degenerate $*$-homomorphism from $\alg$ to $\M(\blg)$, and the set of all such maps is denoted by $\Mor(\alg,\blg)$. Representations of $\cst$-algebras in this paper are always non-degenerate $*$-representations.
We will tacitly identify elements of $\alg^*$ and $\Mor(\alg,\blg)$ with their unique extensions to $\M(\alg)$ that are strictly continuous on the closed unit ball \cite[Proposition 2.5 and Corollary 5.7]{Lance}. This allows composition of morphisms in the usual manner. We will freely use the basic facts of the theory of multiplier algebras and Hilbert modules -- again, see \cite{Lance} for details. We will also occasionally use the `leg' notation for operators acting on tensor products of Hilbert spaces or $\cst$-algebras.

For a normal, semi-finite, faithful (n.s.f.)~weight $\theta$ on a
von Neumann algebra $N$, we denote by $\left(\Ltwo{N,\theta},\gnsmap_{\theta}\right)$
the associated GNS construction and by $\mathcal{N}_{\theta}$ the left
ideal $\left\{ x\in N:\theta(x^{*}x)<\infty\right\} $. We write $\nabla_{\theta}$,
$J_{\theta}$ and $T_{\theta}$ for the modular operator and modular
conjugation of $\theta$ and for the closure of the anti-linear map $\gnsmap_{\theta}(x)\mapsto\gnsmap_{\theta}(x^{*})$,
$x\in\mathcal{N}_{\theta}\cap\mathcal{N}_{\theta}^{*}$, respectively,
all acting on $\Ltwo{N,\theta}$ \cite{Stratila__mod_thy,Takesaki__book_vol_2}.
In particular, $T_{\theta}=J_{\theta}\nabla_{\theta}^{1/2}$. For
a locally compact quantum group $\G$, we write $\nabla,J,T,\gnsmap$
for the objects associated with the left Haar weight of $\G$ (see the following subsection).

\subsection{Quantum groups} \label{Subsect:quantum groups}
The fundamental objects studied in this paper will be \emph{locally compact quantum groups} in the sense of Kustermans and Vaes. We refer the reader to \cite{Kustermans_Vaes__LCQG_C_star,Kustermans_Vaes__LCQG_von_Neumann,Van_Daele__LCQGs} and to the lecture notes \cite{Kustermans__LCQG_lecture_notes} for an introduction to the general theory -- we will in general use the conventions of these papers and also those of \cite{Daws_Fima_Skalski_White_Haagerup_LCQG}.

A locally compact quantum group is a pair $\G = (\Linfty{\G},\Delta)$, with $\Linfty{\G}$ being a von Neumann algebra and $\Delta:\Linfty{\G} \to \Linfty{\G} \tensorn \Linfty{\G}$ a unital normal $*$-homomorphism called the \emph{comultiplication} (or \emph{coproduct})
that is coassociative:
$$(\Delta\tensor\i)\Delta=(\i\tensor\Delta)\Delta,$$
such that there exist n.s.f.~weights $\varphi,\psi$ on $\Linfty{\G}$, called
the left and right \emph{Haar weights}, respectively, that satisfy
\[
\begin{gathered}\varphi((\om\tensor\i)\Delta(x))=\varphi(x)\om(\one)\quad\text{for all }\om\in\Linfty{\G}_{*}^{+},x\in\Linfty{\G}^{+}\text{ with }\varphi(x)<\infty,\\
\psi((\i\tensor\om)\Delta(x))=\psi(x)\om(\one)\quad\text{for all }\om\in\Linfty{\G}_{*}^{+},x\in\Linfty{\G}^{+}\text{ with }\psi(x)<\infty.
\end{gathered}
\]

For a locally compact quantum group $\QG$, the predual of $L^{\infty}(\QG)$ will be denoted by $\Lone{\QG}$ and the corresponding algebra of `continuous functions on $\QG$ vanishing at infinity', which is a weakly dense $\cst$-subalgebra of $\Linfty{\G}$, will be denoted by $\C_0(\QG)$. The comultiplication reduces to a map $\Delta\in\Mor\bigl(\C_0(\QG),\C_0(\QG)\ot \C_0(\QG)\bigr)$. The \emph{dual locally compact quantum group} of $\QG$ will be denoted by $\hh \QG$.
The GNS constructions with respect to the left Haar weights of $\G, \widehat\G$ give the same Hilbert space.
Therefore, we usually assume that both $\Linfty{\G}$ and $\Linfty{\widehat\G}$ act standardly on the Hilbert space $L^2(\QG)$.
The comultiplication is implemented by the
multiplicative unitary $W \in \M(\C_0(\QG)\otimes \C_0(\hh\QG))$: $\Com(x) = W^* (\one \ot x)W$ for all $x \in L^\infty(\QG)$. The multiplicative unitary of $\hQG$, to be denoted $\hh{W}$, is equal to $\sigma(W)^*$, where $\sigma:\M(\C_0(\QG)\otimes
\C_0(\hh\G)) \rightarrow \M(\C_0(\hh\QG) \otimes \C_0(\QG))$ is the flip map.
The `universal' version of $\C_0(\QG)$ (see \cite{Kustermans__LCQG_universal}) will be denoted by $\C_0^u(\QG)$, with the (coassociative) comultiplication $\Delta_u\in\Mor\bigl(\CzU{\QG},\CzU{\QG}\ot\CzU{\QG}\bigr)$, the canonical \emph{reducing morphism}
$\Lambda:\C_0^u(\G)\rightarrow \C_0(\G)$ and the \emph{counit} $\epsilon:\C_0^u(\G) \to \bc$.

The \emph{antipode} $S$ is an ultraweakly closed, densely defined, generally unbounded
linear operator on $\Linfty{\G}$. It maps a dense subspace of $\Cz{\G}$ into $\Cz{\G}$ and admits a `polar decomposition'
$S=R\circ\tau_{-i/2}$, where $R$ is an anti-automorphism of $\Linfty{\G}$ called the \emph{unitary antipode} and $\left(\tau_{t}\right)_{t\in\RR}$ is a group of automorphisms of $\Linfty{\G}$ called the \emph{scaling group}. In general the respective maps related to the dual quantum group will be adorned with hats, and the `universal versions' of objects we introduce will be adorned with the  index $u$, so that for example the right invariant weight on $\widehat{\QG}$ will be denoted by $\widehat{\psi}$ and the unitary antipode of $\C_0^u(\hQG)$ by $\widehat{R}^u$; if necessary we will also decorate the symbols with the index of the quantum group to which they are associated.

Both $R$ and $\tau_t$ leave $\Cz{\G}$ invariant and are implemented on $\Ltwo{\G}$ as follows (recall the convention introduced above: $\widehat{\nabla}$ is the modular operator of ($\Linfty{\hQG}, \widehat{\varphi})$):
\begin{equation} \label{eq:J_tau_impl}
R(x)=\widehat{J} x^* \widehat{J}, \quad \tau_t (x) = \widehat{\nabla}^{it} x \widehat{\nabla}^{-it} \qquad \forall_{x\in \Linfty{\G}} \; \forall_{t\in\RR}.
\end{equation}

Using the inclusion $\Cz{\G} \subseteq \Linfty{\G}$ and the epimorphism $\Lambda:\C_0^u(\G)\rightarrow \C_0(\G)$, one obtains the natural embeddings $\Lone{\G} \hookrightarrow \Cz{\G}^* \hookrightarrow \CzU{\G}^*$.

The comultiplication induces the \emph{convolution} product on $\CzU{\G}^*$ by the formula $\om_1 \conv \om_2 :=(\om_1 \tensor \om_2)\circ \Delta_u$ for $\om_1,\om_2 \in\CzU{\G}^*$, turning $\CzU{\G}^*$, $\Cz{\G}^*$ and $\Lone{\G}$ into completely contractive Banach algebras. The embeddings from the last paragraph respect the convolution product and make the smaller spaces closed ideals in the larger ones.
The subspace
\[
\LoneSharp{\G}:=\left\{ \om\in\Lone{\G}:\exists_{\rho\in\Lone{\G}} \; \forall_{x\in D(S)}\quad\rho(x)=\overline{\omega}(S(x))\right\}
\]
is a dense subalgebra of $\Lone{\G}$. For $\om \in \LoneSharp{\G}$,
let $\om^{\sharp}$ be the unique element $\rho\in\Lone{\G}$ such that
$\rho(x)=\overline{\om}(S(x))$ for each $x\in D(S)$. Then $\om\mapsto\om^{\sharp}$
is an involution on $\LoneSharp{\G}$.

Sometimes we will assume that the locally compact quantum groups we study are \emph{second countable}, by which we mean that $\C_0(\QG)$ is separable.

If the left and right Haar weights of $\QG$ coincide, we say that $\QG$ is \emph{unimodular}. In the case when $\QG$ is \emph{compact} (so that $\C_0(\QG)$ is unital, and we denote it simply by $\C(\QG)$) recall that $\QG$ is  \emph{of Kac type} if its antipode $S$ is bounded, or equivalently, $\widehat{\QG}$ is unimodular.

The multiplicative unitary $W$ admits `semi-universal' versions (see \cite{Kustermans__LCQG_universal}),
namely unitaries $\Ww \in \M(\C_0^u(\QG) \otimes \C_0(\hh\QG))$ and $\wW \in \M(\C_0(\QG) \otimes \C_0^u(\hh\QG))$, characterised by the following properties. For every $\cst$-algebra $\blg$, there is a bijection between
\begin{itemize}
\item unitary elements $U\in \M(\blg \otimes \C_0(\hh\QG))$ with
$(\id\otimes\hh\Delta)(U) = U_{13} U_{12}$ and
\item non-degenerate $*$-homomorphisms $\phi=\phi_U:\C_0^u(\QG)\rightarrow \M(\blg)$,
\end{itemize}
given by the relation $(\phi\otimes\id)(\Ww) = U$.

Similarly, for every $\cst$-algebra $\blg$, there is a bijection between
\begin{itemize}
\item unitary elements $U\in \M(\C_0(\QG) \otimes \blg)$ with
$(\Delta \otimes \id)(U) = U_{13} U_{23}$ and
\item non-degenerate $*$-homomorphisms $\phi=\phi_U:\C_0^u(\hh\QG)\rightarrow \M(\blg)$,
\end{itemize}
given by the relation $(\id \otimes \phi)(\wW) = U$.

There is also a truly
universal bicharacter $\WW\in\M(\C_0^u(\G)\otimes \C_0^u(\hh\G))$. It satisfies the equalities $\Ww = (\id\otimes\Lambda_{\hh\G})(\WW)$ and $\wW= (\Lambda_{\G}\otimes\id)(\WW)$.

We shall use repeatedly that $\Ww$ implements `one-half' of the coproduct at the
universal level, see \cite[Proposition~6.2]{Kustermans__LCQG_universal}, in that
\begin{equation} \label{reducedimplementation} (\id\otimes\Lambda_{\G})\Delta_u(x) = \Ww^*(\one\otimes\Lambda_{\G}(x))\Ww
\qquad (x\in \C_0^u(\G)). \end{equation}

An element $a \in \M (\C_0(\G))$ is said to be a \emph{positive-definite function/element} if there exists
$\mu \in \CzU{\hh{\G}}^*_+$ such that $a = (\i \tensor \mu)(\wW^*)$. We say that $a$ is \emph{normalised} if $\mu$ is a state. For more information on positive-definite functions, including equivalent definitions, see \cite{Daws__CPM_LCQGs_2012,Daws_Salmi__CPD_LCQGs_2013}.

\subsection{Representations of quantum groups}\label{sub:reps}
In this subsection we recall basic facts concerning unitary representations of quantum groups.

\begin{deft} A unitary representation of $\QG$ (or a unitary corepresentation
of $\C_0(\G)$) on a Hilbert space $\Hil$ is a unitary $U\in \M(\C_0(\QG)\otimes
\mathcal K(\Hil))$ with $(\Delta\otimes\id)U = U_{13} U_{23}$.  We will often write $\Hil_U$ for the Hilbert space upon which $U$ acts.
The trivial representation of $\QG$, i.e.\ $U=\one \ot 1 \in  \M(\C_0(\QG)\otimes
\mathcal \bc)$, will be denoted simply by $\one$.
\end{deft}

In the above definition, it suffices to require that $U \in L^\infty(\G) \tensorn \mc B(\Hil)$, as this automatically implies that $U\in \M(\C_0(\QG)\otimes \mathcal K(\Hil))$.  This is folklore, see for example \cite[Theorem 4.12]{Brannan_Daws_Samei__cb_rep_of_conv_alg_of_LCQGs}.

As in this paper we will only consider unitary representations, we will most of the time simply talk about `representations of $\QG$'.

\begin{deft}
The \emph{contragradient} representation of a
representation $U\in \M(\C_0(\QG)\otimes \mc K(\sH))$ is defined to be $U^c=(R\otimes\top)U
\in \M(\C_0(\QG) \otimes \mc K(\overline{\sH}))$.
Here $\overline{\sH}$ is the complex conjugate Hilbert space and $\top:\mc K(\sH) \rightarrow
\mc K(\overline{\sH})$ is the `transpose' map, defined by
$\top(x)(\overline{\xi}) = \overline{ x^*(\xi) }$.
\end{deft}

Evidently, if $\J$ is an anti-unitary from $\sH$ onto another Hilbert space $\J \sH$ and we consider the $*$-anti-isomorphism $j:\mc K(\sH) \to \mc K(\J \sH)$ given by $j(x) := \J x^* \J^*$, $x\in \mc K(\sH)$, then the representation $(R \otimes j)(U)$ is unitarily equivalent to $U^c$.

Recall that an anti-unitary on a Hilbert space is \emph{involutive} if its square is equal to $\one$.

\begin{deft}\label{def:cond_R}
Say that $U$ satisfies condition $\mathscr{R}$ if there exists an involutive anti-unitary $\J$ on $\sH$ such that, with $j : \mc K(\sH) \to \mc K(\sH)$ as above, we have $(R \otimes j)(U) = U$.
\end{deft}

By the foregoing, if $U$ satisfies condition $\mathscr{R}$ then $U^c$ is unitarily equivalent to $U$.

Every representation $U\in \M(\C_0(\QG)\otimes \mc K(\sH))$ satisfies the formal condition
$(S \tensor \id)(U) = U^*$, which means that for all $\om \in \mc B(\H)_*$,
\begin{equation}\label{eq:reps_S}
(\id \tensor \om)(U) \in D(S) \text{ and } S((\id \tensor \om)(U)) = (\id \tensor \om)(U^*).
\end{equation}
Due to the polar decomposition of the antipode and \prettyref{eq:J_tau_impl}, this is equivalent to (see the paragraph before Subsection \ref{Subsect:quantum groups} for the notation)
\begin{equation}\label{eq:reps_T_hat}
(\i\tensor\om)(U)\widehat{T}\subseteq\widehat{T}(\i\tensor\overline{\om})(U), \qquad \forall_{\om\in\mathcal{B}(\Hil)_{*}}.
\end{equation}
The map $\Lone{\G} \to \mc B(\H)$ given by $\om \mapsto (\om \tensor \id)(U)$ is an algebra homomorphism, and its restriction to $\LoneSharp{\G}$ is a $*$-homomorphism. Thus, the norm closure of $\{(\om \tensor \id)(U) : \om \in \Lone{\G} \}$ is a $\cst$-algebra.

We define the \emph{tensor product} of two representations in two ways: namely as the representations
\[ U \tp V = U_{12} V_{13} \qquad \text{and} \qquad U \tpr V :=V_{13} U_{12},\] both acting on $\sH_U \otimes \sH_V$.  (Note that the notation `$\tpr$' is not consistent with \cite{Woronowicz__CMP}.) There is also a natural notion of the direct sum of two representations $U$ and $V$, which is the obvious representation acting on $\sH_U \oplus \sH_V$.
A representation of $\QG$ is called \emph{irreducible} if it is not (unitarily equivalent to) a direct sum of two non-zero representations.

We will often use the fact (a consequence of the correspondence described in the previous subsection) that there is a natural bijection between representations of $\QG$ and representations of the \cst-algebra $\C_0^u(\hQG)$, implemented by the semi-universal multiplicative unitary. It interacts naturally with the passing to the contragredient and the tensor product operations, as the two following lemmas show.

\begin{lem}\label{lem:Contragredient-vs-antipode}
Let $\QG$ be a locally compact quantum group. Let $\mu\in \C_0^u(\hQG)^*$ be a state, let $\pi: \C_0^u(\hQG)\to \B(\K)$ be the GNS representation of $\mu$ and let $U$ be the representation of $\QG$ on $\K$ associated with $\pi$. If $\mu$ is $\widehat{R}^u$-invariant then $U$ satisfies condition $\mathscr{R}$.
\end{lem}
\begin{proof}
Consider the GNS construction $(\K,\pi,\gnsmap)$
of $(\CzU{\widehat{\G}},\mu)$. The map $\gnsmap(x)\mapsto\gnsmap(\widehat{R}^{u}(x^{*}))$,
$x\in\CzU{\widehat{\G}}$, extends to an involutive anti-unitary $\J$ on $\K$
by invariance of $\mu$ and because $\widehat{R}^{u}$ is an anti-automorphism
of $\CzU{\widehat{\G}}$ satisfying $(\widehat{R}^{u})^2 = \i$. Now, defining an anti-automorphism of $\mathcal{B}(\K)$
by $j(x):=\J x^{*}\J$, $x\in\mathcal{B}(\K)$,
we have $j\circ\pi=\pi\circ\widehat{R}^{u}$. As a result,
\[
(R\tensor j)(U)=(\i\tensor\pi)(R\tensor\widehat{R}^{u})(\wW)=(\i\tensor\pi)(\wW)=U. \qedhere
\]
\end{proof}

\begin{lem}\label{lem:tensprod}
Let $U, V$ be representations of a locally compact quantum group $\QG$ on Hilbert spaces $\Hil_U, \Hil_V$, and let $\phi_U:\C_0^u(\hQG) \to B(\Hil_U)$, $\phi_V:\C_0^u(\hQG) \to B(\Hil_V)$  be the corresponding representations of $\C^u(\hQG)$. Then $\phi_{U \tpsmall V}=\phi_U \star \phi_V$, where $\phi_U \star \phi_V= (\phi_U \ot \phi_V) \circ \sigma \circ \hh \Com_u$, and $\sigma: \M(\C_0^u(\hQG)\ot \C_0^u(\hQG))\to \M(\C_0^u(\hQG)\ot \C_0^u(\hQG))$ is the tensor flip.
\end{lem}
\begin{proof}
Since $(\i \tensor \hh \Com_u)(\wW) = \wW_{13} \wW_{12}$, the non-degenerate $*$-homomorphism $\phi_U \star \phi_V : \C_0^u(\hQG) \to \mc B(\Hil_U \tensor \Hil_V)$ satisfies
\begin{equation*}\begin{split}
(\i \tensor (\phi_U \star \phi_V))(\wW) & =
(\i \tensor (\phi_U \tensor \phi_V)\sigma)(\wW_{13} \wW_{12}) =
(\i \tensor (\phi_U \tensor \phi_V))(\wW_{12} \wW_{13}) \\ & =
(\i \tensor \phi_U)(\wW)_{12} (\i \tensor \phi_V)(\wW)_{13} = U_{12} V_{13} = U \tp V.
\end{split}
\end{equation*}
By the universal property of $\wW$, we have $\phi_{U \tpsmall V}=\phi_U \star \phi_V$.
\end{proof}

We will also sometimes use another `picture' of representations, as follows.
For $i=1,2$, let $\alg_i$ be a $\cst$-algebra and $\mathsf{X}_i$ a Hilbert $\alg_i$-module.
The exterior tensor product $\mathsf{X}_1 \tensor \mathsf{X}_2$ \cite[Chapter 4]{Lance}
is the Hilbert $\alg_1 \tensor \alg_2$-module obtained from the algebraic tensor product $\mathsf{X}_1 \odot \mathsf{X}_2$  using the $\alg_1 \odot \alg_2$-valued inner product
\[ \langle x_1 \tensor x_2, x_1'  \tensor x_2'  \rangle := \langle x_1, x_1' \rangle \tensor \langle x_2, x_2' \rangle\]
after completion. Now, we deduce from \cite[Theorem 2.4 and p.~37]{Lance} that
\begin{equation} \label{eq:adj_ext_tens_prod} \mc L(\mathsf{X}_1 \tensor \mathsf{X}_2) \cong \M (\mc K (\mathsf{X}_1 \tensor \mathsf{X}_2))
\cong \M (\mc K (\mathsf{X}_1) \tensor \mc K(\mathsf{X}_2))
\end{equation}
canonically as $\cst$-algebras.
In the particular case when $\mathsf{X}_1$ is a $\cst$-algebra $\alg$ and $\mathsf{X}_2$ is a Hilbert space $\H$, the Hilbert $\alg$-module $\alg \tensor \H$ satisfies $\mc L(\alg \tensor \H) \cong \M(\alg \tensor \mc K (\H))$ by \prettyref{eq:adj_ext_tens_prod} since $\mc K(\alg) \cong \alg$. This $*$-isomorphism takes $U \in \M(\alg \tensor \mc K (\H))$ to the unique element $\mc U \in \mc L(\alg \tensor \H)$ characterised by
\[ \langle b \tensor \eta, \mc U (a \tensor \z) \rangle = b^* (\i \tensor \om_{\eta,\z})(U) a \qquad \forall_{a,b\in\alg} \; \forall_{\z,\eta\in\H}.\]
When $\G$ is a locally compact quantum group and $U\in\M(\Cz{\G}\tensor\mathcal{K}(\H))$ is a representation of $\G$ on $\H$,
we henceforth use the correspondences $U \leftrightarrow \mc U \in\mc L(\C_0(\QG)\otimes\sH)\leftrightarrow \phi_U \in \Mor(\C_0^u(\hQG), \mc K (\Hil))$ without further comment.

\subsection{Compact/discrete quantum groups}\label{sub:compact_discrete}

 Recall that a locally compact quantum group is called \emph{compact} if the algebra $\C_0(\G)$ is unital (we then denote it simply by $\C(\G)$), or, equivalently, the Haar weight is in fact a bi-invariant state. It is said to be \emph{discrete} if $\C_0(\G)$ is a direct sum of matrix algebras (and is then denoted $\c0(\G)$, and likewise we write $\linfty{\G}$, $\lone{\G}, \ltwo{\G}$), or, equivalently, $\widehat{\QG}$ is compact. See \cite{Woronowicz__symetries_quantiques,Effros_Ruan__discrete_QGs,Van_Daele__discrete_CQs} for the original definitions and \cite{Runde__charac_compact_discr_QG} for the equivalent characterisations.

For a compact quantum group $\QG$  the symbol $\Irred{\QG}$ will denote the family of all equivalence classes of (necessarily finite-dimensional) irreducible unitary representations of $\QG$. The trivial representation will be denoted simply by $\one$. We will always assume that for each $\alpha\in \Irred{\QG}$ a particular representative has been chosen and moreover identified with a unitary matrix $U^{\alpha}= (u_{ij}^{\alpha})_{i,j=1}^{n_{\alpha}} \in M_{n_{\alpha}} (\C^u(\QG))$. So for all $\a \in \Irred{\QG}$ and $1\leq i,j \leq n_\a$,
\begin{equation}\label{eq:fin_dim_reps}
\Delta_u(u^\a_{ij}) = \sum_{k=1}^{n_\a} u^\a_{ik} \tensor u^\a_{kj}.
\end{equation}
The span of all \emph{coefficients} $u_{ij}^{\alpha}$ is a dense (Hopf) $*$-subalgebra of $\C^u(\QG)$, denoted $\Pol(\QG)$. The algebra of (vanishing at infinity) functions on the dual discrete quantum group is given by the equality $\c0(\widehat{\QG})= \bigoplus_{\alpha \in \Irred{\QG}} M_{n_\alpha}$. Thus the elements affiliated to $\c0(\widehat{\QG})$ can be identified with functionals on $\Pol(\QG)$. Note that as the Haar state of $\QG$ is faithful on $\Pol(\QG)$ we can also view the latter algebra as a subalgebra of $\C(\QG)$. The universal multiplicative unitary of $\G$ is then given by the formula
\begin{equation} \label{unitcompact} \Ww = \sum_{\alpha \in \Irred{\QG}}  u_{ij}^\alpha \ot e_{ij}^{\alpha} \in
\prod_{\alpha \in \Irred{\G}} \C^u(\G) \otimes M_{n_\alpha} = \M(\C^u(\G) \ot \c0(\widehat{\QG})). \end{equation}

The following definition will play an important role later on: a compact quantum group $\QG$ is \emph{low} if the set $\{n_{\alpha}:\alpha \in \Irred{\G}\}$ is bounded in $\mathbb N$. The dual of every discrete group satisfies this trivially.
Finally note that a compact quantum group $\QG$ is second countable if and only if $\Irred{\QG}$ is countable.

\begin{rem}\label{Remark:low}
A classical locally compact group is low if and only if it has an abelian subgroup of finite index, as shown by Moore (\cite{Moore_fin_dim_irreds}). We do not know if an analogous result holds for quantum groups: one would of course need to replace the `abelian group' in the statement above by a `cocommutative quantum group'. In Sections \ref{sec:TT11}--\ref{sec:Connes_Weiss} of our paper a key role will be played by `second countable discrete unimodular quantum groups with low duals'.
Non-trivial examples of such objects can be produced via the bi-crossed
product construction of \cite{Fima_Mukherjee_Patri_compact_bicross_prod}. Indeed, if we start with a \emph{matched pair} of a finite group $G$ and a countable discrete
group $\Gamma$, Theorem 3.4 of \cite{Fima_Mukherjee_Patri_compact_bicross_prod}  produces a quantum
group $\QG$ in the class above (and Theorem 4.3 of the same paper shows that $\QG$  has Property (T) -- to be discussed in \prettyref{sec:prop_T} --
if and only if $\Gamma$ has Property (T)).
Another construction of the quantum groups in the class above is as follows: take any discrete (quantum) group $\QG$ in the class (so for example the
dual of a second countable compact group which has an abelian subgroup of finite index) and an action of a
countable discrete group $\Gamma$ on the dual of $\QG$ by automorphisms. Then the
crossed product construction of Wang gives another quantum group in the same class (see Theorem
6.1 of \cite{Fima_Mukherjee_Patri_compact_bicross_prod}). We refer for the details to \cite{Fima_Mukherjee_Patri_compact_bicross_prod}.
\end{rem}

By a \emph{state} on $\Pol(\G)$ we mean a linear functional $\mu:\Pol(\G)\rightarrow\mathbb C$ which is positive in the sense that $\mu(a^*a)\geq0$ for all $a\in\Pol(\G)$.  The algebra $\C^u(\G)$ is the enveloping $\cst$-algebra of $\Pol(\G)$ and there is a bijection between states of $\C^u(\G)$ and states of $\Pol(\G)$ by \cite[Theorem 3.3]{Bedos_Murphy_Tuset__coam_of_CQGs}. Note that any functional $\mu$ on $\Pol(\QG)$ can be identified with a sequence of matrices $(\mu^\alpha)_{\alpha \in \Irred{\QG}}$, with $\mu^\alpha \in M_{n_\alpha}$ defined by $(\mu^\alpha)_{i,j}= \mu (u_{ij}^{\alpha})$, $i,j=1,\ldots, n_{\alpha}$.

A \emph{generating functional} on a compact quantum group $\QG$ is a functional $L:\Pol (\QG)
\rightarrow \bc$ which is selfadjoint ($L(a^*) = \overline{L(a)}$ for all $a \in \Pol (\QG)$), vanishes at $\one$ and is conditionally negative definite, i.e.\ negative on the kernel of the counit (formally: if $a\in \Pol(\QG)$ and $\Cou(a)=0$, then $L(a^*a) \leq 0$). Note a sign difference with, for example, \cite{Das_Franz_Kula_Skalski__one_to_one_corres} -- here we choose to work with conditionally \emph{negative definite functions}, as in \cite{Kyed__cohom_prop_T_QG,Daws_Fima_Skalski_White_Haagerup_LCQG}. A (weakly continuous) \emph{convolution semigroup of states} on $\C^u(\QG)$ is a family $(\mu_t)_{t\geq 0}$ of states of $\C^u(\QG)$ such that
\begin{enumerate}[label=(\roman*)]
\item $ \mu_{s+t} = \mu_s \conv \mu_t$  for all $s,t \geq 0$;
\item $\mu_0 = \Cou$;
\item $\mu_t \xrightarrow{t \to 0^+} \Cou$ in the weak$^*$ topology of $\C^u(\QG)^*$.
\end{enumerate}

We will need at some point the next lemma (following from  \cite{Lindsay_Skalski__quant_stoch_conv_cocyc_2} -- see also \cite{Schurmann__white_noise_bialg}).

\begin{lem} \label{lem:convgen}
Let $\QG$ be a compact quantum group.
There is a one-to-one correspondence between
\begin{enumerate}
\item convolution semigroups of states $(\mu_t)_{t\geq 0}$ of $\C^u(\QG)$;
\item generating functionals $L$ on $\QG$.
\end{enumerate}
It is given by the following formulas: for each $a\in \Pol(\QG) \subseteq \C^u(\QG)$ we have
\[ L(a) = \lim_{t\to 0^+} \frac{\Cou(a) - \mu_t(a)}{t},\]
\[ \mu_t (a) = \exp_{\conv}(-tL) (a):= \sum_{n=0}^{\infty} \frac{(-t)^n}{n!} L^{\conv n} (a).\]
\end{lem}

We say that a generating functional $L:\Pol(\QG) \to \bc$ is \emph{strongly unbounded} if for each $M>0$ there exists $\alpha \in \Irred{\QG}$ such that $L^{\alpha}$ is a positive matrix of norm larger than $M$.
Note that if $L$ is strongly unbounded then it is also unbounded with respect to the universal norm on $\C^u(\QG)$ (if $L$ were bounded we could identify $L=(L^{\alpha})_{\alpha \in \Irred{\QG}}\in \M(\c0(\QG))$ with  $l = (L \ot \id)(\Ww)$; but then $l$ would have to be bounded).

\subsection{Morphisms of quantum groups}

Let $\QG, \QH$ be locally compact quantum groups. As shown in \cite{Meyer_Roy_Woronowicz__hom_quant_grps} (see also \cite{Kustermans__LCQG_universal}), there is a natural correspondence between the following classes of objects, each of which should be thought of as representing a \emph{quantum group morphism from $\QH$ to $\QG$}:
\begin{enumerate}
\item morphisms
\begin{equation*}
\pi\in\Mor\bigl(\C_0^u(\QG),\C_0^u(\QH)\bigr)
\end{equation*}
such that
\begin{equation*}
(\pi\ot\pi)\circ\Delta_\QG^u=\Delta_\QH^u\circ\pi;
\end{equation*}
\item bicharacters (from $\QH$ to $\QG$), i.e.\ unitaries
\begin{equation*}
V\in\M\bigl(\C_0(\QH)\ot \C_0(\hQG)\bigr)
\end{equation*}
such that
\begin{equation*}\label{DelDel}
\begin{split}
(\id_{\C_0(\QH)} \ot \Delta_{\hQG})(V)&=V_{13}V_{12},\\
(\Delta_{\QH}\ot\id_{\C_0(\hQG)})(V)&=V_{13}V_{23}.
\end{split}
\end{equation*}
\end{enumerate}
The correspondence is given by
\begin{equation}\label{eq:morphs_bichars}
V = (\Lambda_{\QH} \pi \tensor \i)(\Ww_\G).
\end{equation}
The apparent difference with \cite{Meyer_Roy_Woronowicz__hom_quant_grps} stems from the fact that we are using the conventions of \cite{Kustermans__LCQG_universal} and \cite{Daws_Fima_Skalski_White_Haagerup_LCQG} rather than these of \cite{Meyer_Roy_Woronowicz__hom_quant_grps} or \cite{Daws_Kasprzak_Skalski_Soltan__closed_q_subgroups_LCQGs}.

To each morphism from $\QH$ to $\QG$ as above corresponds a \emph{dual morphism} from $\hQG$ to $\hQH$, described (for example) by the morphism $\hh \pi \in \Mor(\C_0^u(\hQH), \C_0^u(\hQG))$ determined uniquely by the equality
\begin{equation}\label{eq:dual_morphism}
(\pi \ot \id)(\WW_{\G}) = (\id \ot \hh\pi)(\WW_{\QH}).
\end{equation}

\begin{deft}\label{Def:dense_image}
Let $\QG$ and $\QH$ be locally compact quantum groups. A morphism from $\QH$ to $\QG$ is said to have \emph{dense image} if the map $\Lone{\hQG} \rightarrow \M(\C_0^u(\QH))$ induced by the restriction of the associated morphism
$\pi \in \Mor(\C_0^u(\G),\C_0^u(\QH))$, namely
$\omega\mapsto \pi( (\id\otimes\omega)\Ww^\G )$,
is injective.
\end{deft}

For a detailed discussion of this notion and several equivalent formulations we refer to the Appendix.

\section{Invariant and almost invariant vectors for representations} \label{sec:inv_alm_inv_vects}
In this section we recall and expand some facts concerning the notions of invariant and almost invariant vectors (as defined for example in \cite{Daws_Fima_Skalski_White_Haagerup_LCQG}) and connect these to the classical Hilbert space convexity arguments. In particular, we extend a classical result of Godement
about an invariant mean on the Fourier--Stieltjes algebra of a locally compact group.

Let $U\in\M(\Cz{\G}\tensor\mathcal{K}(\H))$ be a representation
of a locally compact quantum group $\G$ on a Hilbert space $\H$.

\begin{lem}[{\cite[Proposition 3.4]{Daws_Fima_Skalski_White_Haagerup_LCQG}}]\label{lem:inv_vec}
For $\z\in\H$, the following conditions are equivalent:
\begin{enumerate}
\item \label{enu:inv_vec_1} $U(\eta\tensor\z)=\eta\tensor\z$ for every $\eta\in\Ltwo{\G}$, when
considering $U$ as acting on $\Ltwo{\G}\tensor\H$;
\item \label{enu:inv_vec_2} $(\om\tensor\i)(U)\z=\om(\one)\z$ for every $\om\in\Lone{\G}$;
\item \label{enu:inv_vec_3} $\phi_{U}(\cdot)\z=\hat{\epsilon}(\cdot)\z$.
\end{enumerate}
\end{lem}

\begin{rem}\label{rem:inv_vec}
Conditions \prettyref{enu:inv_vec_1} and \prettyref{enu:inv_vec_2} are clearly equivalent for every $U \in \mc B(\Ltwo{\G} \tensor \H)$. When $U$ is unitary, they are thus equivalent to the same conditions with $U^*$ in place of $U$.
\end{rem}

\begin{defn}
A vector that satisfies the equivalent conditions of the previous
lemma is said to be \emph{invariant under $U$}. Denote by $\inv(U)$
the closed subspace of all vectors in $\H$ that are invariant under
$U$, and by $p^{U}$ the projection of $\H$ onto $\inv(U)$. If
$U$ has no non-zero invariant vectors, it is called \emph{ergodic}.
\end{defn}
Note that the restriction $U(\one\tensor(\one-p^{U}))$ of $U$ to
$\inv(U)^{\perp}$ is a representation of $\G$ on $\inv(U)^{\perp}$.
\begin{lem}
\label{lem:almost_inv_vec}For a net $\left(\z_{i}\right)_{i \in \Ind}$ of unit
vectors in $\H$, the following conditions are equivalent:
\begin{enumerate}[label=\tu{(I.\alph*)}]
\item \label{enu:almost_inv_vect_1}for every
$\eta\in\Ltwo{\G}$, $\left\Vert U(\eta\tensor\z_{i})-\eta\tensor\z_{i}\right\Vert \stackrel{i \in \Ind}{\longrightarrow} 0$;
\item \label{enu:almost_inv_vect_2}the net $(\i \tensor \om_{\z_i})(U)$ converges in the weak$^*$ topology of $\Linfty{\G}$ to $\one$;
\item \label{enu:almost_inv_vect_3}for every $a\in \Cz{\G}$, $\left\Vert \mathcal{U} (a \tensor \z_i) - a \tensor \z_i \right\Vert \stackrel{i \in \Ind}{\longrightarrow} 0$;
\item \label{enu:almost_inv_vect_4}for every $a\in\CzU{\widehat{\G}}$, $\left\Vert \phi_{U}(a)\z_{i}-\widehat{\epsilon}(a)\z_{i}\right\Vert \stackrel{i \in \Ind}{\longrightarrow}0$.
\end{enumerate}

Furthermore, the following conditions are equivalent:
\begin{enumerate}[label=\tu{(II.\alph*)}]
\item \label{enu:almost_inv_vects_1}there exists
a net satisfying the above equivalent conditions;
\item \label{enu:almost_inv_vects_2}the trivial representation of $\G$
is weakly contained in $U$, that is, there exists a state $\Psi$
on $\Img\phi_{U}$ such that $\Psi\circ\phi_{U}=\widehat{\epsilon}$;
\item \label{enu:almost_inv_vects_3}there exists a state $\Psi$ on $\mathcal{B}(\H)$
such that $(\i\tensor\Psi)(U)=\one$, or equivalently, $(\Psi|_{\Img\phi_{U}})\circ\phi_{U}=\widehat{\epsilon}$.
\end{enumerate}

A state $\Psi$ as in \prettyref{enu:almost_inv_vects_2} or \prettyref{enu:almost_inv_vects_3}
satisfies $xp^{U}=\Psi(x)p^{U}$ for all $x\in\Img\phi_{U}$.

\end{lem}
\begin{proof}
Apart from the last sentence, everything is taken from \cite[Proposition 3.7 and Corollary 3.8]{Daws_Fima_Skalski_White_Haagerup_LCQG}.

As for all $a\in\CzU{\widehat{\G}}$ and $\z\in\inv(U)$, we have $\phi_{U}(a)\z=\widehat{\epsilon}(a)\z$, if $\Psi$ is as in \prettyref{enu:almost_inv_vects_2} or \prettyref{enu:almost_inv_vects_3},
so that $\Psi\circ\phi_{U}=\widehat{\epsilon}$, we get $xp^{U}=\Psi(x)p^{U}$
for all $x\in\Img\phi_{U}$.\end{proof}
\begin{defn}
A net satisfying the equivalent conditions \prettyref{enu:almost_inv_vect_1}--\prettyref{enu:almost_inv_vect_4} of the previous lemma is said
to be \emph{almost invariant under $U$}. If such a net exists (namely,
\prettyref{enu:almost_inv_vects_1}--\prettyref{enu:almost_inv_vects_3}
are fulfilled), $U$ is said to \emph{admit almost-invariant vectors}.\end{defn}
The next corollary will not be required later, but should be mentioned
in this context. For analogous results in the classical context see \cite{Li_Ng__spect_gap_inv_states},
\cite[Proposition 1.5.5]{Peterson__erg_thy_lect_notes}, and for quantum groups see {\cite[Theorem 2.3]{Bedos_Murphy_Tuset__coam_of_CQGs}.

\begin{cor}
Let $U$ be a representation of a discrete quantum group
$\G$ on a Hilbert space $\H$ and $\phi_{U}$ be the associated representation
of $\CU{\widehat{\G}}$. The following conditions are equivalent:
\begin{enumerate}
\item \label{enu:one_weakly_subeq_1}$U$ has almost-invariant vectors;
\item \label{enu:one_weakly_subeq_2}there exists a state $\Psi$ of $\mathcal{B}(\H)$
such that
\begin{equation*}
\Psi(\phi_{U}(u_{ik}^{\gamma}))=\delta_{ik}\qquad\forall_{\gamma\in\Irred{\widehat{\G}}} \; \forall_{1\leq i,k\leq n_{\gamma}};
\end{equation*}

\item \label{enu:one_weakly_subeq_3}for every finite $F\subseteq\Irred{\widehat{\G}}$,
the operator $a_{F}:=\frac{1}{\left|F\right|}\sum_{\gamma\in F}\frac{1}{n_{\gamma}}\sum_{i=1}^{n_{\gamma}}\phi_{U}(u_{ii}^{\gamma})$
satisfies $1\in\sigma(\Ree a_{F})$.
\end{enumerate}
\end{cor}
\begin{proof}
\prettyref{enu:one_weakly_subeq_1}$\iff$\prettyref{enu:one_weakly_subeq_2}
is immediate from \prettyref{lem:almost_inv_vec} because the operators of the form
$\phi_{U}(u_{ik}^{\gamma})$ are linearly dense in $\phi_{U}(\CUcomp{\widehat{\G}})$.

\prettyref{enu:one_weakly_subeq_1}$\implies$\prettyref{enu:one_weakly_subeq_3}:
let $\left(\z_{\iota}\right)_{\iota \in \Ind}$ be a net of unit vectors that is almost-invariant
under $U$. Then for every finite $F\subseteq\Irred{\widehat{\G}}$, we have $a_{F}\z_{\iota}-\z_{\iota} \stackrel{\iota \in \Ind}{\longrightarrow}0$
and $(\Ree a_{F})\z_{\iota}-\z_{\iota}\stackrel{\iota \in \Ind}{\longrightarrow}0$. Hence
$1\in\sigma(\Ree a_{F}).$

\prettyref{enu:one_weakly_subeq_3}$\implies$\prettyref{enu:one_weakly_subeq_1}:
suppose that a finite $F\subseteq\Irred{\widehat{\G}}$ satisfies $1\in\sigma(\Ree a_{F})$.
Then there is a sequence $\left(\z_{n}\right)_{n \in \bn}$ of unit vectors in $\Hil$ such that $(\Ree a_{F})\z_{n}-\z_{n}\stackrel{n \to \infty}{\longrightarrow}0$.
Since each $\phi_{U}(u_{ii}^{\gamma})$ is contractive,
by uniform convexity we get $\phi_{U}(u_{ii}^{\gamma})\z_{n}-\z_{n}\stackrel{n \to \infty}{\longrightarrow}0$
for every $\gamma\in F, 1\leq i\leq n_{\gamma}$. Since
$\phi_{U}^{(n_{\gamma})}(u^{\gamma})$ is unitary, thus contractive,
we infer that $\phi_{U}(u_{ik}^{\gamma})\z_{n}-\delta_{ik}\z_{n}\stackrel{n \to \infty}{\longrightarrow}0$
for $\gamma\in F, 1\leq i,k\leq n_{\gamma}$. As $F$ was
arbitrary, $U$ has almost-invariant vectors.\end{proof}

\begin{defn} \label{def:wm}
Let $\G$ be a locally compact quantum group.
\begin{itemize}
\item A finite-dimensional representation $u\in\M(\Cz{\G}\tensor M_{n})$
of $\G$ is \emph{admissible} \cite[Definition 2.2]{Soltan__quantum_Bohr_comp}
if, when viewed as a matrix $\left(u_{ij}\right)_{1\leq i,j\leq n}$
of elements in $\M(\Cz{\G})$, its transpose $\left(u_{ji}\right)_{1\leq i,j\leq n}$
is invertible.
\item A unitary representation $U$ of $\G$ is \emph{weakly mixing} \cite[Definition 2.9]{Viselter__LCQGs_weak_mixing}
if it admits no non-zero admissible finite-dimensional sub-representation.
\end{itemize}
\end{defn}

Let $\G$ be a locally compact quantum group with \emph{trivial scaling group}. The characterisations
of weak mixing in \cite[Theorem 2.11]{Viselter__LCQGs_weak_mixing}
become much simpler under this assumption, as follows. For starters,
every finite-dimensional (unitary) representation of $\G$ is admissible.
In \cite{Viselter__LCQGs_weak_mixing}, a `(unitary) representation
of $\G$ on $\H$' (called there a `corepresentation') is a unitary $X\in\mathcal{B}(\H)\tensorn\Linfty{\G}$
such that $(\i\tensor\Delta)(X)=X_{12}X_{13}$. Such $X$ is weakly
mixing (namely, it has no non-zero finite-dimensional sub-representation)
if and only if $(\top\tensor R)(X)_{13}X_{23}$ is ergodic, if and only
if $Y_{13}X_{23}$ is ergodic for all representations $Y$
of $\G$ as above. Let $U\in\Linfty{\G}\tensorn\mathcal{B}(\H)$ be
unitary. Then $U$ is a representation of $\G$ in our sense if
and only if $X:=\sigma(U^{*})$ is a representation of the opposite
locally compact quantum group $\G^{\op}$ \cite[Section 4]{Kustermans_Vaes__LCQG_von_Neumann}
in the above sense. In this case, since $R^{\op}=R$, we have $(\top\tensor R^{\op})(\sigma(U^{*}))=\sigma(U^{c*})$.
Also, for any other representation $V\in\Linfty{\G}\tensorn\mathcal{B}(\K)$
of $\G$ on a Hilbert space $\K$, setting $Y:=\sigma(V^{*})$ we
get $\sigma_{12,3}(Y_{13}X_{23})^{*}=U_{13}V_{12}=V\tpr U$. Notice
also that a representation is ergodic, respectively weakly
mixing, if and only if its contragradient has this property. Putting
all this together, weak mixing of $U$ is equivalent to ergodicity
of any of $U^{c}\tp U$, $U\tp U^{c}$, $U^{c}\tpr U$ and $U\tpr U^{c}$,
and also to ergodicity -- and thus, \emph{a posteriori}, to weak mixing -- of
any of $V\tp U$, $U\tp V$, $V\tpr U$ and $U\tpr V$ for every unitary
representation $V$ of $\G$ (for instance, $V\tp U$ is weakly
mixing because $(V\tp U)^{c}\tp(V\tp U)=((V\tp U)^{c}\tp V)\tp U$
is ergodic). Similar results can be also found in \cite{Chen_Ng__prop_T_LCQGs}.

The notions of (almost-) invariant vectors discussed above are framed in terms of the algebra $\C_0(\G)$ and not $\C_0^u(\G)$.  Recall from \cite[Proposition~6.6]{Kustermans__LCQG_universal} that there is a bijection between representations $U\in \M(\C_0(\G)\otimes\mc K(\sH))$ and unitaries $U^u\in \M(\C_0^u(\G)\otimes\mc K(\sH))$ satisfying the `representation equation' $(\Delta_u\otimes\id)U^u = U^u_{13} U^u_{23}$. The bijection is given by the relation $(\Lambda_\G \otimes\id)(U^u)=U$; thus we have $U^u = (\id \tensor \phi_U)(\WW)$.

\begin{deft}
Let $V\in \M(\C_0^u(\G)\otimes\mc K(H))$ be a representation in the sense discussed above.
Then $\xi\in \sH$ is \emph{invariant} for $V$ if $(\mu\otimes\id)(V)\xi=
\mu(\one)\xi$ for all $\mu\in \C_0^u(\G)^*$.  A net of unit vectors
$(\xi_i)_{i \in \Ind}$ in $\sH$ is \emph{almost invariant} for $V$ if
$\| (\mu\otimes\id)(V)\xi_i - \xi_i \| \stackrel{i \in \Ind}{\longrightarrow} 0$ for
each state $\mu\in \C_0^u(\G)^*$.
\end{deft}

\begin{prop}\label{Prop:invUniv}
With notation as above, $\inv(U)=\inv(U^u)$ and a net $(\xi_i)_{i \in \Ind}$
of unit vectors in $\Hil$ is almost invariant for $U$ if and only if it is almost
invariant for $U^u$.
\end{prop}
\begin{proof}
Applying the reducing morphism shows that $\inv(U^u)\subseteq\inv(U)$.
Conversely, let $\xi$ be invariant for $U$.
By \eqref{reducedimplementation}
\[ U^u_{13} U_{23} = ((\id\otimes\Lambda_\G)\Delta_u\otimes\id)(U^u)
= \Ww^*_{12} U_{23} \Ww_{12}. \]
Fix a state $\omega\in \Lone{\QG}$ , and let $\mu\in \C_0^u(\G)^*$.  Then
\[ (\mu\otimes\omega\otimes\id)(U^u_{13} U_{23}) \xi
= (\mu\otimes\id)(U^u) (\omega\otimes\id)(U)\xi
= (\mu\otimes\id)(U^u)\xi, \]
as $\xi$ is invariant.  Let $\pi:\C_0^u(\G)\rightarrow\mc B(\sK)$ be a
representation such that there are $\alpha,\beta\in \sK$ with
$\mu = \omega_{\alpha,\beta}\circ\pi$. Let $\omega=\omega_\gamma$ for some
$\gamma\in L^2(\G)$.  Then, with $X=(\pi\otimes\id)(\Ww)$ and $\zeta \in \sH$,
\begin{align*} \la \zeta, (\mu\otimes\id)(U^u)\xi \ra
&= \la \zeta, (\mu\otimes\omega\otimes\id)(\Ww^*_{12} U_{23} \Ww_{12})\xi \ra
= \la \alpha\otimes\gamma\otimes\zeta, X^*_{12} U_{23} X_{12}
(\beta\otimes\gamma\otimes\xi) \ra \\
&= \la \alpha\otimes\gamma\otimes\zeta,  X^*_{12} X_{12} (\beta\otimes\gamma\otimes\xi) \ra = \la \zeta, \mu(\one) \xi\ra,
\end{align*}
using that $\xi$ is invariant, so $U(\eta\otimes\xi)=\eta\otimes\xi$
for all $\eta\in L^2(\G)$, and that $X$ is unitary.  Thus $\xi$ is
invariant for $U^u$ as claimed.

If $(\xi_i)_{i \in \Ind}$ is a net of almost invariant vectors for $U^u$, then let $\omega\in \Lone{\QG}$
be a state and set $\mu=\omega\circ\Lambda_{\G}$.  It follows that
$\|(\omega\otimes\id)(U)\xi_i-\xi_i\|=\|(\mu\otimes\id)(U^u)\xi_i-\xi_i\|\stackrel{i \in\Ind}{\longrightarrow} 0$.  By linearity
$\|(\omega\otimes\id)(U)\xi_i- \om(\one) \xi_i\|\stackrel{i \in\Ind}{\longrightarrow}0$ for all $\omega
\in \Lone{\QG}$ and so it follows that $(\id\otimes\omega_{\xi_i})(U)
\stackrel{i \in\Ind}{\longrightarrow}\one$ in the weak$^*$ topology of $L^\infty(\G)$, thus verifying condition \prettyref{enu:almost_inv_vect_2}
of \prettyref{lem:almost_inv_vec}.
So $(\xi_i)_{i \in \Ind}$ is almost invariant for $U$.

Conversely, we use the same argument as above, but with more care.
Let $\mu=\omega_\beta\circ\pi$ be a state of $\C_0^u(\G)$, and
$\omega=\omega_\gamma$ a state in $\Lone{\QG}$, with $\pi:\C_0^u(\G)\rightarrow\mc B(\sK)$ a
representation, $\beta \in \K$ and $\gamma \in L^2(\G)$.  If $(\xi_i)_{i \in \Ind}$ is
almost invariant for $U$,
\begin{align*} \lim_{i \in \Ind} \| (\mu\otimes\omega\otimes\id)(U^u_{13} U_{23})
\xi_i - \xi_i \|
&= \lim_{i \in \Ind} \| (\mu\otimes\id)(U^u) (\omega\otimes\id)(U)
\xi_i - \xi_i \| \\
&= \lim_{i \in \Ind} \| (\mu\otimes\id)(U^u) \xi_i - \xi_i \|.
\end{align*}
Let $(e_j)_{j\in \mathcal{J}}$ be an orthonormal basis for $\sK$, let $X$ be as above and let
$X(\beta\otimes\gamma) = \sum_{j\in \mathcal{J}} e_j\otimes\gamma_j$, so $\sum_{j\in \mathcal{J}} \|\gamma_j\|^2
= 1$.  Then, for $\eta\in \sH$,
\begin{align*}
& \big| \la \eta,(\mu\otimes\omega\otimes\id)(\Ww^*_{12} U_{23} \Ww_{12})\xi_i
\ra - \la \eta, \xi_i \ra \big| \\
&= \big| \la  \beta\otimes\gamma\otimes\eta,  X^*_{12} U_{23} X_{12} (\beta\otimes\gamma\otimes\xi_i)
\ra -  \la \eta, \xi_i \ra \big| \\
&= \Big| \sum_{j\in \mathcal{J}} \la \gamma_j\otimes\eta,U(\gamma_j\otimes\xi_i)\ra
- \la \eta, \xi_i \ra  \Big|.
\end{align*}
For $\epsilon>0$ there is a finite set $F\subseteq \mathcal{J}$ with $\sum_{j\not\in F}
\|\gamma_j\|^2 < \epsilon$, and then there is $i_0\in \Ind$ so that if
$i \geq i_0$ and $j \in F$ then $\| U(\gamma_j\otimes\xi_i) -
\gamma_j\otimes\xi_i \| < \epsilon / \sqrt{|F|}$.  Thus for such $i$,
\begin{align*}
\Big| \sum_{j\in \mathcal{J}} \la \gamma_j\otimes\eta,  U(\gamma_j\otimes\xi_i)\ra
& - \la \eta, \xi_i \ra \Big| \leq
\Big| \sum_{j\in F} \la \gamma_j\otimes\eta,U(\gamma_j\otimes\xi_i)\ra
- \sum_{j\in F} \la \gamma_j\otimes\eta,\gamma_j\otimes\xi_i\ra \Big|\\
&\qquad + \Big| \sum_{j\in F} \la \gamma_j\otimes\eta, \gamma_j\otimes\xi_i \ra
-  \la \eta, \xi_i \ra \Big|
+ \Big| \sum_{j\not\in F} \la \gamma_j\otimes\eta, U(\gamma_j\otimes\xi_i) \ra
\Big| \\
&\leq \sum_{j\in F} \epsilon |F|^{-1/2} \|\gamma_j\otimes\eta\|
+ |\la \eta, \xi_i \ra| \Big| 1 - \sum_{j\in F} \|\gamma_j\|^2 \Big|
+ \sum_{j\not\in F} \|\gamma_j\|^2 \|\eta\| \\
&\leq \|\eta\|\Big( \epsilon + \epsilon + \epsilon \Big)
\leq 3\epsilon \|\eta\|.
\end{align*}
So $\lim_{i\in \Ind} \| (\mu\otimes\omega\otimes\id)(\Ww^*_{12} U_{23} \Ww_{12})
\xi_i - \xi_i \| = 0$ and hence $\lim_{i\in \Ind} \|(\mu\otimes\id)(U^u)
\xi_i - \xi_i \| =0$, as required.
\end{proof}

We will later need to see invariant vectors of a given representation $U$ as arising from a version of an averaging procedure. Similar arguments can be found in \cite[Section 4]{Das_Daws__quantum_Eberlein}.

\begin{deft}
Let $U$ be representation of a locally compact quantum group $\QG$ on a Hilbert space $\Hil$. A subset $\mc S\subseteq\mc B(\sH)$ is an \emph{averaging semigroup for $U$}
if $\mc S$ is a semigroup of contractions such that
\begin{enumerate}
\item \label{enu:avrg_sg_1} each $T\in\mc S$
leaves $\inv(U)$ pointwise invariant;
\item \label{enu:avrg_sg_2} each $T\in\mc S$ leaves $\inv(U)^\perp$ invariant;
\item \label{enu:avrg_sg_3} if $\xi\in \sH$ is such that
$T(\xi)=\xi$ for all $T\in\mc S$, then $\xi\in\inv(U)$.
\end{enumerate}
Given such a semigroup, the \emph{$\mc S$-average} of $\xi\in \sH$ is the
unique vector of minimal norm in the closed convex hull of
$\{ T(\xi) : T\in\mc S \}$.
\end{deft}

We will now give two examples of averaging semigroups.

\begin{lem}\label{lem:avg_semigp_one}
For any representation $U$ of $\G$ the family $\mc S=\{ (\omega\otimes\id)(U)
: \omega\in \Lone{\QG} \text{ is a state}\}$ is an averaging semigroup for $U$.
\end{lem}
\begin{proof}
The fact that $U$ is a representation shows that $\mc S$ is a semigroup of
contractions.  We now check the conditions: \prettyref{enu:avrg_sg_1} and \prettyref{enu:avrg_sg_3} hold more or less trivially.
Now let $\beta\in\inv(U)^\perp$ and $\omega\in \Lone{\QG}$. For $\alpha\in\inv(U)$, we find that
\[ \la  \alpha,   (\omega\otimes\id)(U)\beta \ra
	= \la (\overline\omega\otimes\id)(U^*) \alpha, \beta \ra
	= \la \overline\omega(\one) \alpha, \beta \ra
	= \la  \alpha, \beta \ra \overline{\overline\omega(\one)} = 0\]
using that $\alpha$ is invariant, so that $(\overline\omega\otimes\id)(U^*) \alpha
= \overline\omega(\one) \alpha$ by \prettyref{rem:inv_vec}.  It follows that
$(\omega\otimes\id)(U) \beta \in \inv(U)^\perp$, as required to show \prettyref{enu:avrg_sg_2}.
\end{proof}

\begin{lem}\label{lem:avg_semigp_two}
For any representation $U$ of $\G$ the family $\mc S=\{ (\omega\otimes\id)(U^*)
: \omega\in \Lone{\QG} \text{ is a state}\}$ is an averaging semigroup for $U$.
\end{lem}
\begin{proof}
As $\omega\mapsto (\omega\otimes\id)(U^*)$ is an anti-homomorphism, $\mc S$
is a semigroup of contractions.
Conditions \prettyref{enu:avrg_sg_1} and \prettyref{enu:avrg_sg_3} hold by \prettyref{rem:inv_vec}, and the argument used in the previous proof is easily adapted to show \prettyref{enu:avrg_sg_2}.
\end{proof}

\begin{prop}\label{prop:ave_equals_proj_gen}
Let $U$ be a representation of $\G$ on a Hilbert space $\sH$, and let $\mc S$ be an averaging semigroup
for $U$.  Then:
\begin{enumerate}
\item\label{prop:ave_equals_proj_gen:one}
$U$ is ergodic if and only if, for all $\xi\in\sH$,
the closed convex hull of $\{ T(\xi) : T\in\mc S \}$ contains $0$;
\item\label{prop:ave_equals_proj_gen:two}
for $\xi\in \sH$, the $\mc S$-average of $\xi$, the unique vector of minimal norm
in the closed convex hull of $\{ T(\xi) : T\in\mc S \}$, is equal to
the orthogonal projection of $\xi$ onto $\inv(U)$.
\end{enumerate}
\end{prop}
\begin{proof}
For \ref{prop:ave_equals_proj_gen:two}, let $\xi\in \sH$ and let $C$ be the
closed convex hull of $\{ T(\xi) : T\in\mc S \}$.  As $\mc S$ is a semigroup,
it follows that $T(C) \subseteq C$ for each $T\in\mc S$.  If $\eta$ denotes the
$\mc S$-average of $\xi$, then as each $T\in\mc S$ is a contraction,
$\|T(\eta)\| \leq \|\eta\|$ and so by uniqueness, $T(\eta)=\eta$.  By the
assumptions on $\mc S$, it follows that $\eta\in\inv(U)$.  Now let
$\xi=\xi_0+\xi_1\in \inv(U) \oplus \inv(U)^\perp$ be the orthogonal
decomposition of $\xi$.  For $T\in\mc S$, we have $T(\xi_0) = \xi_0$,
and so
\[ \{ T(\xi) : T\in\mc S \} = \{ \xi_0 + T(\xi_1) : T\in\mc S \}. \]
If we denote by $C_0$ the closed convex hull of $\{ T(\xi_1) : T\in\mc S \}$,
then $C = \xi_0 + C_0$ and so $\eta=\xi_0 + \eta_0$ for some $\eta_0\in C_0$.
As $\inv(U)^\perp$ is invariant for $\mc S$, it follows that
$C_0\subseteq\inv(U)^\perp$, and so $\eta_0\in\inv(U)^\perp$.  We hence conclude
that $\eta-\xi_0 = \eta_0 \in \inv(U) \cap \inv(U)^\perp=\{0\}$, and so
$\eta=\xi_0$ as required.
(We remark that actually, by Pythagoras's Theorem, $\eta_0$ is the unique vector of minimal norm in $C_0$.  It follows that $0\in C_0$.)

\ref{prop:ave_equals_proj_gen:one} now follows immediately, as $U$ is
ergodic if and only if $\inv(U) = \{0\}$.
\end{proof}

From the last result we obtain the following quantum analogue of the ergodic theorem in \cite[Section 146]{Riesz_Nagy__FA_book}.

\begin{prop}
\label{prop:conv_to_Inv_U_proj}For a representation $U$ of $\G$
and an averaging semigroup $\mathcal{S}$ for $U$, the orthogonal projection onto $\inv(U)$
belongs to the strong closure of the convex hull of $\mathcal{S}$.
\end{prop}
\begin{proof}
The convex hull of $\mathcal{S}$ is also an averaging semigroup for
$U$. Thus, we can and do assume the $\mathcal{S}$ is convex. Endow
$\mathcal{S}$ with a preorder `$\preceq$' by saying that $T_{1}\preceq T_{2}$
if $T_{1}=T_{2}$ or there exists $T\in\mathcal{S}$ such that $T_{2}=TT_{1}$.

Let $\z\in\inv(U)^{\perp}$ and $\e>0$. By \prettyref{prop:ave_equals_proj_gen} there
exists $T_{0}\in\mathcal{S}$ such that $\left\Vert T_{0}\z\right\Vert <\e$.
Since $\mathcal{S}$ consists of contractions, it follows that $\left\Vert T\z\right\Vert <\e$
for all $T_{0}\preceq T\in\mathcal{S}$.

Fix $\z_{1},\ldots,\z_{n}\in\inv(U)^{\perp}$ and $\e>0$. By the
foregoing, there is $T_{1}\in\mathcal{S}$ such that $\left\Vert T\z_{1}\right\Vert <\e$
for all $T_{1}\preceq T\in\mathcal{S}$. Since $T_{1}\z_{2}\in\inv(U)^{\perp}$
there is, again by the foregoing, $T_{2}'\in\mathcal{S}$ such that
$\left\Vert TT_{1}\z_{2}\right\Vert <\e$ for all $T_{2}'\preceq T\in\mathcal{S}$.
Setting $T_{2}:=T_{2}'T_{1}$, we have $\left\Vert T\z_{j}\right\Vert <\e$
for all $j\in\left\{ 1,2\right\} $ and $T_{2}\preceq T\in\mathcal{S}$.
Proceeding by induction, there is $T_{n}\in\mathcal{S}$ such that
$\left\Vert T\z_{j}\right\Vert <\e$ for all $j\in\left\{ 1,\ldots,n\right\} $
and $T_{n}\preceq T\in\mathcal{S}$, and in particular for $T=T_{n}$.
This proves the assertion because $\mathcal{S}$ leaves $\inv(U)$
pointwise invariant.
\end{proof}

As another consequence of \prettyref{prop:ave_equals_proj_gen}, we generalise a classical result of Godement \cite[Section 23]{Godement__positive_def_func}
about the existence of an invariant mean on the Fourier--Stieltjes
algebra of an arbitrary (not necessarily amenable) locally compact
group. This should be compared with Section 4, and in particular Theorem
4.4, of \cite{Das_Daws__quantum_Eberlein}.

Recall that the module
actions of the Banach algebra $\Lone{\G}$ on its dual $\Linfty{\G}$
are given by
\[
\om\cdot a=(\i\tensor\om)(\Delta(a)),\quad a\cdot\om=(\om\tensor\i)(\Delta(a))\qquad(a\in\Linfty{\G},\om\in\Lone{\G}).
\]
Write $\mathcal{E}$ for the set of all $a\in\Linfty{\G}$ such that
each of the sets
\[
\overline{\left\{ \om\cdot a:\om\in\Lone{\G}\text{ is a state}\right\} }^{\left\Vert \cdot\right\Vert },\quad\overline{\left\{ a\cdot\om:\om\in\Lone{\G}\text{ is a state}\right\} }^{\left\Vert \cdot\right\Vert }
\]
intersects $\CC\one$.

For $a\in\mathcal{E}$, the intersections of the above sets with $\CC\one$
are equal, and are a singleton. Indeed, if $\left(\om_{n}^{1}\right)_{n=1}^{\infty}$,
$\left(\om_{n}^{2}\right)_{n=1}^{\infty}$ are states in $\Lone{\G}$
and $\om_{n}^{1}\cdot a\xrightarrow{n\to\infty}\lambda_{1}\one$,
$a\cdot\om_{n}^{2}\xrightarrow{n\to\infty}\lambda_{2}\one$ for some
$\lambda_{1},\lambda_{2}\in\CC$, then
\[
\lambda_{1}\one\xleftarrow{n\to\infty}(\om_{n}^{1}\cdot a)\cdot\om_{n}^{2}=\om_{n}^{1}\cdot(a\cdot\om_{n}^{2})\xrightarrow{n\to\infty}\lambda_{2}\one,
\]
so that $\lambda_{1}=\lambda_{2}$. Let us denote this common scalar
by $M(a)$. It follows readily that $\mathcal{E}$ is norm closed
and selfadjoint, and that for every $a\in\mathcal{E}$ and $\lambda,\mu\in\CC$,
we have $\lambda a+\mu\one\in\mathcal{E}$, $M(\lambda a+\mu\one)=\lambda M(a)+\mu$,
$M(a^{*})=\overline{M(a)}$, $\left|M(a)\right|\leq\left\Vert a\right\Vert $,
and if $\rho\in\Lone{\G}$ is such that $\rho\cdot a\in\mathcal{E}$
(respectively, $a\cdot\rho\in\mathcal{E}$), then $M(\rho\cdot a)=\rho(\one)M(a)$
(respectively, $M(a\cdot\rho)=\rho(\one)M(a)$).

Denote by $B(\G)$ the Fourier--Stieltjes algebra of $\G$,
namely the subalgebra of $\M(\Cz{\G})$ consisting of the coefficients
of all representations of $\G$:
\[
\begin{split}B(\G) & :=\left\{ (\i\tensor\om)(U):U\text{\text{ is a representation of }\ensuremath{\G}\text{ on a Hilbert space }\ensuremath{\H}}\text{ and }\om\in\mathcal{B}(\H)_{*}\right\} \\
 & =\left\{ (\i\tensor\om_{\z,\eta})(U):U\text{\text{ is a representation of }\ensuremath{\G}\text{ on a Hilbert space }\ensuremath{\H}}\text{ and }\z,\eta\in\H\right\} .
\end{split}
\]
Write $B(\G)^{*}$ for $\left\{ a^{*}:a\in B(\G)\right\} $. We were informed that a result essentially  equivalent to the following proposition had been recently obtained by B.~Das.
\begin{prop}
\label{prop:B_G__E}The subalgebras $\overline{B(\G)}^{\left\Vert \cdot\right\Vert }$
and $\overline{B(\G)^{*}}^{\left\Vert \cdot\right\Vert }$ of $\M(\Cz{\G})$
are contained in $\mathcal{E}$, and $M$ is linear on each of them.
Thus, $M$ restricts to an invariant mean on both $\overline{B(\G)}^{\left\Vert \cdot\right\Vert }$
and $\overline{B(\G)^{*}}^{\left\Vert \cdot\right\Vert }$.
\end{prop}

\begin{proof}
Let $U$ be a representation of $\G$ on a Hilbert space $\H$. Fix
$\z,\eta\in\H$, and consider $a:=(\i\tensor\om_{\eta,\z})(U)$. For
every $\rho\in\Lone{\G}$,
\[
\begin{split}\rho\cdot a & =(\i\tensor\rho)(\Delta(a))=(\i\tensor\rho)(\i\tensor\i\tensor\om_{\eta,\z})(U_{13}U_{23})\\
 & =(\i\tensor\om_{\eta,(\rho\tensor\i)(U)\z})(U)
\end{split}
\]
and
\[
a\cdot\rho=(\rho\tensor\i)(\i\tensor\i\tensor\om_{\eta,\z})(U_{13}U_{23})=(\i\tensor\om_{(\overline{\rho}\tensor\i)(U^{*})\eta,\z})(U).
\]
Denote by $p$ the orthogonal projection onto $\inv(U)$. By Lemmas
\ref{lem:avg_semigp_one}, \ref{lem:avg_semigp_two} and \prettyref{prop:ave_equals_proj_gen}, there are sequences $\left(\rho_{n}^{1}\right)_{n=1}^{\infty}$,
$\left(\rho_{n}^{2}\right)_{n=1}^{\infty}$ of states in $\Lone{\G}$
such that $(\rho_{n}^{1}\tensor\i)(U)\z\xrightarrow{n\to\infty}p\z$
and $(\rho_{n}^{2}\tensor\i)(U^{*})\eta\xrightarrow{n\to\infty}p\eta$.
Consequently,
\[
\rho_{n}^{1}\cdot a\xrightarrow{n\to\infty}(\i\tensor\om_{\eta,p\z})(U)=\left\langle \eta,p\z\right\rangle \one
\]
and
\[
a\cdot\rho_{n}^{2}\xrightarrow{n\to\infty}(\i\tensor\om_{p\eta,\z})(U)=\left\langle p\eta,\z\right\rangle \one
\]
in norm. This implies that $a\in\mathcal{E}$. Hence, $B(\G)\subseteq\mathcal{E}$.
Furthermore, the above calculation shows that $\rho\cdot a,a\cdot\rho\in B(\G)\subseteq\mathcal{E}$,
and thus $M(\rho\cdot a)=M(a)=M(a\cdot\rho)$, for every $a\in B(\G)$
and every state $\rho\in\Lone{\G}$.

Let $a_{1},a_{2}\in B(\G)$. Given $\e>0$, find states $\rho^{1},\rho^{2}\in\Lone{\G}$
such that $\left\Vert \rho^{1}\cdot a_{1}-M(a_{1})\one\right\Vert \leq\e$
and $\left\Vert a_{2}\cdot\rho^{2}-M(a_{2})\one\right\Vert \leq\e$.
Then
\[
M(a_{1}+a_{2})=M(\rho^{1}\cdot(a_{1}+a_{2})\cdot\rho^{2}),
\]
so that
\[
M(a_{1}+a_{2})-(M(a_{1})+M(a_{2}))=M\left((\rho^{1}\cdot a_{1}-M(a_{1})\one)\cdot\rho^{2}+\rho^{1}\cdot(a_{2}\cdot\rho^{2}-M(a_{2})\one)\right).
\]
But the norm of the element to which $M$ is applied on the right-hand
side does not exceed $2\e$. As a result, $\left|M(a_{1}+a_{2})-(M(a_{1})+M(a_{2}))\right|\leq2\e$.
Therefore, $M$ is additive, hence linear, on $B(\G)$. This entails
that $M$ is norm-continuous, and so linear, on $\overline{B(\G)}^{\left\Vert \cdot\right\Vert }$,
thus also on $\overline{B(\G)^{*}}^{\left\Vert \cdot\right\Vert }$.
\end{proof}

The set of all representations of a second countable locally compact quantum group $\QG$ on a fixed infinite-dimensional separable Hilbert space $\Hil$, denoted $\RepGH$, has a natural Polish topology.
This is the topology it inherits as a closed subset from the unitary group of $\M(\Cz{\G} \tensor \mc K(\H))$ equipped with the strict topology, which itself is a Polish space since $\Cz{\G} \tensor \mc K(\H)$ is a separable $\cst$-algebra. For more details, see \cite[Section 5]{Daws_Fima_Skalski_White_Haagerup_LCQG}.

Once we fix a unitary $\mathsf{u}:\sH \rightarrow
\sH\otimes\sH$, we can equip $\RepGH$ with the product
\[ U \boxtimes V = (\one\otimes \mathsf{u}^*)(U\tp V)(\one\otimes \mathsf{u}). \]

The lemmas below will be needed in Section \ref{sec:TT11}.

\begin{lem}\label{lem:tp_cts}
The maps $\RepGH \times \RepGH \to \RepGH$, $(U,V)\mapsto U \boxtimes V$ and
$\RepGH \to \RepGH$, $U \mapsto U^c$ are continuous (here $U^c$ is formed using some fixed anti-unitary from $\H$ onto itself).
Hence, so is the map $\RepGH \to \RepGH$, $U\mapsto U \boxtimes U^c$.
\end{lem}
\begin{proof}
The second map is continuous by the definitions of $U^c$ and of the strict topology.

Notice that if $U_n\stackrel{n \in \bn}{\longrightarrow} U$ in $\RepGH$, that is, in the strict topology on $\M(\C_0(\G)\otimes\mc K(\sH))$, then $(U_n)_{12} - U_{12} \to 0$ in the strict topology on $\M(\Cz{\G} \tensor \mc K(\H \tensor \H))$ as $(U_n)_{n \in \bn}$ is a sequence of unitaries, so that in particular it is bounded.
Let $U_n\stackrel{n \in \bn}{\longrightarrow} U$ and $V_n\stackrel{n \in \bn}{\longrightarrow} V$ in $\RepGH$. To prove that $U_n\boxtimes V_n\stackrel{n \in \bn}{\longrightarrow}
U\boxtimes V$, we should show that
\[ \big\| (U_n\tp V_n)A - (U\tp V)A \big\| \stackrel{n \in \bn}{\longrightarrow} 0 \]
for all $A \in \Cz{\G} \tensor \mc K(\H \tensor \H)$. However, this follows from the inequality
\begin{align*}
\big\| (U_n\tp V_n)A & -
(U\tp V)A \big\| \\
&\leq \big\| (U_n)_{12} (V_n)_{13}A -
(U_n)_{12} V_{13}A \big\| + \big\| (U_n)_{12} V_{13}A - U_{12} V_{13}A \big\| \\
&= \big\| \big( (V_n)_{13} - V_{13}\big)
   A \big\|
+ \big\| \big((U_n)_{12}-U_{12}\big)
   V_{13}A \big\|,
\end{align*}
where both summands tend to zero as explained above.
\end{proof}

\begin{lem}\label{lem:ergodic_gdelta}
The collection of all ergodic representations is a $G_\delta$ subset of $\RepGH$. If the scaling group of $\QG$ is trivial, so is the collection of all weakly mixing representations.
\end{lem}
\begin{proof}
By Lemma \ref{lem:avg_semigp_one} and Proposition \ref{prop:ave_equals_proj_gen} if $\mathfrak{S}$ denotes the collection of states
in $\Lone{\QG}$, then $U$ is ergodic if and only if for every non-zero
$\xi\in \sH$ the closure of $\{ (\omega\otimes\id)(U)\xi : \omega\in \mathfrak{S} \}$
contains $0$.  Let $(\xi_n)_{n \in \bn}$ be a dense sequence in the unit ball of $\sH$.
As each operator $(\omega\otimes\id)(U)$ is a contraction, we see that $U$
is ergodic if and only if
\[ U \in \bigcap_{n,m\in \bn} \bigcup_{\omega\in \mathfrak{S}} \big\{ V\in\RepGH :
\| (\omega\otimes\id)(V)\xi_n \| < 1/m \big \}. \]
So the proof of the assertion on ergodic representations will be complete if we show that for each $n,m \in \bn$ and $\om \in \mathfrak{S}$ the set $\{ V\in\RepGH :
\| (\omega\otimes\id)(V)\xi_n \| < 1/m \}$ is open, equivalently, if
$\{ V\in\RepGH : \| (\omega\otimes\id)(V)\xi_n \| \geq 1/m \}$ is closed.
However, if $(V_n)_{n \in \bn}$ is a sequence in this set converging strictly to $V$, then,
as slice maps are strictly continuous, $(\omega\otimes\id)(V_n)\rightarrow
(\omega\otimes\id)(V)$ strictly in $\mc B(\sH)=\M(\mc K(\sH))$, and hence
converges strongly, showing the result.
Combining the foregoing with the discussion after Definition \ref{def:wm} and with \prettyref{lem:tp_cts} implies the assertion on weakly mixing representations.
\end{proof}

\begin{lem} \label{lem:dense_orbit}
There exists $U \in \RepGH$ whose equivalence class in $\RepGH$, namely $\{(\one \tensor u^*)U(\one \tensor u) : u\in\B(\Hil)\text{ is unitary}\}$, is dense in $\RepGH$.
\end{lem}
\begin{proof}
Choose a sequence $(U_n)_{n=1}^\infty$ that is dense in $\RepGH$. One can view $\bigoplus_{n=1}^{\infty} U_n$ as an element $U$ in $\RepGH$ by fixing a unitary $v : \Hil \to \Hil \tensor \ltwo{\N}$ and letting $U := (\one \tensor v^*) \bigoplus_{n=1}^{\infty} U_n (\one \tensor v)$, where the direct sum is calculated according to an orthonormal basis $(\eta_i)_{i=1}^\infty$ of $\ltwo{\N}$. Let $V \in \RepGH$, $\e > 0$, $a_1 ,\ldots, a_m \in \Cz{\G}$ and $k_1 ,\ldots, k_m \in \mc K(\Hil)$ be such that the operators $k_i$ are of finite rank. By definition, there is $n \in \N$ such that $\|(U_n - V)(\sum_{i=1}^m a_i \tensor k_i)\| < \e / 3$.
Since $U_{n}(\sum_{i=1}^ma_{i}\tensor k_{i})\in\Cz{\G}\tensor\mathcal{K}(\Hil)$,
we can find finite collections $\left(b_{j}\right)_{j}$ in $\Cz{\G}$
and $\left(k_{j}'\right)_{j}$ of finite rank operators in $\mathcal{K}(\Hil)$
such that $\left\Vert U_{n}(\sum_{i}a_{i}\tensor k_{i})-\sum_{j}b_{j}\tensor k_{j}'\right\Vert <\e/3$.
Write $\Hil_{1}$ for the (finite-dimensional) linear span of the
ranges of all operators $k_{j}$ and $k_{j}'$. Let $u$ be a unitary
on $\Hil$ such that $u\z=v^{*}(\z\tensor\eta_{n})$ for all $\z\in\Hil_{1}$.
Then
\[
\Bigl\Vert\left[(\one\tensor u^{*})U(\one\tensor u)-U_{n}\right](\sum_{i=1}^m a_{i}\tensor k_{i})\Bigr\Vert<2\e/3,
\]
and so
\[
\Bigl\Vert\left[(\one\tensor u^{*})U(\one\tensor u)-V\right](\sum_{i=1}^m a_{i}\tensor k_{i})\Bigr\Vert<\e.
\]
The same technique works also for finite families of the form $\sum_{i=1}^m a_{i}\tensor k_{i}$,
applied to the left and to the right of $U,V$. We deduce that the
equivalence class of $U$ is dense in $\RepGH$.
\end{proof}

\section{Actions of locally compact quantum groups on von Neumann algebras}\label{sec:actions}

In this section we discuss actions of locally compact quantum groups and their properties. After quoting the basic definitions we discuss the canonical unitary implementation of such actions due to Vaes, various notions of ergodicity, and (almost) invariant states. We obtain in particular a new characterisation of the canonical unitary implementation of discrete quantum group actions (\prettyref{prop:act_pos_cone_applications},
\prettyref{enu:act_pos_cone_applications_3} and \prettyref{cor:U_preserves_pos_cone}). These technical results will be of crucial use in Sections~\ref{sec:spectral_gaps}
and~\ref{sec:Connes_Weiss}.

Let $\G$ be a locally compact quantum group and let $N$ be a von Neumann algebra. By an \emph{action of $\QG$ on $N$} we understand an injective normal unital $*$-homomorphism $\a:N\to\Linfty{\G}\tensorn N$ satisfying the \emph{action equation}
\[ (\Delta \ot \id_N) \circ \a = (\id_{\Linfty{\QG}} \ot \a)\circ \a.\]
The \emph{crossed product of $N$ by the action $\alpha$} is then the von Neumann subalgebra of $\mathcal{B}(\Ltwo{\QG}) \tensorn N$ generated by $\Linfty{\hQG} \ot \mathds{1}$ and $\alpha(N)$.

We say that $\alpha$ is \emph{implemented by} a unitary $V\in \mathcal{B}(\Ltwo{\QG} \ot \Kil)$ if $N \subseteq \mathcal{B}(\Kil)$ and
\[ \alpha(x) = V^* (\one \ot x)V, \;\; x \in N.\]
In \cite{Vaes__unit_impl_LCQG} Vaes shows that every action of a locally compact quantum group is implemented in a canonical way.
We use the notation of that paper,
but our version of a representation (so also what we call the unitary
implementation of an action) is the adjoint of the one in \cite{Vaes__unit_impl_LCQG}.
Let us then fix an n.s.f.~weight $\theta$ on $N$. Define
\[
\mathcal{D}_{0}:=\linspan\left\{ (\widehat{a}\tensor\one)\a(x):\widehat{a}\in\mathcal{N}_{\widehat{\varphi}},x\in\mathcal{N}_{\theta}\right\} ,
\]
and consider the linear map $\tilde{\gnsmap}_{0}:\mathcal{D}_{0}\to\Ltwo{\G}\tensor\Ltwo{N,\theta}$
given by
\[
\tilde{\gnsmap}_{0}((\widehat{a}\tensor\one)\a(x)):=\widehat{\gnsmap}(\widehat{a})\tensor\gnsmap_{\theta}(x)\qquad(\widehat{a}\in\mathcal{N}_{\widehat{\varphi}},x\in\mathcal{N}_{\theta}).
\]
Then $\tilde{\gnsmap}_{0}$ is indeed well-defined, and is $*$-ultrastrong--norm
closable. Denote its closure by $\tilde{\gnsmap}:\mathcal{D}\to\Ltwo{\G}\tensor\Ltwo{N,\theta}$.
Now $\mathcal{D}$ is a weakly dense left ideal in the crossed product
$\G\presb{\a}{\ltimes}N$, and there exists a (unique) n.s.f.~weight
$\tilde{\theta}_{0}$ on $\G\presb{\a}{\ltimes}N$ such that $(\Ltwo{\G}\tensor\Ltwo{N,\theta},\i,\tilde{\gnsmap})$
is a GNS construction for $\tilde{\theta}_{0}$ (see \cite[Lemma 3.3 and the preceding discussion, as well as Definition 3.4]{Vaes__unit_impl_LCQG}).
We let $\tilde{J},\tilde{\nabla}$ stand for the corresponding modular
conjugation and modular operator, respectively, and set $\tilde{T}:=\tilde{J}\tilde{\nabla}^{1/2}$.
\emph{The unitary implementation} of $\a$ is then the unitary $U:=(\widehat{J}\tensor J_{\theta})\tilde{J}$.
It satisfies
\begin{equation}
\a(x)=U^{*}(\one\tensor x)U\qquad \forall_{x\in N}\label{eq:unitary_impl}
\end{equation}
and
\begin{equation}
U(\widehat{J}\tensor J_{\theta})=(\widehat{J}\tensor J_{\theta})U^{*}.\label{eq:U_Js}
\end{equation}
Furthermore, $U\in\M(\Cz{\G}\tensor\mathcal{K}(\Ltwo{N,\theta}))$
is a representation of $\G$ on $\Ltwo{N,\theta}$ \cite[Definition 3.6, Proposition 3.7, Proposition 3.12 and Theorem 4.4]{Vaes__unit_impl_LCQG}. The choice of the weight $\theta$ is of no importance up to unitary equivalence by \cite[Proposition 4.1]{Vaes__unit_impl_LCQG}. It is easy to see that $U$ satisfies condition $\mathscr{R}$ of \prettyref{def:cond_R} with respect to $J_\theta$.

\begin{lem}[{\cite[Lemma 3.11]{Vaes__unit_impl_LCQG}}]
\label{lem:T_tilde}Denoting by $\widehat{T}$ the closure of the anti-linear
map $\widehat{\gnsmap}(\widehat{x})\mapsto\widehat{\gnsmap}(\widehat{x}^{*})$, $\widehat{x}\in\mathcal{N}_{\widehat{\varphi}}\cap\mathcal{N}_{\widehat{\varphi}}^{*}$,
we see that the subspace
\[
\linspan\bigl\{\a(x^{*})(\eta\tensor\gnsmap_{\theta}(y)):x,y\in\mathcal{N}_{\theta},\eta\in D(\widehat{T})\bigr\}
\]
is a core for $\tilde{T}$ and for every $x,y\in\mathcal{N}_{\theta}$
and $\eta\in D(\widehat{T})$,
\begin{equation}
\tilde{T}\a(x^{*})(\eta\tensor\gnsmap_{\theta}(y))=\a(y^{*})(\widehat{T}\eta\tensor\gnsmap_{\theta}(x)).\label{eq:T_tilde_T_hat}
\end{equation}

\end{lem}

\begin{defn}
Let $\a:N\to\Linfty{\G}\tensorn N$ be an action of a locally compact quantum group $\G$ on
a von Neumann algebra $N$. A (not necessarily normal) state $\theta$ of $N$ is \emph{invariant
under} $\a$ if $(\i\tensor\theta)\a=\theta(\cdot)\one$, or equivalently,
if $\theta(\om\tensor\i)\a=\om(\one)\theta(\cdot)$ for every $\om\in\Lone{\G}$.\end{defn}
\begin{rem}
\label{rem:inv_states}Suppose that $\G=G$, a locally compact group.
The above notion of invariance, which perhaps should have been called \emph{topological} invariance,
is stronger than the classical one,
i.e., $\theta\circ\a_{t}=\theta$ for all $t\in G$. Indeed, it
is not difficult to observe that the former entails the latter. The
converse, however, is not always true: for every compact (quantum)
group $\G$, there exists only one $\Delta$-invariant state of $\Linfty{\G}$,
namely the Haar state $h$, as any $\Delta$-invariant state $\theta$
satisfies $h(\cdot)=\theta(h(\cdot)\one)=\theta(h\tensor\i)\Delta=\theta(\cdot)$;
but, for instance, taking $\G$ to be the complex unit circle, $h$
is not the only state that is classically invariant under the translation action, represented in our context by $\Delta$
\cite[Proposition 2.2.11]{Lubotzky__book_1994}. Nevertheless,
if $G$ is discrete or $\theta$ is normal the two notions of invariance are easily seen to
be equivalent.\end{rem}

When $\a$ is an action of $\G$ on a von Neumann algebra $N$ that leaves invariant a \emph{faithful normal} state $\theta$, its unitary implementation takes a particularly simple form. To elaborate, one verifies that the formula
\begin{equation} \label{eq:unit_impl_inv_state_1}
(\i \tensor \om_{\gnsmap_\theta (y),\gnsmap_\theta(x)})(U^*) = (\id \tensor \theta)((\one \tensor y^*) \a(x))\qquad(x,y \in N),
\end{equation}
or equivalently
\begin{equation} \label{eq:unit_impl_inv_state_2}
(\om \tensor \id)(U^*)\gnsmap_\theta(x) = \gnsmap_\theta((\om \tensor \id)(\a(x)))\qquad(x\in N, \om \in \Lone{\G}),
\end{equation}
defines (uniquely) an isometry $U^*$, and that its adjoint $U$ is a representation of $\G$ on $\Ltwo{N,\theta}$, which is indeed a unitary by \cite[Corollary 4.15]{Brannan_Daws_Samei__cb_rep_of_conv_alg_of_LCQGs}, and which satisfies \prettyref{eq:unitary_impl}. One can essentially repeat the argument in the proof of \cite[Proposition 4.3]{Vaes__unit_impl_LCQG}, with $\one$ in place of $\delta^{-1}$ and the fact that $\widehat{\nabla} \tensor \nabla_\theta$ commutes with $U$ (see the proof of \cite[Theorem A.1]{Runde_Viselter_LCQGs_Ergodic_Thy}, and take into account the difference in the terminology) in lieu of \cite[Proposition 2.4, last formula]{Vaes__unit_impl_LCQG}, to infer that $U$ is the unitary implementation of $\a$.

\begin{defn} \label{def:ergodicity}
Let $\a:N\to\Linfty{\G}\tensorn N$ be an action
of a locally compact quantum group $\G$ on a von Neumann algebra $N$ leaving invariant a faithful normal state
$\theta$ of $N$. We say that $\a$ is \emph{ergodic} (respectively, \emph{weakly mixing}) if its implementing
unitary, when restricted to $\Ltwo{N,\theta} \ominus \CC \gnsmap_{\theta}(\one)$,
is ergodic (respectively, weakly mixing).
\end{defn}

For an arbitrary action $\a$ of $\G$ on $N$, one normally defines ergodicity of $\a$ as the equality of
its fixed-point algebra $N^{\a}:=\left\{ a\in N:\a(a)=\one\tensor a\right\} $
 and $\CC\one$. In the context of \prettyref{def:ergodicity}, this definition is known to be equivalent to ours when $\G$ is a group: see
\cite[Lemma 2.2]{Jadczyk__aut_groups} or \cite[Section 2, Theorem]{Herman_Takesaki__states_aut_groups},
where the abelianness assumption is unnecessary. We now show that this holds for all locally compact quantum groups (\prettyref{cor:operator_erg}).

The next result is inspired by \cite{Kovacs_Szucs__erg_thm_vN_alg}
as presented in \cite[Theorem 1]{Doplicher_Kastler_Stormer__inv_states}
and \cite[Section 2]{Jadczyk__aut_groups}; compare \cite{Duvenhage__erg_thm_quantum}
and \cite[Theorem 2.2]{Runde_Viselter_LCQGs_Ergodic_Thy}.
\begin{prop}
\label{prop:act_fixed_point_cond_exp}Let $\a:N\to\Linfty{\G}\tensorn N$
be an action of $\G$ on a von Neumann algebra $N$. Assume that $\a$
preserves a faithful normal state $\theta$ of $N$. Write $U\in\Linfty{\G}\tensorn\mathcal{B}(\Ltwo{N,\theta})$
for the unitary implementation of $\a$, and $p$ for the orthogonal
projection of $\Ltwo{N,\theta}$ onto $\inv(U)$. Then using the notation of \prettyref{prop:B_G__E}, there
exists a faithful normal conditional expectation $E$ of $N$ onto $N^{\alpha}$
given by
\begin{equation}
\om(E(a))=M\left((\i\tensor\om)\a(a)\right)\qquad(a\in N,\om\in N_{*}).\label{eq:fixed_point_cond_exp_def}
\end{equation}
Moreover, $E$ satisfies $E(a)p=pap$ for all $a\in N$. This identity
determines $E$ uniquely. Additionally, $E$ is $\a$-invariant: $E\left((\om\tensor\i)\a(a)\right)=E(a)$
for every $a\in N$ and every state $\om\in\Lone{\G}$.
\end{prop}
\begin{proof}
First, let us explain why the right-hand side of \prettyref{eq:fixed_point_cond_exp_def}
makes sense by proving that $(\i\tensor\om)\a(a)\in\overline{B(\G)^{*}}^{\left\Vert \cdot\right\Vert }$
for each $a\in N$, $\om\in N_{*}$.
It suffices to prove that $(\i\tensor\om_{J_{\theta}\gnsmap_{\theta}(c),J_{\theta}\gnsmap_{\theta}(b)})\a(a)\in B(\G)^{*}$
for every $b,c\in N$. But since $J_{\theta}\gnsmap_{\theta}(b)=J_{\theta}bJ_{\theta}\gnsmap_{\theta}(\one)$,
$J_{\theta}bJ_{\theta}\in N'$ and $\a(a)\in\Linfty{\G}\tensorn N$,
we get $(\i\tensor\om_{J_{\theta}\gnsmap_{\theta}(c),J_{\theta}\gnsmap_{\theta}(b)})\a(a)=(\i\tensor\om_{J_{\theta}\gnsmap_{\theta}(b^{*}c),\gnsmap_{\theta}(\one)})\a(a)$,
which by \prettyref{eq:unit_impl_inv_state_1} equals $(\i\tensor\om_{J_{\theta}\gnsmap_{\theta}(b^{*}c),\gnsmap_{\theta}(a)})(U^{*})\in B(\G)^{*}$.

By \prettyref{prop:B_G__E}, $M$ is a linear contraction on $\overline{B(\G)^{*}}^{\left\Vert \cdot\right\Vert }$.
Hence, \prettyref{eq:fixed_point_cond_exp_def} determines a well-defined
contraction $E:N\to N$. To show that $E$ maps into $N^{\a}$, pick
$a\in N$ and states $\om\in\Lone{\G}$, $\rho\in N_{*}$. We should
show that $(\om\tensor\rho)\a(E(a))=\rho(E(a))$. Setting $\nu:=(\om\tensor\rho)\circ\a$,
we have $(\om\tensor\rho)\a(E(a))=M\left((\i\tensor\nu)\a(a)\right)$
and $\rho(E(a))=M\left((\i\tensor\rho)\a(a)\right)$ by the definition
of $M$. However,
\[
\begin{split}(\i\tensor\nu)(\a(a)) & =(\i\tensor\om\tensor\rho)(\i\tensor\a)(\a(a))=(\i\tensor\om\tensor\rho)(\Delta\tensor\i)(\a(a))\\
 & =(\i\tensor\om)\Delta((\i\tensor\rho)\a(a))=\om\cdot((\i\tensor\rho)\a(a)).
\end{split}
\]
As a result, $(\om\tensor\rho)\a(E(a))=M\left((\i\tensor\nu)\a(a)\right)=M\left((\i\tensor\rho)\a(a)\right)=\rho(E(a))$,
proving that $E(a)\in N^{\a}$. It is also clear that $E(b)=b$ for
every $b\in N^{\a}$. In conclusion, $E$ is a conditional expectation
of $N$ onto $N^{\a}$.

Notice that $N^{\a}$ commutes with $p$. Indeed, it suffices to show
that if $b\in N^{\a}$ then $pbp=bp$. But $\a(b)=\one\tensor b$
if and only if $\one\tensor b$ commutes with $U$, if and only if
$b$ commutes with $(\om\tensor\i)(U)$ for every $\om\in\Lone{\G}$.
Let $\z\in\Ltwo{N,\theta}$. Then $pbp\z$ is in the closure of $\left\{ (\om\tensor\i)(U)bp\z:\om\in\Lone{\G}\text{ is a state}\right\} $
by \prettyref{lem:avg_semigp_one} and \prettyref{prop:ave_equals_proj_gen}. Now $\left(\om\tensor\i\right)(U)bp\z=b\left(\om\tensor\i\right)(U)p\z=bp\z$
for every $\om$, showing that $pbp\z=bp\z$.

Let $a\in N$. To observe that $E(a)p=pap$, take $\z,\eta\in\Ltwo{N,\theta}$.
Then
\[
\begin{split}\om_{\z,\eta}(E(a)p)=\om_{p\z,p\eta}(E(a)) & =M\left((\i\tensor\om_{p\z,p\eta})\a(a)\right)\\
 & =M\left((\i\tensor\om_{p\z,p\eta})(U^{*}(\one\tensor a)U)\right)=M(\om_{\z,\eta}(pap)\one)=\om_{\z,\eta}(pap),
\end{split}
\]
proving the desired equality. Since $p\gnsmap_{\theta}(\one)=\gnsmap_{\theta}(\one)$
and this vector is separating for $N$, the equality $E(a)p=pap$
determines $E$ uniquely, and $E$ is faithful and normal.

Let $a\in N$ and let $\om\in\Lone{\G}$, $\rho\in N_{*}$ be states.
Then
\[
(\i\tensor\rho)\a\left((\om\tensor\i)\a(a)\right)=(\om\tensor\i\tensor\rho)(\i\tensor\a)\a(a)=(\om\tensor\i\tensor\rho)(\Delta\tensor\i)\a(a)=((\i\tensor\rho)\a(a))\cdot\om.
\]
Hence
\[
\begin{split}\rho\left[E\left((\om\tensor\i)\a(a)\right)\right] & =M\left[(\i\tensor\rho)\a\left((\om\tensor\i)\a(a)\right)\right]\\
 & =M\left(((\i\tensor\rho)\a(a))\cdot\om\right)=M\left((\i\tensor\rho)\a(a)\right)=\rho(E(a)),
\end{split}
\]
implying that $E\left((\om\tensor\i)\a(a)\right)=E(a)$. That is,
$E$ is $\a$-invariant.
\end{proof}
\begin{proof}[We now present an alternative proof of the existence and properties of $E$ not using
the mean $M$]
By \prettyref{lem:avg_semigp_two} and \prettyref{prop:conv_to_Inv_U_proj}, there is a net $\left(\om_{i}\right)_{i\in \Ind}$
of states in $\Lone{\G}$ such that $\left((\om_{i}\tensor\i)(U^{*})\right)_{i\in \Ind}$
converges strongly to $p$. Let $a\in N$. Then since $U(\one\tensor p)=\one\tensor p$
by \prettyref{lem:inv_vec}, we have
\begin{equation}
pap=\lim_{i\in \Ind}(\om_{i}\tensor\i)(U^{*})ap=\lim_{i\in \Ind}(\om_{i}\tensor\i)\left(U^{*}(\one\tensor a)U\right)p=\lim_{i\in \Ind}(\om_{i}\tensor\i)\left(\a(a)\right)p\label{eq:fixed_point_cond_exp_lim}
\end{equation}
strongly. The bounded net $\left[(\om_{i}\tensor\i)\left(\a(a)\right)\right]_{i\in \Ind}$
in $N$ has a subnet $\left[(\om_{j}\tensor\i)\left(\a(a)\right)\right]_{j\in \Jnd}$
that converges in the weak$^{*}$ topology to an element $b\in N$.
By \prettyref{eq:fixed_point_cond_exp_lim}, $pap=bp$. This equality
determines $b$ uniquely in $N$ since $p\gnsmap_{\theta}(\one)=\gnsmap_{\theta}(\one)$
and this vector is separating for $N$. Moreover, as $p\gnsmap_{\theta}(a)=\gnsmap_{\theta}(b)$,
it follows from \prettyref{eq:unit_impl_inv_state_1} that $b\in N^{\alpha}$. Consequently, there
exists a linear map $E:N\to N^{\alpha}$ given by $E(a)p=pap$. From
\prettyref{eq:fixed_point_cond_exp_lim} it is clear that $E$ is
a contractive projection, thus a conditional expectation, which is
faithful and normal as in the first proof. Also, similarly to \prettyref{eq:fixed_point_cond_exp_lim},
\[
E((\om\tensor\i)\a(a))p=p(\om\tensor\i)\a(a)p=pap=E(a)p,\text{ thus }E((\om\tensor\i)\a(a))=E(a),
\]
for every $a\in N$ and every state $\om\in\Lone{\G}$, so $E$ is
$\a$-invariant.
\end{proof}
\begin{cor}\label{cor:operator_erg}
\label{cor:oper_erg}Under the assumptions of \prettyref{prop:act_fixed_point_cond_exp},
$\a$ is ergodic (namely: the restriction of $U$ to $\Ltwo{N,\theta}\ominus\CC\gnsmap_{\theta}(\one)$
is ergodic) if and only if $N^{\alpha}=\CC\one$.
\end{cor}
\begin{proof}
The implication $(\implies)$ is trivial by \prettyref{eq:unit_impl_inv_state_1}. To prove the converse,
suppose that $0\neq\z\in\Ltwo{N,\theta}\ominus\CC\gnsmap_{\theta}(\one)$
is invariant under $U$. Let $\left(a_{n}\right)_{n=1}^{\infty}$
be a sequence in $N$ such that $\gnsmap_{\theta}(a_{n})\xrightarrow{n\to\infty}\z$.
Then $\gnsmap_{\theta}(E(a_{n}))=p\gnsmap_{\theta}(a_{n})\xrightarrow{n\to\infty}p\z=\z$.
Since $\left\langle \gnsmap_{\theta}(\one),\z\right\rangle =0$, we
necessarily have $N^{\alpha}\supsetneqq\CC\one$.
\end{proof}

\begin{defn}
\label{def:asympt_inv}Let $\a$ be an action of a locally compact quantum group $\G$ on a
von Neumann algebra $N$ with an invariant faithful normal state $\theta$.
A bounded net $\left(x_{i}\right)_{i \in \Ind}$ in $N$ is called
\emph{asymptotically invariant under} $\a$ if for every normal state
$\om$ of $\Linfty{\G}$, we have $(\om\tensor\i)\a(x_{i})-x_{i} \stackrel{i \in \Ind}{\longrightarrow} {}0$
strongly. Such a net is said to be \emph{trivial} if $x_{i}-\theta(x_{i})\one \stackrel{i \in \Ind}{\longrightarrow} {}0$
strongly. We say that $\a$ is \emph{strongly ergodic} if
all its asymptotically invariant nets are trivial.\end{defn}

Strong ergodicity evidently implies ergodicity.

In the setting of \prettyref{def:asympt_inv}, denote by $U\in\M(\Cz{\G} \tensor \mc K (\Ltwo{N,\theta}))$
the unitary implementation of $\a$. Observe that when $\left(x_{i}\right)_{i \in \Ind}$
is asymptotically invariant and does not converge strongly to $0$,
we may assume, by passing to a subnet if necessary, that $\left(\left\Vert \gnsmap_{\theta}(x_{i})\right\Vert \right)_{i\in \Ind}$
is bounded from below, and the normalised net $(\frac{1}{\left\Vert \gnsmap_{\theta}(x_{i})\right\Vert }\gnsmap_{\theta}(x_{i}))_{i\in \Ind}$
is then almost invariant under $U$, because by \prettyref{eq:unit_impl_inv_state_2}, for every normal state
$\om$ of $\Linfty{\G}$, since $\a$ preserves $\theta$,
\[
(\om\tensor\i)(U^{*})\gnsmap_{\theta}(x_{i})-\gnsmap_{\theta}(x_{i})=\gnsmap_{\theta}((\om\tensor\i)\a(x_{i})-x_{i}) \stackrel{i \in \Ind}{\longrightarrow} 0.
\]

For the next definition, let $N$ be a von Neumann algebra in standard
form on $\Ltwo N$ with modular conjugation $J$ and let $n\in\N$. We
will consider $M_{n}\tensor N$ as acting standardly on $\Ltwo{M_{n},\tr}\tensor\Ltwo N$
which, as a vector space, is just $M_{n}\tensor\Ltwo N$. The positive
cone \cite{Haagerup__standard_form} of $M_{n}\tensor N$ in this
representation is the closure of the set of sums of vectors of the
form $\left(\xi_{i}J\xi_{j}\right)_{i,j=1}^{n}$ for $\xi_{1},\ldots,\xi_{n}$
in the left Hilbert algebra of $N$ inside $\Ltwo N$.
\begin{defn}[{\cite[2.5]{Sauvageot__impl_canon_co_actions}}]
\label{def:pres_pos_cone}Let $N$ be a von Neumann algebra and $u=\left(u_{ij}\right)_{i,j=1}^{n}\in M_{n}\tensor\mathcal{B}(\Ltwo N)$.
We say that $u$ \emph{preserves the positive cone} if for every matrix
$\left(\z_{ij}\right)_{i,j=1}^{n}\in\Ltwo{M_{n},\tr}\tensor\Ltwo N$
in the positive cone of $M_{n}\tensor N$, the matrix $\left(u_{ij}\z_{ij}\right)_{i,j=1}^{n}$
also belongs to the positive cone of $M_{n}\tensor N$.\end{defn}
\begin{rem}
\label{rem:pres_pos_cone_adj}By self-duality of the positive cone,
$\left(u_{ij}\right)_{i,j=1}^{n}$ in $M_{n}\tensor\mathcal{B}(\Ltwo N)$
preserves it if and only if $\left(u_{ij}^{*}\right)_{i,j=1}^{n}$
preserves it.
\end{rem}
We will require a generalisation of several results of Sauvageot \cite{Sauvageot__impl_canon_co_actions}.
Since the proof of \cite[Remarque 4.6, 2]{Sauvageot__impl_canon_co_actions}
is not explicit, and, as mentioned in \cite{Vaes__unit_impl_LCQG},
\cite[Lemme 4.1]{Sauvageot__impl_canon_co_actions} is incorrect,
we provide full details.

\begin{prop}
\label{prop:scaled_U_preserves_pos_cone}Let $\G$ be a locally compact quantum group acting
on a von Neumann algebra $N$ by an action $\a:N\to\Linfty{\G}\tensorn N$.
Let $U\in\M(\Cz{\G}\tensor\mathcal{K}(\Ltwo N))$ be the unitary
implementation of $\a$. Then for every $\xi_{1},\ldots,\xi_{n}\in D(\widehat{\nabla}^{-1/2})$,
the matrices $((\om_{\widehat{\nabla}^{-1/2}\xi_{i},\xi_{j}}\tensor\i)(U^{*}))_{1\leq i,j\leq n}$
and $((\om_{\xi_{j},\widehat{\nabla}^{-1/2}\xi_{i}}\tensor\i)(U))_{1\leq i,j\leq n}$
preserve the positive cone.\end{prop}
\begin{proof}
Fix $a\in N$, $b,c,d\in\mathcal{N}_{\theta}$, $\xi\in D(\widehat{\nabla}^{-1/2})$
and $\eta\in\Ltwo{\G}$. Note that $bJ_{\theta}\gnsmap_{\theta}(d)=J_{\theta}dJ_{\theta}\gnsmap_{\theta}(b)$
with $J_{\theta}dJ_{\theta}\in N'$. By \prettyref{eq:unitary_impl},
we have
\begin{equation}
\begin{split}\left\langle \eta\tensor bJ_{\theta}\gnsmap_{\theta}(d),U^{*}(\xi\tensor aJ_{\theta}\gnsmap_{\theta}(c))\right\rangle  & =\left\langle \eta\tensor J_{\theta}dJ_{\theta}\gnsmap_{\theta}(b),\a(a)U^{*}(\xi\tensor J_{\theta}\gnsmap_{\theta}(c))\right\rangle \\
 & =\left\langle \a(a^{*})(\eta\tensor\gnsmap_{\theta}(b)),(\one\tensor J_{\theta}d^{*}J_{\theta})U^{*}(\xi\tensor J_{\theta}\gnsmap_{\theta}(c))\right\rangle .
\end{split}
\label{eq:unitary_impl__inn_prod}
\end{equation}
Using \prettyref{eq:unitary_impl}, \prettyref{eq:U_Js} and \prettyref{eq:T_tilde_T_hat}
we obtain, as $\widehat{J}\xi=\widehat{T}\widehat{\nabla}^{-1/2}\xi$,
\begin{equation}
\begin{split}(\one\tensor J_{\theta}d^{*}J_{\theta})U^{*}(\xi\tensor J_{\theta}\gnsmap_{\theta}(c)) & =(\widehat{J}\tensor J_{\theta})(\one\tensor d^{*})U(\widehat{J}\xi\tensor\gnsmap_{\theta}(c))\\
 & =(\widehat{J}\tensor J_{\theta})U\a(d^{*})(\widehat{J}\xi\tensor\gnsmap_{\theta}(c))\\
 & =\tilde{J}\tilde{T}\a(c^{*})(\widehat{\nabla}^{-1/2}\xi\tensor\gnsmap_{\theta}(d))\\
 & =\tilde{\nabla}^{1/2}\a(c^{*})(\widehat{\nabla}^{-1/2}\xi\tensor\gnsmap_{\theta}(d)).
\end{split}
\label{eq:unitary_impl__T_tilde}
\end{equation}
To conclude,
\begin{equation}
\left\langle \eta\tensor bJ_{\theta}\gnsmap_{\theta}(d),U^{*}(\xi\tensor aJ_{\theta}\gnsmap_{\theta}(c))\right\rangle =\left\langle \a(a^{*})(\eta\tensor\gnsmap_{\theta}(b)),\tilde{\nabla}^{1/2}\a(c^{*})(\widehat{\nabla}^{-1/2}\xi\tensor\gnsmap_{\theta}(d))\right\rangle .\label{eq:U_preserves_pos_cone_stage_1}
\end{equation}

Take now $a_{1},\ldots,a_{n},b_{1},\ldots,b_{n}\in\mathcal{N}_{\theta}$
and $\xi_{1},\ldots,\xi_{n}\in D(\widehat{\nabla}^{-1/2})$. From \prettyref{eq:U_preserves_pos_cone_stage_1},
\begin{multline*}
\sum_{i,j=1}^{n}\left\langle b_{i}J_{\theta}\gnsmap_{\theta}(b_{j}),(\om_{\widehat{\nabla}^{-1/2}\xi_{i},\xi_{j}}\tensor\i)(U^{*})a_{i}J_{\theta}\gnsmap_{\theta}(a_{j})\right\rangle \\
\begin{split} & =\sum_{i,j=1}^{n}\left\langle \widehat{\nabla}^{-1/2}\xi_{i}\tensor b_{i}J_{\theta}\gnsmap_{\theta}(b_{j}),U^*(\xi_{j}\tensor a_{i}J_{\theta}\gnsmap_{\theta}(a_{j}))\right\rangle \\
 & =\sum_{i,j=1}^{n}\left\langle \a(a_{i}^{*})(\widehat{\nabla}^{-1/2}\xi_{i}\tensor\gnsmap_{\theta}(b_{i})),\tilde{\nabla}^{1/2}\a(a_{j}^{*})(\widehat{\nabla}^{-1/2}\xi_{j}\tensor\gnsmap_{\theta}(b_{j}))\right\rangle \\
 & =\left\Vert \tilde{\nabla}^{1/4}\sum_{i=1}^{n}\a(a_{i}^{*})(\widehat{\nabla}^{-1/2}\xi_{i}\tensor\gnsmap_{\theta}(b_{i}))\right\Vert ^{2}\geq0.
\end{split}
\end{multline*}
The desired conclusion follows from the self-duality of the positive
cone and \prettyref{rem:pres_pos_cone_adj}.
\end{proof}

\begin{prop}
\label{prop:act_pos_cone_applications}Let $\a:N\to\linfty{\G}\tensorn N$
be an action of a discrete quantum group $\G$ on a von Neumann algebra
$N$. Denote its unitary implementation by $U$.
\begin{enumerate}
\item \label{enu:act_pos_cone_applications_1}Assume that a state $m$ of
$N$ is invariant under $\a$. Then there exists a net of unit vectors
$\left(\z_{\iota}\right)_{\iota \in \Ind}$ in the positive cone $\mathcal{P}$ of
$N$ in $\Ltwo N$ that is almost invariant under $U$ and such that
$\om_{\z_{\iota}} \stackrel{\iota \in \Ind}{\longrightarrow} m$ in the weak$^{*}$ topology of $N^{*}$.
\item \label{enu:act_pos_cone_applications_2}If a normal positive functional
$\rho$ of $N$ is invariant under $\a$, then the unique vector $\z\in\mathcal{P}$
with $\rho=\om_{\z}$ is invariant under $U$.
\item \label{enu:act_pos_cone_applications_3}The unitary implementation
of $\a$ is the unique unitary $U$ in $\M(\cz{\G}\tensor\mathcal{K}(\Ltwo N))$
that satisfies \prettyref{eq:unitary_impl} and such that for every
$\xi_{1},\ldots,\xi_{n}\in D(\widehat{\nabla}^{-1/2})$, the matrix $((\om_{\xi_{j},\widehat{\nabla}^{-1/2}\xi_{i}}\tensor\i)(U))_{1\leq i,j\leq n}$
preserves the positive cone.
\end{enumerate}
\end{prop}
\begin{proof}
\prettyref{enu:act_pos_cone_applications_1} From the standard convexity
argument it follows that one can find a net $\left(m_{\iota}\right)_{\iota \in \Ind}$
of normal states of $N$ converging in the weak$^{*}$ topology to the state
$m$ such that 
\begin{equation}\label{eq:almost_inv_states_m_iota}
\left\Vert (\om_{\eta',\eta}\tensor m_{\iota})\a-\left\langle \eta',\eta\right\rangle m_{\iota}\right\Vert \stackrel{\iota \in \Ind}{\longrightarrow} {}0 \qquad \forall_{\eta,\eta'\in\ell^{2}(\G)}.
\end{equation}
For every $\iota \in \Ind$, write $\zeta_\iota$ for the unique vector in $\mathcal{P}$ such that $m_\iota=\om_{\z_\iota}$.

Recall that $\ell^{2}(\G)$ decomposes as $\bigoplus_{\gamma \in \Irred{\widehat{\G}}}\ell^{2}(\G)_{\gamma}$. Moreover, for every $\gamma\in\Irred{\widehat{\G}}$, the operator $\widehat{\nabla}$
restricts to a bijection over $\ell^{2}(\G)_{\gamma}$. Fix a non-empty finite set
$F\subseteq\Irred{\widehat{\G}}$ and $\e>0$. Let $\left(\eta_{i}\right)_{i=1}^{n}$
be an orthonormal basis of $\ell^{2}(\G)_{F} := \bigoplus_{\gamma \in F}\ell^{2}(\G)_{\gamma}$ and set $\overline{u}:=((\om_{\eta_{j},\widehat{\nabla}^{-1/2}\eta_{i}}\tensor\i)(U))_{1\leq i,j\leq n}$ and $c := \Vert \widehat{\nabla}^{1/2} | _{\ell^{2}(\G)_{F}}\Vert$.
Assume that $\rho$ is a normal state of $N$ with
\begin{equation} \label{eq:almost_inv_states_quantitative}
\left\Vert (\om_{\widehat{\nabla}^{-1/2}\eta_{j},\widehat{\nabla}^{-1/2}\eta_{i}}\tensor\rho)\a-\langle\widehat{\nabla}^{-1/2}\eta_{j},\widehat{\nabla}^{-1/2}\eta_{i}\rangle\rho\right\Vert \leq \frac{\e^2}{n^4 c^2} 
\qquad \forall_{1\leq i,j\leq n},
\end{equation}
and write $\z$ for the unique vector
in $\mathcal{P}$ such that $\rho=\om_{\z}$. For every $\eta\in\ell^{2}(\G)_{F}$,
\[
U(\eta\tensor\z)=\sum_{k=1}^{n}\eta_{k}\tensor(\om_{\eta_{k},\eta}\tensor\i)(U)\z.
\]
Therefore, for every $1\leq i,j\leq n$ and $x\in N$,
\[
(\om_{\widehat{\nabla}^{-1/2}\eta_{j},\widehat{\nabla}^{-1/2}\eta_{i}}\tensor\rho)\a(x)=\left\langle U(\widehat{\nabla}^{-1/2}\eta_{j}\tensor\z),(\one\tensor x)U(\widehat{\nabla}^{-1/2}\eta_{i}\tensor\z)\right\rangle =\sum_{k=1}^{n}\left\langle \overline{u}_{jk}\z,x\overline{u}_{ik}\z\right\rangle ,
\]
so that
\begin{equation}
\left|\sum_{k=1}^{n}\left\langle \overline{u}_{jk}\z,x\overline{u}_{ik}\z\right\rangle -\langle\widehat{\nabla}^{-1/2}\eta_{j},\widehat{\nabla}^{-1/2}\eta_{i}\rangle\left\langle \z,x\z\right\rangle \right|\leq\frac{\e^2}{n^4 c^2}\left\Vert x\right\Vert .\label{eq:almost_inv_states_vectors}
\end{equation}

Set $Z_{1}:=(\langle\eta_{j},\widehat{\nabla}^{-1/2}\eta_{i}\rangle\z)_{1\leq i,j\leq n}$
and $Z_{2}:=\left(\overline{u}_{ij}\z\right)_{1\leq i,j\leq n}$.
Both belong to the positive cone of $M_{n}\tensor N$ in $\Ltwo{M_{n},\tr}\tensor\Ltwo N$,
the second by \prettyref{prop:scaled_U_preserves_pos_cone}. Consider
the functionals $\om_{Z_{1}},\om_{Z_{2}}$ over $M_{n}\tensor N$.
For every $X=\left(x_{ij}\right)_{1\leq i,j\leq n}\in M_{n}\tensor N$,
we have
\[
\om_{Z_{1}}(X)=\sum_{i,j,k=1}^{n}\left\langle \langle\eta_{k},\widehat{\nabla}^{-1/2}\eta_{j}\rangle\z,x_{ji}\langle\eta_{k},\widehat{\nabla}^{-1/2}\eta_{i}\rangle\z\right\rangle =\sum_{i,j=1}^{n}\langle\widehat{\nabla}^{-1/2}\eta_{j},\widehat{\nabla}^{-1/2}\eta_{i}\rangle\left\langle \z,x_{ji}\z\right\rangle
\]
and $\om_{Z_{2}}(X)=\sum_{i,j,k=1}^{n}\left\langle \overline{u}_{jk}\z,x_{ji}\overline{u}_{ik}\z\right\rangle $.
Hence, by \prettyref{eq:almost_inv_states_vectors},
\[
\left|(\om_{Z_{1}}-\om_{Z_{2}})(X)\right|\leq\sum_{i,j=1}^{n}\frac{\e^2}{n^4 c^2}\left\Vert x_{ji}\right\Vert \leq\frac{\e^2}{n^2 c^2}\left\Vert X\right\Vert .
\]
In conclusion, $\left\Vert \om_{Z_{1}}-\om_{Z_{2}}\right\Vert \leq \frac{\e^2}{n^2 c^2}$.
By the Powers--St{\o}rmer inequality \cite[Lemma 2.10]{Haagerup__standard_form},
it follows that $\left\Vert Z_{1}-Z_{2}\right\Vert \leq \frac{\e}{n c}$. 
As a result, for each $1 \leq i,j \leq n$,
\[
\left\Vert (\om_{\eta_{j},\widehat{\nabla}^{-1/2}\eta_{i}}\tensor\i)(U - \one_{\Linfty{\G}}\tensor\one_{\B(\Ltwo{N})})\z \right\Vert =
\left\Vert (Z_1)_{ij} - (Z_2)_{ij} \right\Vert \le \frac{\e}{n c}.
\]
A simple calculation thus shows that for every $\a,\be \in \ell^{2}(\G)_{F}$,
\[
\left\Vert (\om_{\a,\be}\tensor\i)(U - \one_{\Linfty{\G}}\tensor\one_{\B(\Ltwo{N})})\z \right\Vert
\leq \frac{\e}{n c} n c \left\Vert \a \right\Vert \left\Vert \be \right\Vert = \e \left\Vert \a \right\Vert \left\Vert \be \right\Vert.
\]

To complete the proof, notice that by \eqref{eq:almost_inv_states_m_iota}, there exists $\iota_0 \in \Ind$ such that $\rho := m_\iota$ satisfies \eqref{eq:almost_inv_states_quantitative} for every $\iota_0 \leq \iota \in \Ind$. Since $\e$ was arbitrary, we get 
\[
\lim_{\iota\in\Ind} (\om_{\a,\be}\tensor\i)(U - \one_{\Linfty{\G}}\tensor\one_{\B(\Ltwo{N})})\z_\iota = 0
\]
for all $\a,\be \in \ell^{2}(\G)_{F}$. Letting $F$ vary, this holds for all $\a,\be \in \ell^{2}(\G)$ by density. From \prettyref{lem:almost_inv_vec}, \prettyref{enu:almost_inv_vect_2}, the net $(\z_\iota)_{\iota \in \Ind}$ is almost invariant under $U$.

\prettyref{enu:act_pos_cone_applications_2} Repeat the proof of \prettyref{enu:act_pos_cone_applications_1}
with $\e=0$.

\prettyref{enu:act_pos_cone_applications_3} Suppose that $U,V\in\linfty{\G}\tensorn\mathcal{B}(\Ltwo N)$
satisfy the given requirements. Fix $\gamma\in\Irred{\widehat{\G}}$, let
$\left(\eta_{i}\right)_{i=1}^{n}$ be an orthonormal basis of $\ell^{2}(\G)_{\gamma}$,
and set $\overline{u}:=((\om_{\eta_{j},\widehat{\nabla}^{-1/2}\eta_{i}}\tensor\i)(U))_{1\leq i,j\leq n}$,
$\overline{v}:=((\om_{\eta_{j},\widehat{\nabla}^{-1/2}\eta_{i}}\tensor\i)(V))_{1\leq i,j\leq n}$.
As in the proof of \prettyref{enu:act_pos_cone_applications_1}, for
every $x\in N$ and $\z\in\Ltwo N$,
\[
\sum_{k=1}^{n}\left\langle \overline{u}_{jk}\z,x\overline{u}_{ik}\z\right\rangle =(\om_{\widehat{\nabla}^{-1/2}\eta_{j},\widehat{\nabla}^{-1/2}\eta_{i}}\tensor\om_{\z})\a(x)=\sum_{k=1}^{n}\left\langle \overline{v}_{jk}\z,x\overline{v}_{ik}\z\right\rangle .
\]
Consequently, if $\z\in\mathcal{P}$ and we put $Z_{u}:=\left(\overline{u}_{ij}\z\right)_{1\leq i,j\leq n}$
and $Z_{v}:=\left(\overline{v}_{ij}\z\right)_{1\leq i,j\leq n}$,
then for all $X=\left(x_{ij}\right)_{1\leq i,j\leq n}$ we have $\om_{Z_{u}}(X)=\sum_{i,j,k=1}^{n}\left\langle \overline{u}_{jk}\z,x_{ji}\overline{u}_{ik}\z\right\rangle =\sum_{i,j,k=1}^{n}\left\langle \overline{v}_{jk}\z,x_{ji}\overline{v}_{ik}\z\right\rangle =\om_{Z_{v}}(X)$.
Since $Z_{u},Z_{v}$ belong to the positive cone of $M_{n}\tensor N$
in $\Ltwo{M_{n},\tr}\tensor\Ltwo N$ by assumption,
this entails $Z_{u}=Z_{v}$. As $\mathcal{P}$ is total in $\Ltwo N$,
we get $\overline{u}=\overline{v}$. Then it is easy to see that in fact $U=V$, as $\gamma \in \Irred{\hQG}$ was arbitrary and $\left(\eta_{i}\right)_{i=1}^{n}$ was a basis of $\ell^{2}(\G)_{\gamma}$. Proposition \ref{prop:scaled_U_preserves_pos_cone} ends the proof.\end{proof}

\begin{cor}[{cf.~\cite[Lemme 4.4 and Remarque 4.6]{Sauvageot__impl_canon_co_actions}}]
\label{cor:U_preserves_pos_cone}Let $\G$ be a locally compact quantum group with trivial
scaling group, acting on a von Neumann algebra $N$ by an action $\a:N\to\Linfty{\G}\tensorn N$.
Let $U\in\M(\Cz{\G}\tensor\mathcal{K}(\Ltwo N))$ be the unitary
implementation of $\a$. Then for every $\xi_{1},\ldots,\xi_{n}\in\Ltwo{\G}$,
the matrices $((\om_{\xi_{i},\xi_{j}}\tensor\i)(U^{*}))_{1\leq i,j\leq n}$
and $((\om_{\xi_{j},\xi_{i}}\tensor\i)(U))_{1\leq i,j\leq n}$ preserve
the positive cone. If $\G$ is discrete, then the unitary implementation
of $\a$ is the unique unitary $U$ in $\M(\cz{\G}\tensor\mathcal{K}(\Ltwo N))$
that satisfies \prettyref{eq:unitary_impl} and such that for every
$\xi_{1},\ldots,\xi_{n}\in\ltwo{\G}$, the matrix $((\om_{\xi_{j},\xi_{i}}\tensor\i)(U))_{1\leq i,j\leq n}$
preserves the positive cone.\end{cor}
\begin{proof}
Let $\xi_{1},\ldots,\xi_{n}\in D(\widehat{\nabla}^{-1/2})$. Since the
scaling group of $\G$ is trivial, $\widehat{\nabla}$ is affiliated with
the commutant $\Linfty{\G}'$. Hence, as $U\in\Linfty{\G}\tensorn\mathcal{B}(\Ltwo N)$,
we have $(\om_{\widehat{\nabla}^{-1/2}\xi_{i},\xi_{j}}\tensor\i)(U^{*})=(\om_{\widehat{\nabla}^{-1/4}\xi_{i},\widehat{\nabla}^{-1/4}\xi_{j}}\tensor\i)(U^{*})$
for every $i,j$. Thus $((\om_{\widehat{\nabla}^{-1/4}\xi_{i},\widehat{\nabla}^{-1/4}\xi_{j}}\tensor\i)(U^{*}))_{1\leq i,j\leq n}$
preserves the positive cone by \prettyref{prop:scaled_U_preserves_pos_cone}.
Since $\widehat{\nabla}^{-1/4}D(\widehat{\nabla}^{-1/2})$ is dense in $\Ltwo{\G}$,
$((\om_{\xi_{i},\xi_{j}}\tensor\i)(U^{*}))_{1\leq i,j\leq n}$ preserves
the positive cone for all $\xi_{1},\ldots,\xi_{n}\in\Ltwo{\G}$. The
second statement follows similarly by using \prettyref{prop:act_pos_cone_applications},
\prettyref{enu:act_pos_cone_applications_3}.
\end{proof}
For the next lemma, recall that for a Hilbert space $\K$, $\mathcal{B}(\K)$
is standardly represented (e.g., by using the trace on $\mathcal{B}(\K)$)
on $\K\tensor\overline{\K}$ by $\pi_{\K}:\mathcal{B}(\K)\ni x\mapsto x\tensor\one$
with conjugation $\J:\K\tensor\overline{\K}\to\K\tensor\overline{\K}$,
$\z\tensor\overline{\eta}\mapsto\eta\tensor\overline{\z}$. The positive
cone of $M_{n}\tensor\mathcal{B}(\K)$ in $M_{n}\tensor\K\tensor\overline{\K}$
is the closure of the set of vectors of the form $(\sum_{k=1}^{m}\z_{k}^{i}\tensor\overline{\z_{k}^{j}})_{1\leq i,j\leq n}$
with $\z_{k}^{i}\in\K$ ($1\leq i\leq n$, $1\leq k\leq m$).
\begin{lem}
\label{lem:V_V_bar_unit_impl}Let $V\in\M(\cz{\G}\tensor\mathcal{K}(\K))$
be a representation of a discrete quantum group $\G$ on
a Hilbert space $\K$. The unitary implementation of the action $\a$
of $\G$ on $\mathcal{B}(\K)$ given by $\a(x):=V^{*}(\one\tensor x)V$,
$x\in\mathcal{B}(\K)$, on $\K\tensor\overline{\K}$ is $V\tpr V^{c}$.\end{lem}
\begin{proof}
Observe that by the definition of $V^{c}$ and \prettyref{eq:J_tau_impl},
the operator $V\tpr V^{c}=V_{13}^{c}V_{12}$ acting
on $\ltwo{\G}\tensor\K\tensor\overline{\K}$ is equal to $(\widehat{J}\tensor\J)V^{*}(\widehat{J}\tensor\J)V$,
where $V\in\linfty{\G}\tensorn\mathcal{B}(\K)$ is represented on
$\ltwo{\G}\tensor\K\tensor\overline{\K}$ using $\i\tensor\pi_{\K}$.

The representation $V\tpr V^{c}$ obviously implements $\a$.
By \prettyref{prop:act_pos_cone_applications}, \prettyref{enu:act_pos_cone_applications_3}
it suffices to prove that for every $\a_{1},\ldots,\a_{n}\in D(\widehat{\nabla}^{-1/2})$,
the matrix $((\om_{\a_{j},\widehat{\nabla}^{-1/2}\a_{i}}\tensor\i)(V\tpr V^{c}))_{1\leq i,j\leq n}$
preserves the positive cone. For $\a \in D(\widehat{\nabla}^{-1/2})$, $\be \in \ltwo{\G}$
and $\z_{1},\z_{2},\eta_{1},\eta_{2}\in\K$, we have
\begin{multline*}
\left\langle \be\tensor\z_{2}\tensor\overline{\eta_{2}},(V\tpr V^{c})(\widehat{\nabla}^{-1/2}\a\tensor\z_{1}\tensor\overline{\eta_{1}})\right\rangle   \\
\begin{split} &=\left\langle \be\tensor\z_{2}\tensor\overline{\eta_{2}},(\widehat{J}\tensor\J)V^{*}(\widehat{J}\tensor\J)V(\widehat{\nabla}^{-1/2}\a\tensor\z_{1}\tensor\overline{\eta_{1}})\right\rangle \\
 & =\left\langle (\widehat{J}\tensor\J)V(\widehat{J}\be\tensor\eta_{2}\tensor\overline{\z_{2}}),V(\widehat{\nabla}^{-1/2}\a\tensor\z_{1}\tensor\overline{\eta_{1}})\right\rangle \\
 & =\left\langle \widehat{J}(\i\tensor\om_{\eta_{1},\eta_{2}})(V)\widehat{J}\be,(\i\tensor\om_{\z_{2},\z_{1}})(V)\widehat{\nabla}^{-1/2}\a\right\rangle .
\end{split}
\end{multline*}
Thus further, by \prettyref{eq:reps_T_hat},
\begin{equation}
\begin{split}\left\langle \be\tensor\z_{2}\tensor\overline{\eta_{2}},(V\tpr V^{c})(\widehat{\nabla}^{-1/2}\a\tensor\z_{1}\tensor\overline{\eta_{1}})\right\rangle  & =\left\langle \widehat{J}(\i\tensor\om_{\eta_{1},\eta_{2}})(V)\widehat{J}\be,\widehat{T}(\i\tensor\om_{\z_{1},\z_{2}})(V)\widehat{J}\a\right\rangle \\
 & =\left\langle \widehat{\nabla}^{1/2}(\i\tensor\om_{\z_{1},\z_{2}})(V)\widehat{J}\a,(\i\tensor\om_{\eta_{1},\eta_{2}})(V)\widehat{J}\be\right\rangle .
\end{split}
\label{eq:V_V_bar_pos}
\end{equation}

Take $n,m\in\N$, $\a_{1},\ldots,\a_{n}\in D(\widehat{\nabla}^{-1/2})$
and $\z_{k}^{i},\xi_{k}^{i}\in\K$ ($1\leq i\leq n$, $1\leq k\leq m$).
By \prettyref{eq:V_V_bar_pos}, applying $((\om_{\a_{j},\widehat{\nabla}^{-1/2}\a_{i}}\tensor\i)(V\tpr V^{c}))_{1\leq i,j\leq n}$
to $(\sum_{k}\z_{k}^{i}\tensor\overline{\z_{k}^{j}})_{1\leq i,j\leq n}$
component-wise and taking the inner product with $(\sum_{\ell}\xi_{\ell}^{i}\tensor\overline{\xi_{\ell}^{j}})_{1\leq i,j\leq n}$,
we obtain
\begin{multline*}
\sum_{i,j,k,\ell}\left\langle \a_{j}\tensor\xi_{\ell}^{i}\tensor\overline{\xi_{\ell}^{j}},(V\tpr V^{c})(\widehat{\nabla}^{-1/2}\a_{i}\tensor\z_{k}^{i}\tensor\overline{\z_{k}^{j}})\right\rangle \\
\begin{split} & =\sum_{i,j,k,\ell}\left\langle \widehat{\nabla}^{1/2}(\i\tensor\om_{\z_{k}^{i},\xi_{\ell}^{i}})(V)\widehat{J}\a_{i},(\i\tensor\om_{\z_{k}^{j},\xi_{\ell}^{j}})(V)\widehat{J}\a_{j}\right\rangle \\
 & =\sum_{k,\ell=1}^m\left\langle \widehat{\nabla}^{1/2}\left(\sum_{i=1}^n(\i\tensor\om_{\z_{k}^{i},\xi_{\ell}^{i}})(V)\widehat{J}\a_{i}\right),\sum_{j=1}^n(\i\tensor\om_{\z_{k}^{j},\xi_{\ell}^{j}})(V)\widehat{J}\a_{j}\right\rangle \geq0
\end{split}
\end{multline*}
by the positivity of $\widehat{\nabla}^{1/2}$. By the self-duality of
the positive cone, it is preserved by $((\om_{\a_{j},\widehat{\nabla}^{-1/2}\a_{i}}\tensor\i)(V\tpr V^{c}))_{1\leq i,j\leq n}$,
and the proof is complete.
\end{proof}

\section{Property (T)}\label{sec:prop_T}

In this section we recall the fundamental notion of Property (T) for locally compact quantum groups.  We define different notions of having a `Kazhdan pair' in this setting, and
show that they are all equivalent to Property (T).  We also show that if a locally compact quantum group $\QG$ admits a morphism with a dense image from a quantum group with Property (T), then also $\QG$ itself has Property (T).

\begin{deft}
A locally compact quantum group has \emph{Property (T)} if each of its representations which has almost-invariant vectors has a non-trivial invariant vector.
\end{deft}

\begin{deft} \label{def:T_r_s}
A locally compact quantum group $\QG$ has \emph{Property (T)$^{1,1}$} if for every representation $U$ of $\QG$ with almost invariant vectors the representation $U\tp U^{c}$ has a non-zero invariant vector. Further if $r,s \in \bn$ then we say that $\QG$ has \emph{Property (T)$^{r,s}$} if for every representation $U$ of $\QG$ with almost invariant vectors the representation $U^{\tpsmall r}\tp (U^{c})^{\tpsmall s}$ has a non-zero invariant vector.
\end{deft}

Note that by the discussion after Definition \ref{def:wm} if $\QG$ has a trivial scaling group, then it has Property (T)$^{1,1}$ if and only if each of its representations which has almost-invariant vectors is not weakly mixing.

\begin{prop}\label{prop:T11easy}
If a locally compact quantum group $\G$ has Property (T), then it also has Property  (T)$^{r,s}$ for every $r, s \in \bn$.
\end{prop}
\begin{proof}
This is trivial, for if a representation $U$ admits a non-zero invariant vector, so does $U^{\tpsmall r}\tp (U^{c})^{\tpsmall s}$.
\end{proof}

\begin{deft}
Let $\G$ be locally compact quantum group. 
\begin{enumerate}
\item A \emph{Kazhdan pair of type 1 for $\G$} is a finite set $Q 
\subseteq L^2(\G)$ and $\e>0$, such that for any representation
$U\in \M(\C_0(\G)\otimes \mc K(\sH))$, if there is a unit vector $\xi\in\sH$
with $\| U(\eta\otimes\xi) - \eta\otimes\xi \| < \e$ for each $\eta \in Q$,
then $U$ has a non-zero invariant vector.
\item A \emph{Kazhdan pair of type 2 for $\G$} is a finite set $Q 
\subseteq \C_0(\G)$ and $\e>0$, such that for any representation
$\mc U\in\mc L(\C_0(\G)\otimes\sH)$, if there is a unit vector $\xi\in\sH$
with $\| \mc U(a\otimes\xi) - a\otimes\xi \| < \e$ for each $a \in Q$,
then $\mc U$ has a non-zero invariant vector.
\item A \emph{Kazhdan pair of type 3 for $\G$} is a finite set $Q
\subseteq \C_0^u(\hh\G)$ and $\e>0$, such that for any representation
$\phi:\C_0^u(\hh\G)\rightarrow\mc B(\sH)$, if there is a unit vector
$\xi\in \sH$ with $\|\phi(x)\xi - \hh\epsilon(x)\xi\| < \e$ for
each $x \in Q$, then there is a non-zero vector $\xi'$ with $\phi(y)\xi'=
\hh\epsilon(y)\xi'$ for all $y\in \C_0^u(\G)$.
\end{enumerate}
For a Kazhdan pair of any type $(Q,\e)$, a vector $\xi \in \H$ that satisfies the inequalities in the respective definition above is said to be $(Q,\e)$-\emph{invariant}.
\end{deft}

\begin{tw}\label{thm:kazpair}
$\G$ has Property (T) if and only if $\G$ has a Kazhdan pair of any
(or equivalently, all) type(s).
\end{tw}
\begin{proof}
By the various equivalent notions of having almost-invariant vectors in \prettyref{lem:almost_inv_vec}, it is clear that if $\G$ has a Kazhdan pair
of any type, then it has Property (T).

Conversely, suppose $\G$ does not have a Kazhdan pair of type 2, but that
$\G$ has Property (T).  Denote by $\mc F$ the family of all finite subsets of the
unit ball of $\C_0(\G)$.  For each $F\in\mc F$ and $\e>0$, we can find a
representation $\mc U_{F,\e}$ with no non-zero invariant vector, but
which admits a unit vector $\xi_{F,\e}$ with
$\| \mc U_{F,\e}(a\otimes \xi_{F,\e})
- a\otimes \xi_{F,\e} \| <\e$ for $a\in F$.  Ordering
$\mc F\times(0,\infty)$ in the obvious way, we find that
\[ \lim_{(F,\e)} \big\| \mc U_{F,\e} (a\otimes\xi_{F,\e})
- a\otimes\xi_{F,\e} \big\| = 0. \]
It follows easily that $\mc U = \bigoplus \mc U_{F,\e}$ has almost
invariant vectors, and hence has a non-zero invariant vector, say
$\eta=(\eta_{F,\e})$.  However, then if $\eta_{F,\e}\not=0$,
then $\eta_{F,\e}$ is a non-zero invariant vector for
$\mc U_{F,\e}$, a contradiction.

Exactly the same argument applies to Kazhdan pairs of type 1; and a very
similar argument applies to Kazhdan pairs of type 3.
\end{proof}

The following is valid for a Kazhdan pair of any type; we prove the
version most suited to our needs.  If we fix a representation $\phi:\C_0^u(\hh\G)\rightarrow\mc B(\Hil)$, we will define the space $\inv(U)\subseteq \sH$ via the associated
representation $U=(\id\otimes\phi)(\wW)$.  By \prettyref{lem:inv_vec},
$\xi\in\inv(U)$ if and only if $\phi(x)\xi = \hh\epsilon(x)\xi$ for all
$x\in \C_0^u(\hh\G)$.

\begin{lem}\label{lem:cty_consts}
Let $(Q,\e)$ be a Kazhdan pair of type 3 for $\G$, and let $\delta>0$.
For any representation $\phi$ of $\C_0^u(\hh\G)$ on $\sH$, if $\xi\in \H$ is
$(Q,\e\delta)$-invariant, then $\| \xi-P\xi \| \leq \delta\|\xi\|$
where $P:\sH\rightarrow\inv(U)$ is the orthogonal projection.
\end{lem}
\begin{proof}
Let $\alpha\in\inv(U)^\perp$ and $\beta\in\inv(U)$.  For $x\in \C_0^u(\hh\G)$,
\[ \la \beta, \phi(x)\alpha \ra = \la \phi(x^*)\beta, \alpha \ra
= \overline{ \hh\epsilon(x^*) } \la \beta, \alpha \ra = 0, \]
so $\phi(x)\alpha \in \inv(U)^\perp$.  Thus $\phi$ restricts to the subspace
$\inv(U)^\perp$.

Let $\xi = P(\xi) + \xi_0$ where $\xi_0\in\inv(U)^\perp$.
As $\inv(U)^\perp$ contains no non-zero invariant vectors, there is
$x\in Q$ with $\| \phi(x)\xi_0 - \hh\epsilon(x)\xi_0 \|
\geq \e \|\xi_0\|$.  By assumption,
$\| \phi(x)\xi - \hh\epsilon(x)\xi \| \leq \e\delta\|\xi\|$, and as
$P(\xi)$ is invariant, $\phi(x)P(\xi) = \hh\epsilon(x)P(\xi)$.
Thus
\begin{align*}
\| \phi(x)\xi_0 - \hh\epsilon(x)\xi_0 \|
&= \| \phi(x)\xi_0 - \phi(x)\xi + \phi(x)\xi - \hh\epsilon(x)\xi_0 \|
= \| -\phi(x)P(\xi) + \phi(x)\xi - \hh\epsilon(x)\xi_0 \| \\
&= \| -\hh\epsilon(x)P(\xi) + \phi(x)\xi - \hh\epsilon(x)\xi_0 \|
= \| \phi(x)\xi - \hh\epsilon(x)\xi \| \leq \e\delta\|\xi\|.
\end{align*}
We conclude that $\e \|\xi_0\| \leq \e\delta\|\xi\|$, showing
that $\|\xi-P\xi\| = \|\xi_0\| \leq \delta\|\xi\|$ as claimed.
\end{proof}

\subsection{Hereditary property}

Classically if a locally compact group $G$ has Property (T) and $K$ is a closed normal subgroup of $G$ then $G/K$ has Property (T).  In fact even if there is only a morphism $G\rightarrow H$ with dense range, where $H$ is another locally compact group, then Property (T) passes from $G$ to $H$.

In \cite{Chen_Ng__prop_T_LCQGs} it is shown under some additional assumptions that the existence of a surjective $*$-homomorphism from $\C_0^u(\hQG)$ to $\C_0^u(\hh\QH)$ implies that Property (T) passes from $\QG$ to $\QH$. Here we strengthen this result, at the same time dropping the extra assumptions.

\begin{tw}\label{thm:passes_to_dense_range}
Let $\QH, \G$ be locally compact quantum groups. If there is a morphism from $\QH$ to $\QG$ with dense image, and $\QH$ has Property (T), then so does $\G$.
\end{tw}
\begin{proof}
Let the relevant morphism be given by $\pi \in \Mor(\C_0^u(\G), \C_0^u(\QH))$.
Let $U$ be a representation of $\G$ on a Hilbert space $\sH$ which has almost-invariant
vectors, $(\xi_i)_{i \in \Ind}$.  They are also almost invariant for  $U^u$ by Proposition \ref{Prop:invUniv}.
Then $V:=(\Lambda_\QH \circ \pi\otimes\id)(U^u)$ is a representation of $\QH$.
Let $\omega\in \Lone{\QH}$, and note that
\begin{align*}
\lim_{i\in \Ind} \omega((\id\otimes\omega_{\xi_i})(V))=
 \lim_{i\in \Ind} \left \langle \xi_i,  (\omega\circ\Lambda_{\QH}\circ\pi\otimes\id)(U^u)\xi_i \right \rangle
= \omega(\one)
\end{align*}
by the definition of what it means for $(\xi_i)_{i\in \Ind}$ to be invariant for
$U^u$.  Hence $(\id\otimes\omega_{\xi_i})(V) \xrightarrow{i\in\Ind} \one$ in the weak$^*$ topology of
$L^\infty(\QH)$, so by \prettyref{lem:almost_inv_vec} the net $(\xi_i)_{i\in \Ind}$ is
almost invariant for $V$.  As $\QH$ has Property (T), $V$ has a non-zero invariant vector,
say $\xi_0\in \Hil$. Thus, \prettyref{lem:inv_vec} implies that
$\phi_V(a)\xi_0 = \hh\epsilon(a)\xi_0$ for all $a\in \C_0^u(\hh\QH)$.

Let $\omega\in \Lone{\QH}$ and $y=(\omega\otimes\id)(\wW_{\QH})$. Then using \eqref{eq:dual_morphism} we see that
\begin{align*} \phi_U(\hh\pi(y)) &=
\phi_U\big( (\omega\otimes\hh\pi)(\wW_{\QH}) \big)
= \phi_U\big( (\omega\Lambda_\QH\pi\otimes\id)(\WW_{\G}) \big)
= (\omega\Lambda_\QH\pi\otimes\id)(U^u)
= (\omega\otimes\id)(V) \\
&= \phi_V\big( (\omega\otimes\id)(\wW_\QH) \big) = \phi_V(y).
\end{align*}
By density, this holds for all $y\in \C_0^u(\hh\QH)$. The same calculation yields that $\hh\epsilon_{\G}\circ\hh\pi = \hh\epsilon_{\QH}$.
Hence $\phi_U(\hh\pi(y))\xi_0 = \phi_V(y) \xi_0 = \hh\epsilon_{\QH}(y)\xi_0 = \hh\epsilon_{\G}(\hh\pi(y))\xi_0
$ for all $y\in \C_0^u(\hh\QH)$. By condition \ref{sp:four} in Theorem \ref{thm:dense_image_char},
$\hh\pi$ has strictly dense range. In addition,  $\phi_U:\C_0^u(\hh\G)\rightarrow\mc B(\sH)=\M(\mc K(\sH))$ is strictly continuous,
hence strictly-strongly continuous. It follows that
$\phi_U(x) \xi_0 = \hh\epsilon_\G(x)\xi_0$ for all $x\in \M(\C_0^u(\hh\G))$.
In particular, $\xi_0$ is invariant for $U$, as required.
\end{proof}

\section{Applications of Kazhdan pairs}\label{sec:Kazhdan_pairs_apps}

The main aim of this section is to apply the results on Kazhdan pairs to give an alternative proof and strengthen in a crucial way Theorem 3.1 of \cite{Kyed__cohom_prop_T_QG}. In connection with the recent results of \cite{Arano__unit_sph_rep_Drinfeld_dbls} this will enable us in the next section to establish equivalence of Property (T) and Property (T)$^{1,1}$ for a class of discrete quantum groups.

The equivalence of \prettyref{enu:wstar_norm_one} and \prettyref{enu:wstar_norm_two} of the next result was established in \cite[Theorem 3.1]{Kyed__cohom_prop_T_QG} under the additional assumption that $\G$ is a second countable discrete quantum group. However, the following proof, and especially \ref{enu:wstar_norm_three}$\implies$\ref{enu:wstar_norm_one}, is quite different: not having the particular structure of discrete quantum groups at our disposal, we use the averaging semigroups machinery of \prettyref{sec:inv_alm_inv_vects}. Note that outside of the Property (T) context \cite[Theorem 4.6]{Runde_Viselter_LCQGs_PosDef} provides an equivalence between various modes of convergence of normalised positive-definite elements on locally compact quantum groups.

\begin{tw}\label{thm:wstar_norm}
Let $\G$ be a locally compact quantum group.  The following are equivalent:
\begin{enumerate}
\item\label{enu:wstar_norm_one}
$\G$ has Property (T);
\item\label{enu:wstar_norm_two}
if $(\mu_i)_{i \in \Ind}$ is a net of states of $\C_0^u(\hh\G)$ converging
in the weak$^*$ topology to the counit $\hh\epsilon$, then $(\mu_i)_{i \in \Ind}$ converges in norm
to $\hh\epsilon$;
\item\label{enu:wstar_norm_three}
if $(a_i)_{i \in \Ind}$ is a net of normalised positive-definite functions in $\M(\C_0(\G))$, which converges strictly to $\one$, then
$\lim_{i \in \Ind} \| a_i- \one \| =0$.
\end{enumerate}
\end{tw}
\begin{proof}
\ref{enu:wstar_norm_one}$\implies$\ref{enu:wstar_norm_two}:
If $\G$ has Property (T) then for each $i \in \Ind$ let $(\phi_i,\sH_i,\xi_i)$
be the GNS construction for $\mu_i$.  Then let $\phi:=\bigoplus_{i \in \Ind}\phi_i$,
a representation of $\C_0^u(\hh\G)$ on $\sH=\bigoplus_{i\in \Ind}\sH_i$.
By a slight abuse of notation, we treat each $\xi_i$ as a vector in $\sH$.
Then for $a\in \C_0^u(\hh\G)$,
\begin{align*} \lim_{i\in \Ind} \| \phi(a)\xi_i - \hh\epsilon(a)\xi_i \|^2
&= \lim_{i\in \Ind} \|\phi_i(a)\xi_i - \hh\epsilon(a)\xi_i \|^2 \\
&= \lim_{i\in \Ind} \big [ \la \xi_i,\phi_i(a^*a)\xi_i \ra + \hh\epsilon(a^*a)
   - 2\Ree \big \la \hh\epsilon(a)\xi_i, \phi_i(a)\xi_i \big \ra \big ] \\
&= \lim_{i\in \Ind} \big [\mu_i(a^*a) + \hh\epsilon(a^*a)
   - 2\Ree\big( \hh\epsilon(a^*) \mu_i(a)\big) \big ] \\
&= 2\hh\epsilon(a^*a) - 2\Ree \hh\epsilon(a^*a) = 0.
\end{align*}
Let $(Q,\e)$ be a Kazhdan pair of type 3 for $\G$.  For $\delta>0$, if
$i\in \Ind$ sufficiently large, $\xi_i$ is $(Q,\delta\e)$-invariant.
By Lemma \ref{lem:cty_consts} it follows that $\| \xi_i - P(\xi_i) \| \leq
\delta$, and $\eta_i=P(\xi_i)$ is invariant.
Then, for $a\in \C_0^u(\hh\G)$,
\begin{align*} |\mu_i(a) - \hh\epsilon(a)| =
\left| \left\langle \xi_i, \phi(a)\xi_i - \hh\epsilon(a)\xi_i  \right\rangle \right|
= \left| \left\langle \xi_i - \eta_i, \phi(a)\xi_i \right\rangle
- \left\langle \xi_i-\eta_i, \hh\epsilon(a) \xi_i \right\rangle \right|,
\end{align*}
which follows as $\phi(a)^*\eta_i = \phi(a^*)\eta_i =
\hh\epsilon(a^*)\eta_i$.  Thus
\[ |\mu_i(a) - \hh\epsilon(a)| \leq 2\delta\|a\|, \]
so
$\mu_i \stackrel{i \in \Ind}{\longrightarrow} \hh\epsilon$ in norm.

\ref{enu:wstar_norm_two}$\implies$\ref{enu:wstar_norm_three}:
Let $(a_i)_{i\in \Ind}$ be as in the hypotheses, so for each $i\in \Ind$ there is a
state $\mu_i$ on $\C_0^u(\hh\G)$ with $a_i = (\id\otimes\mu_i)
(\wW^*)$.  Let $\omega\in \Lone{\QG}$ and set $a=(\omega\otimes\id)(\wW^*)$;
we note that $a\in \C_0^u(\hh\G)$ and that such elements are norm dense in $\C_0^u(\hh\G)$.  By
Cohen's factorisation theorem, there are elements $\omega'\in \Lone{\QG}$ and $x\in \C_0(\G)$ with
$\omega = x\omega'$.  Thus
\[ \lim_{i\in \Ind} \mu_i(a) = \lim_{i\in \Ind} \omega(a_i)
=  \lim_{i\in \Ind} \omega'(a_i x) = \omega'(x) = \omega(\one)
= \omega((\id\otimes\hh\epsilon)(\wW^*))
= \hh\epsilon(a). \]
As $\{\mu_i:i\in \Ind\}$ is a bounded net, this shows that $\mu_i\rightarrow
\hh\epsilon$ in the weak$^*$ topology, and hence by hypothesis, $\mu_i\rightarrow
\hh\epsilon$ in norm.  It follows immediately that $a_i \rightarrow \one$
in norm, as claimed.

\ref{enu:wstar_norm_three}$\implies$\ref{enu:wstar_norm_one}:
Let $U$ be a representation of $\G$ on $\sH$ which has almost-invariant vectors.
Let $\mc U\in\mc L(\C_0(\G)\otimes \sH)$ be the associated adjointable operator.  By \prettyref{lem:almost_inv_vec}, there is a net $(\xi_i)_{i\in \Ind}$ of unit vectors with
\[ \lim_{i\in \Ind} \big\| \mc U(a\otimes\xi_i) - a\otimes\xi_i
\big\| = 0 \qquad \forall_{a\in \C_0(\G)}. \]
For $\xi\in \sH$ let $(\id\otimes\xi^*):C_0(\G)\otimes \sH\rightarrow \sH$
be the adjointable map defined on elementary tensors by $(\id\otimes\xi^*):
a\otimes\eta \mapsto \la \xi, \eta \ra a$.  As each $\xi_i$ is a unit vector and $\mc U$
is unitary, it follows that
\[ \lim_{i\in \Ind} \big\| (\id\otimes\xi_i^*)\mc U^*(a\otimes\xi_i)
- a \big\| = 0 \qquad \forall_{a\in \C_0(\G)}. \]
Now, the relation between $U$ and $\mc U$ implies that
$(\id\otimes\xi_i^*)\mc U^*(a\otimes\xi_i) =
(\id\otimes\omega_{\xi_i})(U^*) a$.  Thus, if $a_i :=
(\id\otimes\omega_{\xi_i})(U^*)$, then $a_i$ is a normalised positive-definite element, and the net $(a_i)_{i\in \Ind}$ is a left approximate identity for $\C_0(\G)$.
By repeating the same argument with $\mc U$ instead of  $\mc U^*$ and then taking adjoints, we deduce that $(a_i)_{i\in \Ind}$ is also a right approximate identity for $\C_0(\G)$.  Thus
by hypothesis, $a_i\rightarrow \one$ in norm.

For each $i \in \Ind$ let $\eta_i$ be the orthogonal projection of
$\xi_i$ onto $\inv(U)$.  By combining Lemma~\ref{lem:avg_semigp_two}
with Proposition~\ref{prop:ave_equals_proj_gen}, we see that $\eta_i$ is
the unique vector of minimal norm in the closed convex hull of
\[ C = \{ (\omega\otimes\id)(U^*)\xi_i : \omega\in \Lone{\QG}
\text{ is a state} \}. \]
For $\e>0$ choose $i\in \Ind$ sufficiently large so that
$\|a_i-\one\| < \e$.  Then, for $\omega$ a state in $\Lone{\G}$, we see that as
$\xi_i$ is a unit vector,
\[ \big| \la \xi_i, \xi_i - (\omega\otimes\id)(U^*)\xi_i \ra \big|
= |\ip{\one-a_i}{\omega}| < \e. \]
Thus
\begin{align*} \big\| (\omega\otimes\id)(U^*)\xi_i - \xi_i \big\|^2
&= \big\| (\omega\otimes\id)(U^*)\xi_i \big\|^2 + 1
- 2\Ree\la\xi_i, (\omega\otimes\id)(U^*)\xi_i \ra \\
&\leq 2 \big( 1-\Ree\la \xi_i, (\omega\otimes\id)(U^*)\xi_i \ra \big) \\
&\leq 2 \big| \la \xi_i, \xi_i - (\omega\otimes\id)(U^*)\xi_i \ra \big|
< 2\e.
\end{align*}
So every vector in $C$ is at most $\sqrt{2\e}$ from $\xi_i$,
and hence the same is true of $\eta_i$.  In particular, if $\e<1/2$
then $\eta_i$ is non-zero, and invariant, showing that $U$ has non-zero
invariant vectors, as required.
\end{proof}

We will now show that in fact we can have better control over the elements appearing in the above theorem.

\begin{lem}\label{real}
Let $(a_i)_{i\in \Ind}$ be a net of Hilbert space contractions which does not converge to $\one$. Then $(\Ree (a_i))_{i\in \Ind}$ also does not converge to $\one$.
\end{lem}

\begin{proof}
It is enough to show the statement for sequences, as its sequential version is easily seen to be equivalent to the following fact:
\[ \forall_{\e>0} \; \exists_{\delta>0} \; \forall_{a\in \B(\Hil)} \; (\|a\|\leq 1, \|a-\one\|>\e) \Longrightarrow  \|\Ree (a)-\one\|>\delta.\]
So suppose then that the sequential statement fails, so that we have $\e >0$ and $(a_n)_{n=1}^{\infty}$ a sequence of contractions such that $\|a_n-\one\|>\e$ and $\|\Ree(a_n)-\one\| \leq \frac{1}{n}$ for all $n \in \bn$.  Then
\[\e < \|a_n-\one\| =  \|\Ree(a_n)-\one + i \tu{Im} (a_n) \|\leq \frac{1}{n} + \|\tu{Im} (a_n)\| \]
so for $n$ big enough $\|\tu{Im} (a_n)\|> \frac{\e}{2}$. Thus there are vectors $\xi_n \in \Hil$ of norm 1 such that $|\ip{\xi_n}{\tu{Im} (a_n) \xi_n}|>\frac{\e}{2}$.
But then
\[ 1 \geq|\ip{\xi_n}{a_n \xi_n}| = | \ip{\xi_n}{\xi_n} + \ip{\xi_n}{(\Ree (a_n)-\one) \xi_n} + i \ip{\xi_n}{\tu{Im} (a_n) \xi_n}| = |1+\alpha_n + i\beta_n|,\]
where $\alpha_n, \beta_n$ are real numbers, $|\alpha_n|\leq \frac{1}{n}$, $|\beta_n|>\frac{\e}{2}$. This is clearly contradictory and the proof is finished.
\end{proof}

Of course for normal operators it would have been enough to use the functional calculus. This is how one can prove in elementary fashion the next lemma.

\begin{lem} \label{positive}
Let $(a_i)_{i\in \Ind}$ be a net of selfadjoint Hilbert space contractions which does not converge to $\one$. Then the net (of positive operators)
$(\exp(a_i-\one))_{i\in \Ind}$ does not converge to $\one$.
\end{lem}
\begin{proof}
Via the spectral theorem it suffices to note the following elementary fact:
\[ \forall_{\e>0} \; \exists_{\delta>0} \; \forall_{a\in \br} \; (|a|\leq 1, |a-1|>\e) \Longrightarrow  |\exp(a-1) - 1|>\delta. \qedhere \]
\end{proof}

We are now able to strengthen Theorem \ref{thm:wstar_norm} under the assumption of triviality of the scaling group.

\begin{prop} \label{proppositiveuniform}
Let $\QG$ be a locally compact quantum group with trivial scaling group. Then the following conditions are equivalent:
\begin{enumerate}
\item \label{enu:proppositiveuniform_1} $\QG$ has Property (T);
\item \label{enu:proppositiveuniform_2} if $(a_i)_{i \in \Ind}$ is a net of normalised positive-definite elements in $\M(\C_0(\QG))_+$ which converges strictly to $\one$, then
$\lim_{i \in \Ind} \| a_i - \one \| =0$.
\end{enumerate}
\end{prop}

\begin{proof}
The implication \prettyref{enu:proppositiveuniform_1}$\Longrightarrow$\prettyref{enu:proppositiveuniform_2} is contained in Theorem \ref{thm:wstar_norm}.

For the converse direction we begin by observing the following: if $a$ is a normalised positive-definite element of $\M(\C_0(\QG))$, then so is $a^*$ -- this follows from the fact that if $\mu$ is a state of $\C_0^u(\hQG)$, and $\hQG$ has a bounded antipode (as we assume here), then so is $\mu \circ \widehat{S}^{u}$. Further also $\Ree(a)$ is a normalised positive-definite element, as a convex combination of elements of this type. Additionally, if $a\in \M(\C_0(\QG))$ is a normalised positive-definite element, then so is $\exp(a-\one)$ -- this time we use \cite[Theorem 6.3]{Lindsay_Skalski__quant_stoch_conv_cocyc_3} and \cite[Theorem 3.7]{Lindsay_Skalski__conv_semigrp_states}, which imply that if $\mu$ is a state of $\C_0^u(\hQG)$, then so is $\exp_{\conv}(\mu - \widehat{\epsilon}) = \sum_{k=0}^{\infty} \frac{(\mu-\widehat{\epsilon})^{\conv k}}{k!}$.

Assume then that $\QG$ does not have Property (T). By \prettyref{thm:wstar_norm} there exists a net $(a_i)_{i \in \Ind}$ of normalised positive-definite  elements of $\M(\C_0(\QG))$ which converges strictly to $\one$, but does not converge to $\one$ in norm. By the considerations above and Lemmas \ref{real} and \ref{positive} the net
$(\exp(\Ree(a_i) - \one))_{i \in \Ind}$ is a net  of  normalised positive-definite elements of $\M(\C_0(\QG))_+$ which does not converge to $\one$ in norm. It is elementary to verify that it converges to $\one$ strictly.
\end{proof}

The following definition appears in  \cite{Arano__unit_sph_rep_Drinfeld_dbls} (for discrete quantum groups). Recall that a functional in $\C_0^u(\hh\G)^*$ is said to be \emph{central} if it commutes with all elements of $\C_0^u(\hh\G)^*$ with respect to the convolution product.

\begin{deft}
A locally compact quantum group $\G$ has \emph{Central Property (T)} if for any  net $(\mu_i)_{i \in \Ind}$ of central states of $\C_0^u(\hh\G)$ converging
in the weak$^*$ topology to the counit $\hh\epsilon$, the net  $(\mu_i)_{i \in \Ind}$ actually converges in norm
to $\hh\epsilon$.
\end{deft}

It is clear from Theorem \ref{thm:wstar_norm} that Property (T) implies Central Property (T). The following result of Arano establishes the converse implication for discrete unimodular quantum groups.

\begin{tw}[{\cite[Proposition A.7]{Arano__unit_sph_rep_Drinfeld_dbls}}]
A discrete unimodular quantum group has Property (T) if and only if it has Central Property (T).
\end{tw}

Note that for a discrete $\QG$ a state  $\mu \in \C^u(\hQG)^*$ is central if and only if all the matrices $\mu^{\alpha}$ ($\alpha \in \Irred{\hh\QG}$ -- see comments after Remark \prettyref{Remark:low}) are scalar multiples of the identity, if and only if the corresponding normalised positive-definite element (which in this case can be identified simply with $(\mu^\alpha)_{\alpha \in \Irred{\QG}} \in \M (\c0(\QG))$) is in the centre of $\M (\c0(\QG))$. Using this fact we can easily show, following the lines above, the next corollary.

\begin{cor} \label{cor:centralproppositiveuniform}
Let $\QG$ be a unimodular discrete quantum group. Then the following conditions are equivalent:
\begin{enumerate}
\item $\QG$ has Property (T);
\item if $(a_i)_{i \in \Ind}$ is a net of central normalised positive-definite elements in $\M(\c0(\QG))_+$ which converges strictly to $\one$, then
$\lim_{i \in \Ind} \| a_i - \one \| =0$.
\end{enumerate}
\end{cor}

The following result now extends \cite[Theorem 5.1]{Kyed__cohom_prop_T_QG} by adding to it another equivalent condition, namely \prettyref{enu:UnboundedCentral_3} below.

\begin{tw}\label{thm:UnboundedCentral}
Let $\QG$ be a second countable discrete quantum group. The following are equivalent:
\begin{enumerate}
\item \label{enu:UnboundedCentral_1} $\QG$ has Property (T);
\item \label{enu:UnboundedCentral_2} $\QG$ is unimodular and $\Pol(\widehat{\QG})$ admits no unbounded generating functional;
\item \label{enu:UnboundedCentral_3} $\QG$ is unimodular and $\Pol(\widehat{\QG})$ admits no central strongly unbounded $\widehat{S}^u$-invariant generating functional.
\end{enumerate}
\end{tw}

\begin{proof}
\prettyref{enu:UnboundedCentral_1}$\implies$\prettyref{enu:UnboundedCentral_2}:
this is the implication  \prettyref{enu:UnboundedCentral_1}$\implies$\prettyref{enu:UnboundedCentral_2} in \cite[Theorem 5.1]{Kyed__cohom_prop_T_QG}; the first part dates back to \cite{Fima__prop_T}.

\prettyref{enu:UnboundedCentral_2}$\implies$\prettyref{enu:UnboundedCentral_3}: trivial.

\prettyref{enu:UnboundedCentral_3}$\implies$\prettyref{enu:UnboundedCentral_1}:
suppose that \prettyref{enu:UnboundedCentral_3} holds and yet (T) fails.  Let $(K_n)_{n=1}^\infty$ be an increasing family of finite subsets of $\Irred{\hQG}$  such that $\bigcup_{n=1}^{\infty} K_n =  \Irred{\hQG}$. We will argue as in the proof of the implication (a4)$\Longrightarrow$(b3) in \cite{Jolissaint__prop_T_pairs}, attributed there to \cite{Akemann_Walter__unb_neg_def_func}.

\prettyref{cor:centralproppositiveuniform} implies the existence of a net of central normalised positive-definite contractions $(a_i)_{i\in \Ind}$ in $\M(\c0(\QG))_+$ which is strictly convergent to $\one$, but not norm convergent to $\one$. As  $\Irred{\hQG}$ is assumed to be countable, we can replace this net by a sequence with the same properties. Consider the corresponding positive scalar matrices $a_k^{\alpha}\in M_{n_{\alpha}}$. There exists $\epsilon>0$ such that for each $k\in \bn$ there exists $l\geq k$ and  $\alpha_l \in\Irred{\hQG}$ such that $\|I_{n_{\alpha_l}} - a_l^{\alpha_l}\| \geq \epsilon$. Similarly for each $p\in \bn$ there exists $N\in \bn$ such that for all $k\geq N$ we have $\sup_{\alpha \in K_p} \|I_{n_{\alpha}} - a_k^{\alpha}\| \leq \frac{\epsilon}{4^p}$. Using these two facts we can, by passing to a subsequence constructed inductively, assume that in fact for each $l \in \bn$ there exists $\alpha_l\in \Irred{\hQG}$ such that
\begin{equation} \|I_{n_{\alpha_l}} - a_l^{\alpha_l}\| \geq \epsilon, \label{pointbig}\end{equation}
\begin{equation}\sup_{\alpha \in K_l} \|I_{n_{\alpha}} - a_l^{\alpha}\| \leq \frac{\epsilon}{4^l}. \label{small}\end{equation}
Define for each $\alpha \in \Irred{\hh\QG}$
\[L^{\alpha} = \sum_{l=1}^{\infty} 2^l (I_{n_{\alpha}} - a_l^{\alpha}).\]
Condition \eqref{small} assures that each $L_{\alpha}$ is well defined. For each $l \in \bn$ we have
$L^{\alpha_l} \geq 2^l (I_{n_{\alpha_l}} - a_l^{\alpha_l}) \geq 0$,
and thus condition \eqref{pointbig} implies that
$\|L^{\alpha_l}\|\geq 2^l\epsilon $. Finally it is easy to verify that the functional $L$ induced on $\Pol(\hQG)$ via matrices $L^{\alpha}$ is a generating functional, as it is defined via (pointwise convergent) series
\[ L=\sum_{l=1}^{\infty} 2^l (\widehat{\epsilon} - \mu_l),\]
where $\mu_l$ denotes the state of $\Pol(\hQG)$ corresponding to $a_l$, which satisfies $\mu_l \circ \widehat{S}^u = \mu_l$. This ends the proof, as $L$ is clearly strongly unbounded, central and $\widehat{S}^u$-invariant.
\end{proof}

\section{Implication (T)$^{1,1}\implies$(T) for discrete unimodular quantum groups with low duals}  \label{sec:TT11}

In this section we establish one of the main results of the whole paper, that is we show that for discrete unimodular  quantum groups with low duals Property (T) is equivalent to Property (T)$^{1,1}$ (and in fact to Property (T)$^{r,s}$ for any $r, s \in \bn$). The key tool will be \prettyref{thm:UnboundedCentral}, and we  begin with a series of technical lemmas regarding the behaviour of convolution semigroups of states at infinity implied by the properties of their generators.

\begin{lem}\label{KeyLemma0}
Let $\QG$ be a  discrete unimodular quantum group and let $L$ be a central strongly unbounded $\widehat{S}^u$-invariant generating functional on $\Pol(\hQG)$ (so that we have $L^{\alpha} = c_{\alpha} I_{n_{\alpha}} $ for all $\alpha \in \Irred{\hQG}$, where $c _{\alpha} \in \br_+$) and suppose that $(\gamma_l)_{l \in \bn}$ is a sequence of elements of $\Irred{\hQG}$ such that  we have $L^{\gamma_l}=d_l I_{M_{n_{\gamma_l}}}$ ($l \in \bn$), with $(d_l)_{l=1}^{\infty}$ a sequence of positive numbers increasing to infinity. Let $\alpha, \beta \in \Irred{\hQG}$ and define for each $l \in \bn$ the matrix $V^{(l)} \in M_{n_{\alpha}} \ot M_{n_{\gamma_l}} \ot M_{n_{\beta}}$ by the formula
\[ V^{(l)}_{(i,j,k),(p,r,s)} = L\left((u_{ip}^{\alpha})^*  u_{jr}^{\gamma_l} u_{ks}^{\beta}\right),\]
$i,p =1,\ldots, n_{\alpha},  j,r =1,\ldots, n_{\gamma_l},   k,s =1,\ldots, n_{\beta}$.
Then all matrices $V^{(l)}$ are selfadjoint, and there exists a sequence $(e_l)_{l=1}^{\infty}$ of positive numbers converging to infinity such that for each $l \in \bn$
\[ \sigma(V^{(l)}) \subseteq \{\lambda \in \bc: \Ree \lambda \geq e_l\}.\]
\end{lem}

\begin{proof}
The matrices $V^{(l)}$ are selfadjoint, because $L$ being $\widehat{S}^u$-invariant and $\hQG$ being of Kac type imply that for all suitable $i,j,k,p,r,s$,
\[
\begin{split}(V^{(l)*})_{(i,j,k),(p,r,s)} & =\overline{(V_{(p,r,s),(i,j,k)}^{(l)})}=\overline{L((u_{pi}^{\a})^{*}u_{rj}^{\gamma_{l}}u_{sk}^{\be})}=L((u_{sk}^{\be})^{*}(u_{rj}^{\gamma_{l}})^{*}u_{pi}^{\a})\\
 & =(L\circ\widehat{S}^{u})((u_{sk}^{\be})^{*}(u_{rj}^{\gamma_{l}})^{*}u_{pi}^{\a})=L((u_{ip}^{\a})^{*}u_{jr}^{\gamma_{l}}u_{ks}^{\be})=V_{(i,j,k),(p,r,s)}^{(l)}.
\end{split}
\]

Recall that if $\rho$ is a representation of $\C^u(\hQG)$ on a Hilbert space $\Hil$, a linear map $c : \Pol(\hQG) \to \Hil$ is a \emph{cocycle} for $\rho$ if it is a $\rho-\widehat{\epsilon}$-derivation, i.e., for all $a,b\in\Pol(\hQG)$,
\[ c(ab) = \rho(a)c(b) + c(a)\widehat{\epsilon}(b). \]
The `conditional' GNS construction for the generating functional $L$ (originally due to Sch\"urmann, we refer for example to  \cite[Subsection 7.2]{Daws_Fima_Skalski_White_Haagerup_LCQG} for details) yields the existence of a Hilbert space $\Hil$, a  representation $\rho: \C^u(\hQG) \to \B(\Hil)$, and a cocycle  $c: \Pol(\hQG) \to \Hil$ for $\rho$ such that for all $a, b \in \Pol(\hQG)$ we have
\[L(a^*b) = \ol{L(a)} \hCou(b) + \ol{\hCou(a)} L(b) - \la c(a), c(b) \ra. \]
(Remark that the scalars $c_l$ from the lemma's statement are indeed non-negative since $c$ is \emph{real} as $L$ is $\widehat{S}^u$-invariant, see \cite[Proposition 7.21, (7.10), and the succeeding paragraph]{Daws_Fima_Skalski_White_Haagerup_LCQG}.)
Expanding this equality and using the $\rho-\widehat{\epsilon}$ derivation property yields for all $a, b, d \in \Pol(\hQG)$
\begin{align*} L(a^*bd) &= \ol{L(a)} \hCou(bd) + \ol{\hCou(a)} L(bd) - \la c(a), c(bd) \ra \\&=
\ol{L(a)} \hCou(bd) + \ol{\hCou(a)} \left(  \ol{L(b^*)} \hCou(d) + \ol{\hCou(b^*)} L(d) - \la c(b^*), c(d) \ra \right)  - \la c(a), \rho(b) c(d) \ra - \hCou(d) \la  c(a), c(b) \ra
\\&= L(a^*) \hCou(bd) + L (b) \hCou(a^*d) + L(d) \hCou(a^*b) - \hCou(a^*) \la c(b^*), c(d) \ra - \la c(a), \rho(b) c(d) \ra - \hCou(d) \la  c(a), c(b) \ra.
\end{align*}
Using now the fact that $L$ is central and properties of the counit we obtain
\begin{multline*}  V^{(l)}_{(i,j,k),(p,r,s)} =
 \\ \qquad=\delta_{i,p} \delta_{j,r} \delta_{k,s} (c_{\alpha} + c_{l} + c_{\beta}) - \delta_{i,p} \la c((u_{jr}^{\gamma_l})^*), c(u_{ks}^{\beta}) \ra  - \delta_{k,s} \la c(u_{ip}^{\alpha}), c(u_{jr}^{\gamma_l})\ra - \la c(u_{ip}^{\alpha}), \rho(u_{jr}^{\gamma_l}) c(u_{ks}^{\beta})\ra.
\end{multline*}
The trick is now based on expressing the last equality in a matrix form. We will do it step by step. Define the following Hilbert space operators: $T_l, \wt{T_l}: \bc^{n_{\gamma_l}}\to \Hil \ot \bc^{n_{\gamma_l}}$, $\wt{T_\alpha}: \bc^{n_{\alpha}}\to \Hil \ot \bc^{n_{\alpha}}$, and $T_\beta: \bc^{n_{\beta}}\to \Hil \ot \bc^{n_{\beta}}$:
\[ T_l (e_j) = \sum_{a=1}^{n_{\gamma_l}} c(u^{\gamma_l}_{ja}) \ot e_a,\]
\[ \wt{T_l} (e_j) = \sum_{a=1}^{n_{\gamma_l}} c((u^{\gamma_l}_{aj})^*) \ot e_a,\]
\[ \wt{T_{\alpha}} (e_i) = \sum_{b=1}^{n_{\alpha}} c(u^{\alpha}_{bi}) \ot e_b,\]
\[ T_\beta (e_k) = \sum_{b=1}^{n_{\beta}} c(u^{\beta}_{kb}) \ot e_b,\]
for $j \in \{1, \ldots, n_{\gamma_l}\}$, $i \in \{1, \ldots, n_{\alpha}\}$, $k \in \{1, \ldots, n_{\beta}\}$,
where $(e_j)_{j=1}^{n_{\gamma_l}}$, $(e_i)_{i=1}^{n_{\alpha}}$ and $(e_k)_{k=1}^{n_{\beta}}$ are orthonormal bases of the respective spaces.

To simplify the notation we will now write $I_{\alpha}$ for $I_{M_{n_{\alpha}}}$, $I_l$ for $I_{M_{n_{\gamma_l}}}$ and $I_{\beta}$ for $I_{M_{n_{\beta}}}$ and denote the Hilbert space tensor flips by $\Sigma$ (with the position of the flipped arguments understood from the context).
Then we claim the following:
\begin{align*} V^{(l)} =& (c_{\alpha} + c_{l} + c_{\beta}) I_{\alpha} \ot I_l \ot I_{\beta} -
I_{\alpha} \ot \left((\wt{T_l}^* \ot I_{\beta}) \Sigma (I_l \ot T_{\beta})\right)^{\mathrm{t}}
\\& - \left((\wt{T_{\alpha}}^* \ot I_{l}) \Sigma (I_{\alpha} \ot T_l)\right)^{\mathrm{t}} \ot I_{\beta} -
\left[ (\wt{T_{\alpha}}^* \ot I_{l} \ot I_{\beta}) \Sigma (I_{\alpha} \ot \rho^{(n_{\gamma_l})}(u^{\gamma_l})^{\mathrm{t}}\ot I_{\beta}) (I_{\alpha} \ot I_l \ot T_{\beta}) \right]^{\mathrm{t}},
\end{align*}
where $\rho^{(n_{\gamma_l})}(u^{\gamma_l})$ is the unitary operator in $M_{n_{\gamma_l}} \ot B(\Hil)$ obtained via the matrix lifting of the representation $\rho$ and interpreted as an operator from $\bc^{n_{\gamma_l}} \ot \Hil$ to $\bc^{n_{\gamma_l}} \ot \Hil$, and ${\mathrm{t}}$ denotes matrix transposition. Note that $(u^{\gamma_l})^{\mathrm{t}}$ is unitary because $\G$ is of Kac type.

Analyse for example the expression $\left((\wt{T_l}^* \ot I_{\beta}) \Sigma (I_l \ot T_{\beta})\right)^{\mathrm{t}}_{(j,k), (r,s)}$. It is equal to
\begin{align*} \langle e_r \ot e_s, (\wt{T_l}^* \ot I_{\beta}) \Sigma (I_l \ot T_{\beta})& (e_j \ot e_k) \rangle =
\langle  (\wt{T_l} \ot I_{\beta})(e_r \ot e_s), \Sigma (I_l \ot T_{\beta}) (e_j \ot e_k) \rangle
\\&=
\sum_{a,b} \langle  c((u^{\gamma_l}_{ar})^*) \ot e_a \ot e_s, \Sigma (e_j \ot c(u^{\beta}_{kb}) \ot e_b) \ra
\\&=  \sum_{a,b} \langle c((u^{\gamma_l}_{ar})^*) \ot e_a \ot e_s, c(u^{\beta}_{kb}) \ot e_j \ot  e_b \ra =
\la c((u^{\gamma_l}_{jr})^*), c(u^{\beta}_{ks})\ra,
\end{align*}
and similar computations allow us to verify the `matricial' formula for $V^{(l)}$ given above.

Recalling that the sequence $(c_l)_{l=1}^{\infty}$ converges monotonically to infinity, to finish the proof of the lemma it suffices to note that the norms of $T_l$ and $\wt{T_l}$ are equal to $(2c_l)^{\frac{1}{2}}$. Indeed, compute for example
\begin{align*} (T_l^* T_l)_{i,j} &= \la T_l e_i, T_l e_j \ra = \sum_{a,b=1}^{n_{\gamma_l}} \la c(u^{\gamma_l}_{ia}) \ot e_a, c(u^{\gamma_l}_{jb}) \ot e_b \rangle
\\&= \sum_{a=1}^{n_{\gamma_l}}  \la c(u^{\gamma_l}_{ia}), c(u^{\gamma_l}_{ja}) \rangle =
\sum_{a=1}^{n_{\gamma_l}}  (- L((u^{\gamma_l}_{ia})^* u^{\gamma_l}_{ja}) + \hCou(u^{\gamma_l}_{ja}) \ol{L(u^{\gamma_l}_{ia})} + L(u^{\gamma_l}_{ja}) \ol{\hCou(u^{\gamma_l}_{ia})})
\\&= - L(\delta_{ij}\one) + \ol{L(u^{\gamma_l}_{ij})} + L(u^{\gamma_l}_{ji})= 2c_l \delta_{i,j},
\end{align*}
where we use the unitarity of $(u^{n_{\gamma}})^{\mathrm{t}}$ and the fact that $L(\one)=0$. Similarly we compute
\begin{align*} (\wt{T_l}^* \wt{T_l})_{i,j} &= \la \wt{T_l} e_i, \wt{T_l} e_j \ra = \sum_{a,b=1}^{n_{\gamma_l}} \la c((u^{\gamma_l}_{ai})^*) \ot e_a, c((u^{\gamma_l}_{bj})^*) \ot e_b \rangle
\\&= \sum_{a=1}^{n_{\gamma_l}} \la c((u^{\gamma_l}_{ai})^*), c((u^{\gamma_l}_{aj})^*) \rangle =
\sum_{a=1}^{n_{\gamma_l}}  (- L(u^{\gamma_l}_{ai} (u^{\gamma_l}_{aj})^*) + \ol{\hCou(u^{\gamma_l}_{aj})} L(u^{\gamma_l}_{ai}) + \ol{L(u^{\gamma_l}_{aj}) }
\hCou(u^{\gamma_l}_{ai}))
\\&= - L(\delta_{ij}\one)  + L(u^{\gamma_l}_{ji}) + \ol{L(u^{\gamma_l}_{ij})}
= 2c_l \delta_{i,j}. \qedhere
\end{align*}
\end{proof}

\begin{lem}\label{KeyLemma}
Let $\QG$ be a second countable discrete unimodular quantum group with a low dual and let $L$ be a central strongly unbounded $\widehat{S}^u$-invariant generating functional on $\Pol(\hQG)$. Consider the convolution semigroup of states $(\mu_t)_{t\geq 0}$ on $\CU{\hQG}$ generated by $L$ (see \prettyref{lem:convgen}). Let $\pi_t: \CU{\hQG}\to B(\Hil_t)$ be the corresponding GNS representation. Then for $t>0$, $\pi_t \star \pi_t	$ does not contain non-zero invariant vectors.
\end{lem}

\begin{proof}
Fix $t > 0$. Denote the GNS triple of $\mu_t$ by $(\pi_t, \Hil_t, \Omega_t)$. Then the representation $\pi_{t} \star \pi_{t}$ acts on the Hilbert space $\Hil_t \ot \Hil_t$ and $\Hil_0:=\linspan \{\pi_t(a) \Omega_t \ot \pi_t(b) \Omega_t: a, b \in \Pol(\hQG)\}$ is a dense subspace of $\Hil_t \ot \Hil_t$. It suffices then to show that the distance of the unit sphere of $\Hil_0$ from the set of invariant vectors for $\pi_t \star \pi_t$ of norm $1$ is non-zero (we will show it is in fact equal to $1$). To this end consider $\zeta \in \Hil_0$, $\zeta= \sum_{i=1}^k \pi_t(a_i) \Omega_t \ot \pi_t (b_i) \Omega_t$ for some $k \in \bn$, $a_1,\ldots, a_k, b_1, \ldots, b_k \in \Pol(\hQG)$, $\|\zeta \|=1$.

Let $(\gamma_l)_{l \in \bn}$ be a sequence of elements of $\Irred{\hQG}$ such that  we have $L^{\gamma_l}=c_lI_{M_{n_{\gamma_l}}}$ ($l \in \bn$), with $(c_l)_{l=1}^{\infty}$ a sequence of positive numbers increasing to infinity. Define $z_l:= u_{11}^{\gamma_l}$. We are interested in the following expression, calculated using \prettyref{lem:tensprod} and \prettyref{eq:fin_dim_reps}:
\begin{align*} \|(\pi_t \star \pi_t)(z_l)\zeta& - \widehat{\epsilon}(z_l) \zeta\|^2 =  \|(\pi_t \star \pi_t)(z_l)\zeta -  \zeta\|^2  
\\& =  \|\zeta\|^2 + \|(\pi_t \star \pi_t)(z_l)\zeta\|^2 - 2 \Ree \langle \zeta, (\pi_t \star \pi_t)(z_l) \zeta \rangle \\
& \geq 1 - 2 \sum_{i,j=1}^k \Ree \langle\pi_t(a_j) \Omega_t \ot \pi_t(b_j) \Omega_t, (\pi_t \star \pi_t)(z_l) [\pi_t(a_i) \Omega_t \ot  \pi_t(b_i) \Omega_t]  \rangle
\\&= 1 - 2 \sum_{i,j=1}^k \sum_{r=1}^{n_{\gamma_l}} \Ree \la \pi_t(a_j) \Omega_t \ot \pi_t(b_j) \Omega_t, (\pi_t (u_{r1}^{\gamma_l}) \ot \pi_t(u^{\gamma_l}_{1r})) [\pi_t(a_i) \Omega_t \ot \pi_t(b_i) \Omega_t]  \rangle \\
&= 1 - 2 \sum_{i,j=1}^k \sum_{r=1}^{n_{\gamma_l}} \Ree [\mu_t (a_j^* u_{r1}^{\gamma_l} a_i ) \mu_t (b_j^* u_{1r}^{\gamma_l} b_i)].
\end{align*}
The sum $\sum_{r=1}^{n_{\gamma_l}} \mu_t (a_j^* u_{r1}^{\gamma_l} a_i ) \mu_t (b_j^* u_{1r}^{\gamma_l} b_i)$  is a finite (independent of $l$, as we work with fixed $a_i$ and $b_i$) linear combination of the terms of the form
\[
\sum_{r=1}^{n_{\gamma_l}} \mu_t ((u_{ip}^{\alpha})^*  u_{r1}^{\gamma_l} u_{ks}^{\beta}) \mu_t( (u_{i'p'}^{\alpha'})^*  u_{1r}^{\gamma_l} u_{k's'}^{\beta'}),
\]
each of which can by \prettyref{eq:fin_dim_reps} be in turn expressed as
\[ \sum_{r=1}^{n_{\gamma_l}} (\tu{e}^{-t V^{(l)}})_{(i,r,k),(p,1,s)} (\tu{e}^{-t (V')^{(l)}})_{(i',1,k'),(p',r,s')}, \]
where the matrices $V^{(l)}$ and $(V')^{(l)}$  are of the type introduced in the last lemma.

Here we can finally use the lowness assumption: as there exists $N$ such that $n_{\gamma_l}\leq N$ for all $l \in \bn$, it suffices to show that in fact each suitable matrix entry of the matricial sequence $(\tu{e}^{-t V^{(l)}})$  (indexed by $l$) tends to 0. We know however that this sequence even converges to $0$ in norm by Lemma \ref{KeyLemma0}.
\end{proof}

We are now ready to formulate the main result of this section, generalising the result of Bekka and Valette \cite{Bekka_Valette__prop_T_amen_rep}, \cite[Theorem 2.12.9]{Bekka_de_la_Harpe_Valette__book} to discrete unimodular quantum groups with low duals.

\begin{tw}\label{thm:T11difficult}
Let $\QG$ be a second countable discrete unimodular quantum group with a low dual. Then $\QG$ has Property (T) if and only if $\QG$ has Property (T)$^{1,1}$.
\end{tw}

\begin{proof}
The forward implication does not require discreteness or unimodularity, and was noted in Proposition \ref{prop:T11easy}.

Assume then that $\QG$ does not have Property (T).  We will show that there exists a weakly mixing representation of $\QG$ which has almost-invariant vectors. Indeed,  by \prettyref{thm:UnboundedCentral} there exists a central  strongly unbounded $\widehat{S}^u$-invariant generating functional $L$ on $\Pol(\hQG)$. Let  $(\mu_t)_{t\geq 0}$ be the convolution semigroup of states  on $\C^u(\hQG)$ generated by $L$ and for each $t>0$ let $\pi_t: \C^u(\hQG)\to B(\Hil_t)$ denote the GNS representation of $\mu_t$, and $U_t$ the associated representation of $\QG$. Since $L$ is $\widehat{S}^u$-invariant, each of
the states $\mu_t$ is also $\widehat{S}^u$-invariant. By \prettyref{lem:Contragredient-vs-antipode}, each of the representations $U_t$ satisfies condition $\mathscr{R}$ of \prettyref{def:cond_R}, and is thus unitarily equivalent to its contragredient.

Choose a sequence $(t_n)_{n=1}^{\infty}$ of positive real numbers convergent to $0$. Then using pointwise convergence of $\mu_t$ as $t \longrightarrow 0^+$ we can verify that the representation  $V:=\bigoplus_{n \in \bn} U_{t_n}$ has almost-invariant vectors (recall that the passage between $U$ and $\pi$ respects direct sums). We claim that $V$ is weakly mixing. If that was not the case, $ V \tp V^c$ would not be ergodic. This implies that for some $n, m \in \bn$ the representation $ U_{t_n} \tp U_{t_m}^c$ is not ergodic. Hence, $U_{t_n}$ is not weakly mixing, that is, $U_{t_n} \tp U_{t_n}^c \cong U_{t_n} \tp U_{t_n}$ admits a non-zero invariant vector (see \prettyref{sec:inv_alm_inv_vects} for all this), and thus so does $\phi_{U_{t_n} \tp U_{t_n}}$. By Lemma \ref{lem:tensprod} the last map is exactly $\pi_{t_n} \star \pi_{t_n}$, and Lemma \ref{KeyLemma} yields a contradiction.
\end{proof}

\begin{rem}\label{rem:T_1_1_self_conjugate}
In the above proof, since each of the representations $U_t$, $t > 0$, satisfies condition $\mathscr{R}$, the same is true for $V$.
Therefore, for a discrete unimodular quantum group that satisfies the assumptions of
\prettyref{thm:T11difficult}, Property (T) is equivalent to the following weakening of Property (T)$^{1,1}$:
each of its representations which satisfies condition $\mathscr{R}$ and has almost-invariant vectors is not weakly mixing.
\end{rem}

\begin{rem}\label{rem:T_rs}
Note that the low dual assumption was only used in the last part of the proof of Lemma \ref{KeyLemma}, and that in fact an analogous proof shows that for second countable discrete unimodular quantum group with a low dual Property (T) is equivalent to Property (T)$^{r,s}$ for any (equivalently, all) $r,s \in \bn$.

Observe also that the original definition of Property (T)$^{r,s}$ for locally compact groups in \cite[p.~294]{Bekka_Valette__prop_T_amen_rep} translates in our setting to the following condition: for every representation $U$ of $\G$ such that $U^{\tpsmall r}\tp (U^{c})^{\tpsmall s}$ has almost invariant vectors, $U^{\tpsmall r}\tp (U^{c})^{\tpsmall s}$ has a non-zero invariant vector. Clearly, if $U$ has almost invariant vectors, so does $U^c$, and thus so does $U^{\tpsmall r}\tp (U^{c})^{\tpsmall s}$. Consequently, if $\G$ has Property (T)$^{r,s}$ in the sense of \cite{Bekka_Valette__prop_T_amen_rep} as explained above, it has Property (T)$^{r,s}$ in the sense of \prettyref{def:T_r_s}. Therefore, for a discrete unimodular quantum group that satisfies the assumptions of
\prettyref{thm:T11difficult} and $r,s \in \bn$, Property (T) is also equivalent to Property (T)$^{r,s}$ in the sense of \cite{Bekka_Valette__prop_T_amen_rep}.
\end{rem}

As a corollary, we obtain the following generalisation of part of \cite[Theorem 2.5]{Kerr_Pichot__asym_abel_WM_T} to discrete unimodular quantum groups with low duals. It says that lack of Property (T) is equivalent to the genericity of weak mixing for representations on a fixed infinite-dimensional separable Hilbert space.

\begin{tw}\label{thm:dense_gdelta}
Let $\G$ be a second countable discrete unimodular quantum group with a low dual. Then
$\G$ does not have Property (T) if and only if the weakly mixing representations
form a dense set in the space $\RepGH$. If $\G$ has Property (T), then the ergodic representations (hence also the weakly mixing ones) form a nowhere dense closed set in $\RepGH$.
\end{tw}

\begin{proof}
If $\G$ fails (T) then $\G$ fails (T)$^{1,1}$ by Theorem \ref{thm:T11difficult} and so there is a
representation $U$ of $\G$ which has almost-invariant vectors but which is
weakly mixing.
Recall from \prettyref{sub:reps} that weak mixing of representations of $\QG$ is stable under tensoring by
arbitrary representations.
In conclusion, by \cite[Lemma~5.3]{Daws_Fima_Skalski_White_Haagerup_LCQG}, the collection of weakly mixing representations is dense in $\RepGH$.

Suppose that $\G$ has Property (T).
Then the set of ergodic representations is closed (thus, so is the set of weakly mixing representations, by \prettyref{lem:tp_cts}). Indeed, let $(Q,\e)$ be a Kazhdan pair of type 1 for $\G$  (see \prettyref{thm:kazpair}). Suppose that $(U_n)_{n=1}^\infty$ and $U$ are in $\RepGH$, $U_n \to U$ and all $U_n$ are ergodic. If $\xi$ is a unit vector invariant under $U$, then since convergence in $\RepGH$ implies strong convergence on $\Ltwo{\G} \tensor \Hil$, there is $n$ such that $\| U_n(\eta \tensor \xi) - \eta \tensor \xi \| < \e$ for every $\eta \in Q$, contradicting the ergodicity of $U_n$.

Furthermore, using a suitable Kazhdan pair of type 2 for $\G$, we deduce that there exists a neighbourhood $A_1$ of $\one$ in $\RepGH$ all of whose elements are not ergodic. Assume by contradiction that the set of ergodic representations in $\RepGH$ contains an open subset $A_2$. From \prettyref{lem:dense_orbit} it follows that there exists $U \in \RepGH$ whose equivalence class is dense in $\RepGH$. In particular, it intersects both $A_1$ and $A_2$, which is absurd.
\end{proof}

\section{Spectral gaps}\label{sec:spectral_gaps}

In this section we define the notion of a spectral gap for representations and actions of locally compact quantum groups and relate it to ergodic properties of the relevant actions. In conjunction with the results of the previous section it gives a characterisation of Property (T) for a second countable discrete unimodular quantum group with a low dual in terms of the relation between arbitrary invariant states and \emph{normal} invariant states for actions of $\QG$, which extends the analogous classical result of \cite{Li_Ng__spect_gap_inv_states}.

\begin{defn}
We say that a representation $U$ of a locally compact quantum group has \emph{spectral gap}
if the restriction of $U$ to $\inv(U)^{\perp}$ does not have almost-invariant
vectors. An action of a locally compact quantum group on a von Neumann algebra is said to have
\emph{spectral gap} if its implementing unitary has spectral gap.\end{defn}

Evidently, a locally compact quantum group has Property (T) if and only if each of its representations has spectral gap. Recall that given a representation $U$, $p^U$ denotes the orthogonal projection onto the invariant vectors of $U$.

\begin{lem}
\label{lem:spectral_gap}Let $U$ be a representation of
a locally compact quantum group $\G$ on a Hilbert space $\H$ and $\phi_{U}$ be the associated
representation of $\CzU{\widehat{\G}}$. Then $U$ does not have spectral
gap if and only if there exists a state $\Psi$ of $\mathcal{B}(\H)$
satisfying $(\i\tensor\Psi)(U)=\one$ and $\Psi(p^{U})=0$, and if
$p^{U}\neq0$, this is equivalent to the condition $p^{U}\notin\Img\phi_{U}$.\end{lem}
\begin{proof}
Write $U_{1}$ for the restriction of $U$ to $\inv(U)^{\perp}$,
and let $\phi_{U_{1}}=\phi_{U}(\cdot)|_{\inv(U)^{\perp}}$ be the
associated representation of $\CzU{\widehat{\G}}$. Then $U\cong\one_{\mathcal{B}(\inv(U))}\oplus U_{1}$.
By \prettyref{lem:almost_inv_vec}, $U$ does not have spectral gap
if and only if there exists a state $\Psi$ of $\mathcal{B}(\inv(U)^{\perp})$
with $(\i\tensor\Psi)(U_{1})=\one_{\Linfty{\G}}$. This is plainly
the same as the existence of a state state $\Psi$ of $\mathcal{B}(\H)$
with $(\i\tensor\Psi)(U)=\one_{\Linfty{\G}}$ and $\Psi(p^{U})=0$.

Assume now that such a state $\Psi$ exists. If $p^{U}\in\Img\phi_{U}$, then
$0=\Psi(p^{U})p^{U}=p^{U}p^{U}=p^{U}$ by \prettyref{lem:almost_inv_vec}.
Hence $p^{U}=0$.

Conversely, if $p^{U}\notin\Img\phi_{U}$, consider the linear map
$j:\Img\phi_{U}\to\mathcal{B}(\H)$, $j(\phi_{U}(a)):=\phi_{U}(a)(\one-p^{U})=\phi_{U}(a)-\widehat{\epsilon}(a)p^{U}$,
$a\in\CzU{\widehat{\G}}$. It is an injective $*$-homomorphism (injectivity follows from $p^{U}\notin\Img\phi_{U}$) whose
image is $(\Img\phi_{U})(\one-p^{U})$, the subspace we identify with $\Img\phi_{U_{1}}$.
Therefore, fixing a unit vector $\z\in\inv(U)$, the state $\Psi_{1}:=\om_{\z}\circ j^{-1}$
of $\Img\phi_{U_{1}}$ satisfies $\Psi_{1}\circ\phi_{U_{1}}=\widehat{\epsilon}$.
Hence, from \prettyref{lem:almost_inv_vec}, $\phi_{U_{1}}$ has almost-invariant
vectors, namely, $U$ does not have spectral gap.\end{proof}

\begin{lem}
\label{lem:almost_inv_vec_act}Let $\a$ be an action of a locally compact quantum group $\G$
on a von Neumann algebra $N\subseteq B(\Kil)$. If $\a$ is implemented by a unitary $V \in \mathcal{B}(\Ltwo{\QG} \ot \Kil)$ and
a state $\Psi$ of $\mathcal{B}(\Kil)$ satisfies $(\i\tensor\Psi)(V)=\one$,
then the restriction $\Psi|_{N}$ is invariant under $\a$.\end{lem}
\begin{proof}
The argument is as in \cite[Theorem 3.2]{Bedos_Tuset_2003}. Observe
that for a state $\Psi$ of $\mathcal{B}(\Kil)$ as above, $V$ belongs
to the multiplicative domain of the unital completely positive map
$\i\tensor\Psi$ \cite[Theorem 3.18]{Paulsen__book_CB_maps_and_oper_alg}.
Consequently,
\[
(\i\tensor\Psi|_{N})\a(x)=(\i\tensor\Psi)(V^{*}(\one\tensor x)V)=(\i\tensor\Psi)(V^{*})(\i\tensor\Psi)(\one\tensor x)(\i\tensor\Psi)(V)=\Psi|_{N}(x)\one
\]
for every $x\in N$.\end{proof}

We are ready for the first of the main results of this section.

\begin{thm}
\label{thm:actions_spectral_gap}Let $\G$ be a locally compact quantum group acting on a von
Neumann algebra $N$ by an action $\a:N\to\Linfty{\G}\tensorn N$.
Write $U\in \M(\Cz{\G}\tensor\mathcal{K}(\Ltwo N))$ for the implementing
unitary of $\a$, and let $\phi_{U}$ be the associated representation
of $\CzU{\widehat{\G}}$. Consider the following conditions:
\begin{enumerate}
\item \label{enu:action_spectral_gap_1}$\a$ has spectral gap;
\item \label{enu:action_spectral_gap_2}every state $\Psi$ of $\mathcal{B}(\Ltwo N)$
with $(\i\tensor\Psi)(U)=\one$ satisfies $\Psi(p^{U})\neq0$;
\item \label{enu:action_spectral_gap_3}$p^{U}\in\Img\phi_{U}$;
\item \label{enu:action_spectral_gap_4}every state of $N$ that is invariant
under $\a$ is the weak$^{*}$-limit of a net of normal states of
$N$ that are invariant under $\a$.
\end{enumerate}

Then \prettyref{enu:action_spectral_gap_1} is equivalent to \prettyref{enu:action_spectral_gap_2},
and they are equivalent to \prettyref{enu:action_spectral_gap_3}
when $p^{U}\neq0$. If $\G$ is a discrete quantum group, then \prettyref{enu:action_spectral_gap_1}$\implies$\prettyref{enu:action_spectral_gap_4},
and \prettyref{enu:action_spectral_gap_4}$\implies$\prettyref{enu:action_spectral_gap_1}
when $d_{\left\Vert \cdot\right\Vert }(p^{U},N)<\frac{1}{2}$.

\end{thm}
\begin{proof}
The statements about the implications between \prettyref{enu:action_spectral_gap_1},
\prettyref{enu:action_spectral_gap_2} and \prettyref{enu:action_spectral_gap_3}
are just \prettyref{lem:spectral_gap}.

Assume henceforth that $\G$ is a discrete quantum group. For \prettyref{enu:action_spectral_gap_1}$\implies$\prettyref{enu:action_spectral_gap_4},
suppose that $\a$ has spectral gap and $m$ is an $\a$-invariant
state of $N$. By \prettyref{prop:act_pos_cone_applications}, \prettyref{enu:act_pos_cone_applications_1}, there
is a net of unit vectors $\left(\z_{i}\right)_{i \in \Ind}$ in $\Ltwo N$
that is almost invariant under $U$ and such that $\om_{\z_{i}} \stackrel{i \in \Ind}{\longrightarrow} m$
in the weak$^{*}$ topology of $N^{*}$. The net $\left((\one-p^{U})\z_{i}\right)_{i \in \Ind}$
in $\inv(U)^{\perp}$ must converge to zero for $U$ has spectral gap -- otherwise, by normalising and passing to a subnet, we get a net in $\inv(U)^{\perp}$ that is almost invariant under $U$. Therefore, letting
$\z_{i}':=\frac{1}{\left\Vert p^{U}\z_{i}\right\Vert }p^{U}\z_{i}$
for every $i$, $m$ is the weak$^{*}$-limit of the net $\left(\om_{\z_{i}'}\right)_{i \in \Ind}$
of normal states of $N$ that are invariant under $\a$.

\prettyref{enu:action_spectral_gap_4}$\implies$\prettyref{enu:action_spectral_gap_2}
when $d_{\left\Vert \cdot\right\Vert }(p^{U},N)<\frac{1}{2}$: a vector
$\z$ in the positive cone $\mathcal{P}$ of $N$ in $\Ltwo N$ is
invariant under $U$ if and only if $\om_{\z}$ is invariant under
$\a$. Indeed, the forward implication follows from the definitions for arbitrary
$\z\in\Ltwo N$, while the backward implication follows from \prettyref{prop:act_pos_cone_applications},
\prettyref{enu:act_pos_cone_applications_2}. If $\Psi$ is a state
of $\mathcal{B}(\Ltwo N)$ with $(\i\tensor\Psi)(U)=\one$, then from
\prettyref{lem:almost_inv_vec_act} and \prettyref{enu:action_spectral_gap_4},
$\Psi|_{N}$ is the weak$^{*}$-limit of a net of normal states of
$N$ that are invariant under $\a$, say $\left(\om_{\z_{i}}\right)_{i \in \Ind}$
with $\z_{i}\in\mathcal{P}$ for every $i\in \Ind$. So for every
$i\in \Ind$, $\z_{i}$ is invariant under $U$, i.e., $\om_{\z_{i}}(p^{U})=1$.
Let $q\in N$ be such that $\left\Vert p^{U}-q\right\Vert <\frac{1}{2}$.
Since $\Psi|_{N}=\text{weak}^{*}-\lim_{i \in \Ind}\om_{\z_{i}}$, we have $\Psi(q)-1=\lim_{i \in \Ind}\om_{\z_{i}}(q-p^{U})$,
hence $\left|\Psi(q)-1\right|<\frac{1}{2}$. We conclude that
\[
\left|\Psi(p^{U})\right|\geq1-\left|\Psi(p^{U}-q)\right|-\left|\Psi(q)-1\right|>1-\frac{1}{2}-\frac{1}{2}=0,
\]
yielding \prettyref{enu:action_spectral_gap_2}.\end{proof}

\begin{rem}
In the proof of \prettyref{enu:action_spectral_gap_1}$\implies$\prettyref{enu:action_spectral_gap_4}
we were limited to discrete quantum groups only because of \prettyref{prop:act_pos_cone_applications}.
However, \prettyref{prop:act_pos_cone_applications} holds more generally
than it is stated (although we do not know precisely to what extent). This does not contradict the example at the top
of \cite[p.~4919]{Li_Ng__spect_gap_inv_states}, explaining that the analogue of \prettyref{enu:action_spectral_gap_1}$\implies$\prettyref{enu:action_spectral_gap_4} fails when $\a$ is the left translation action of the circle group, because the notion of invariance of (non-normal) states under an action used in \cite{Li_Ng__spect_gap_inv_states} is weaker than ours; see \prettyref{rem:inv_states}. Using our terminology, \prettyref{enu:action_spectral_gap_1}$\implies$\prettyref{enu:action_spectral_gap_4} holds trivially whenever $\G$ is a compact quantum group and $\a := \Delta$.
\end{rem}

The next theorem characterises Property (T) of certain discrete unimodular quantum groups in terms of their actions on von Neumann algebras.

\begin{thm}
\label{thm:property_T_chars}Let $\G$ be a second countable discrete unimodular quantum group with a low dual. Then the following are equivalent:
\begin{enumerate}
\item \label{enu:property_T_chars_1}$\G$ has Property (T);
\item \label{enu:property_T_chars_2}for every action $\a$ of $\G$ on
a von Neumann algebra $N$, every state $\Psi$ of $\mathcal{B}(\Ltwo N)$
with $(\i\tensor\Psi)(U)=\one$ satisfies $\Psi(p^{U})\neq0$, where
$U$ is the unitary implementation of $\a$;
\item \label{enu:property_T_chars_3}for every action $\a$ of $\G$ on
a von Neumann algebra $N$, every state of $N$ that is invariant
under $\a$ is the weak$^{*}$-limit of a net of normal states of
$N$ that are invariant under $\a$;
\item \label{enu:property_T_chars_4}every action $\a$ of $\G$ on $\mathcal{B}(\K)$,
for some Hilbert space $\K$, having no $\a$-invariant normal states
of $\mathcal{B}(\K)$, has no $\a$-invariant states of $\mathcal{B}(\K)$.
\end{enumerate}
\end{thm}
\begin{proof}
By the definition of Property (T) and \prettyref{thm:actions_spectral_gap},
we have \prettyref{enu:property_T_chars_1}$\implies$\prettyref{enu:property_T_chars_2}$\implies$\prettyref{enu:action_spectral_gap_3}.
Also \prettyref{enu:action_spectral_gap_3} implies \prettyref{enu:property_T_chars_4}.

\prettyref{enu:property_T_chars_4}$\implies$\prettyref{enu:property_T_chars_1}:
If $\G$ does not have Property (T), then by \prettyref{thm:T11difficult},
$\G$ does not have Property (T)$^{1,1}$. Let then $V\in\M(\cz{\G}\tensor\mathcal{K}(\K))$
be a representation of $\G$ on a Hilbert space $\K$ which
has an almost-invariant net $\left(\z_{i}\right)_{i \in \Ind}$ of unit vectors
but is weakly mixing. Consider the action $\a$ of $\G$ on $\mathcal{B}(\K)$
given by $\a(x):=V^{*}(\one\tensor x)V$, $x\in\mathcal{B}(\K)$.
Then any weak$^{*}$-cluster point of $\left(\om_{\z_{i}}\right)_{i \in \Ind}$
is an $\a$-invariant state of $\mathcal{B}(\K)$.

Nevertheless, $\a$ does not admit any normal invariant state. Indeed, by \prettyref{lem:V_V_bar_unit_impl}, the (canonical)
unitary implementation of $\a$ on $\K\tensor\overline{\K}$ is $V\tpr V^{c}\in\M(\cz{\G}\tensor\mathcal{K}(\K\tensor\overline{\K}))$.
Weak mixing of $V$ is equivalent to ergodicity of $V\tpr V^{c}$.
By \prettyref{prop:act_pos_cone_applications}, \prettyref{enu:act_pos_cone_applications_2},
any normal $\a$-invariant state would give a $V\tpr V^{c}$-invariant
unit vector in $\K\tensor\overline{\K}$. Hence, such a state cannot
exist.
\end{proof}

\section{A non-commutative Connes--Weiss theorem}\label{sec:Connes_Weiss}

In this short last section we provide a non-commutative version of a classical result of Connes and Weiss from \cite{Connes_Weiss}, showing that for a certain class of discrete quantum groups $\QG$, Property (T) may be characterised by the properties of the trace-preserving actions of $\QG$ on the von Neumann algebra $\textup{VN}(\mathbb{F}_{\infty})$. The key role in this result is played by the construction of `canonical actions' of quantum groups on the free Araki--Woods factors, due to Vaes.

Let us then recall the construction of Vaes \cite[Section 3]{Vaes_strict_out_act}, yielding
canonical actions, induced by certain representations,
of locally compact quantum groups on the free Araki--Woods factors of Shlyakhtenko \cite{Shlyakhtenko__free_quasi_free_states}
(see also \cite{Vaes__etats_quasi_libres_libres}). Let $\Tinv$ be an
involution on a Hilbert space $\K$, that is, a closed, densely-defined,
injective anti-linear operator on $\K$ that satisfies $\Tinv=\Tinv^{-1}$,
and let $\Tinv=\J Q^{1/2}$ be its polar decomposition. On the (full)
Fock space
\[
\mathcal{F}(K):=\CC\Omega\oplus\bigoplus_{n=1}^{\infty}\K^{\tensor n}
\]
(the unit vector $\Omega$ is called the vacuum vector) consider the left creation (shift) operators $\ell(\z)$, $\z\in\K$,
given by $\Omega\mapsto\z$ and $\K^{\tensor n}\ni\eta\mapsto\z\tensor\eta\in\K^{\tensor(n+1)}$,
and the operators
\[
s(\z):=\ell(\z)+\ell(\Tinv\z)^{*},\qquad\z\in D(\Tinv).
\]
The von Neumann algebra $\Gamma(\K_{\J},Q^{it})'':=\{s(\z):\z\in D(\Tinv)\}''$
on $\mathcal{F}(K)$ is called the \emph{free Araki--Woods von Neumann
algebra} associated with $\Tinv$. The vacuum vector $\Omega$ is generating and separating for $\Gamma(\K_{\J},Q^{it})''$.
Thus, the \emph{free quasi-free state} $\om_{\Omega}$ of $\Gamma(\K_{\J},Q^{it})''$
is faithful, and the above representation of $\Gamma(\K_{\J},Q^{it})''$
on $\mathcal{F}(\K)$ can be seen as its GNS representation with respect
to $\om_{\Omega}$.

When the dimension $\dim\K_{\J}$ of the real Hilbert space $\K_{\J}:=\{\z\in\K:\J\z=\z\}$
is at least $2$, the von Neumann algebra $\Gamma(\K_{\J},Q^{it})''$
is a factor. In particular, when $Q=\one$ and $n:=\dim\K_{\J}\geq2$,
we have $(\Gamma(\K_{\J},Q^{it})'',\om_{\Omega})\cong(\VN{\mathbb{F}_{n}},\tau)$,
where $\mathbb{F}_{n}$ is the free group on $n$ generators, $\VN{\mathbb{F}_{n}}$
is its group von Neumann algebra and $\tau$ is the (unique) tracial
state of $\VN{\mathbb{F}_{n}}$.
\begin{prop}[{\cite[Proposition 3.1]{Vaes_strict_out_act}}]
\label{prop:free_Araki_Woods_factor_action}In the above setting,
let $U\in\M(\Cz{\G}\tensor\mathcal{K}(\K))$ be a representation
of the locally compact quantum group $\G$ on $\K$ that satisfies
\begin{equation}
(\om\tensor\i)(U^{*})\Tinv\subseteq \Tinv(\overline{\om}\tensor\i)(U^{*})\qquad \forall_{\om\in\Lone{\G}}.\label{eq:free_Araki_Woods_factor_action_condition}
\end{equation}
Then the representation of $\G$ on $\mathcal{F}(\K)$ given
by
\[
\mathcal{F}(U):=\bigoplus_{n=0}^{\infty}U^{\tprsmall n},
\]
where $U^{\tprsmall0}:=\one\in\M(\Cz{\G}\tensor\mathcal{K}(\CC\Omega))$
and $U^{\tprsmall n}=\underset{n\text{ times}}{\underbrace{U\tpr\cdots\tpr U}}=U_{1(n+1)}\cdots U_{13}U_{12}\in\M(\Cz{\G}\tensor\mathcal{K}(\K^{\tensor n}))$,
$n\geq1$, induces an action $\a$ of $\G$ on $\Gamma(\K_{\J},Q^{it})''$
given by
\begin{equation}
\a_{U}(x):=\mathcal{F}(U)^{*}(\one\tensor x)\mathcal{F}(U),\qquad x\in\Gamma(\K_{\J},Q^{it})''.\label{eq:free_Araki_Woods_factor_action_def}
\end{equation}
Furthermore, the free quasi-free state is invariant under $\a_{U}$.
\end{prop}
We will need the following simple observation: when the assumptions
of \prettyref{prop:free_Araki_Woods_factor_action} hold, the implementing
unitary of $\a$ on $\mathcal{F}(\K)$ with respect to the vacuum
vector $\Omega$ is $\mathcal{F}(U)$. Indeed, by the foregoing and \prettyref{eq:unit_impl_inv_state_1}, one
should show that
\[
(\i\tensor\om_{y\Omega,x\Omega})(\mathcal{F}(U)^{*})=(\i\tensor\om_{y\Omega,\Omega})(\a(x))
\]
for every $x,y\in\Gamma(\K_{\J},Q^{it})''$. But this is clear from
the definition of $\mathcal{F}(U)$.

\begin{lem}\label{lem:cond_R_trivial_tau}
Let $\G$ be a locally compact quantum group with trivial scaling group and $U$ a representation of $\G$ on a Hilbert space $\K$. Then for an involutive anti-unitary $\J$ on $\K$, $U$ satisfies condition $\mathscr{R}$ of \prettyref{def:cond_R} with respect to $\J$ if and only if $U$ fulfils the assumptions of \prettyref{prop:free_Araki_Woods_factor_action}
with $\Tinv := \J$.
\end{lem}

\begin{proof}
As before, write $j$ for the $*$-anti-isomorphism on $\mc K(\K)$ given by $j(x) := \J x^* \J$, $x\in \mc K(\K)$. Now $(R \tensor j)(U) = U$ if and only if
$$(\om \tensor \rho)(U) =(\om\circ R)[(\i\tensor(\rho\circ j))(U)]\qquad \forall_{\om \in \Lone{\G}} \forall_{\rho \in \B(\K)_*},$$
while $U$ satisfies \eqref{eq:free_Araki_Woods_factor_action_condition} for $\Tinv := \J$ if and only if $(\om\tensor\i)(U)=\J(\overline{\om}\tensor\i)(U)\J$ for all $\om \in \Lone{\G}$, if and only if
$$(\om \tensor \rho)(U) = \om[(\i\tensor(\rho\circ j))(U^{*})] \qquad \forall_{\om \in \Lone{\G}} \forall_{\rho \in \B(\K)_*}.$$
Since the scaling group of $\G$ is trivial, we have $S = R$, so by \eqref{eq:reps_S} the right-hand sides of the last two equations are equal.
\end{proof}

The following is a non-commutative analogue of the Connes--Weiss theorem
on discrete groups \cite{Connes_Weiss}, \cite[Theorem 6.3.4]{Bekka_de_la_Harpe_Valette__book},
in which $\VN{\mathbb{F}_{\infty}} := \VN{\mathbb{F}_{\aleph_0}}$ substitutes commutative von Neumann
algebras. Recall Definitions \ref{def:ergodicity} and \ref{def:asympt_inv}.

\begin{thm}
\label{thm:Connes_Weiss}Let $\G$ be a second countable discrete unimodular quantum group
with a low dual. Then the following conditions are
equivalent:
\begin{enumerate}
\item \label{enu:Connes_Weiss_T}$\G$ has Property (T);
\item \label{enu:Connes_Weiss__erg}for every von Neumann algebra $N$ with
a faithful normal state $\theta$, every ergodic $\theta$-preserving
action of $\G$ on $N$ is strongly ergodic;
\renewcommand{\theenumi}{\tu{(b$'$)}}\renewcommand{\labelenumi}{\theenumi}
\item \label{enu:Connes_Weiss__erg_faw}condition \prettyref{enu:Connes_Weiss__erg}
is satisfied for $(N,\theta)=(\VN{\mathbb{F}_{\infty}},\tau)$;\renewcommand{\theenumi}{\tu{(c)}}
\item \label{enu:Connes_Weiss_wm}for every von Neumann algebra $N$ with
a faithful normal state $\theta$, every weakly mixing $\theta$-preserving
action of $\G$ on $N$ is strongly ergodic;\renewcommand{\theenumi}{\tu{(c$'$)}}
\item \label{enu:Connes_Weiss__wm_faw}condition \prettyref{enu:Connes_Weiss_wm}
is satisfied for $(N,\theta)=(\VN{\mathbb{F}_{\infty}},\tau)$.
\end{enumerate}
\end{thm}

\begin{proof}
\prettyref{enu:Connes_Weiss_T}$\implies$\prettyref{enu:Connes_Weiss__erg}:
if $\left(x_{\iota}\right)_{\iota \in \Ind}$ is a net of elements of $N$ which is asymptotically invariant and not
trivial, then letting $y_{\iota}:=x_{\iota}-\theta(x_{\iota})\one$,
we get $\gnsmap_{\theta}(y_{\iota})\in\gnsmap_{\theta}(\one)^{\perp}$
for all $\iota$ and $y_{\iota}\stackrel{\iota \in \Ind}{\longarrownot\longrightarrow}0$ strongly. By the paragraph
succeeding \prettyref{def:asympt_inv}, the restriction
of the implementing unitary of $\a$ to $\gnsmap_{\theta}(\one)^{\perp}$,
which is ergodic by assumption, has almost-invariant vectors, contradicting
Property (T).

\prettyref{enu:Connes_Weiss__erg} implies \prettyref{enu:Connes_Weiss__erg_faw}
and \prettyref{enu:Connes_Weiss_wm}, and either of them implies \prettyref{enu:Connes_Weiss__wm_faw}.

\prettyref{enu:Connes_Weiss__wm_faw}$\implies$\prettyref{enu:Connes_Weiss_T}:
suppose that $\G$ does not have Property (T). By \prettyref{thm:T11difficult},
$\G$ does not have Property (T)$^{1,1}$. So there exists a
representation $U\in\M(\cz{\G}\tensor\mathcal{K}(\K))$ of $\G$
on a Hilbert space $\K$ that satisfies condition $\mathscr{R}$, has a net $\left(\z_{i}\right)_{i \in \Ind}$ of unit vectors that is
almost-invariant under $U$ and is weakly mixing (see \prettyref{rem:T_1_1_self_conjugate}).
Since $\G$ is unimodular, the assumptions of \prettyref{prop:free_Araki_Woods_factor_action}
are fulfilled with $\Tinv$ being a suitable involutive anti-unitary $\J$ on $\K$ by \prettyref{lem:cond_R_trivial_tau}.
Consider the induced action $\a:=\a_{U}$ of $\G$ on the free Araki--Woods
factor $\Gamma(\K_{\J},\i)''\cong \VN{\mathbb{F}_{\infty}}$ given by
\prettyref{eq:free_Araki_Woods_factor_action_def}, and the canonical
free quasi-free state (the vector state $\om_{\Omega}$ of the vacuum
vector $\Omega$), which is invariant under $\a$, and which, in this
case, is just the canonical trace $\tau$ on $\VN{\mathbb{F}_{\infty}}$.

Consider the bounded net $\left(s(\z_{i})\right)_{i \in \Ind}$ in $\Gamma(\K_{\J},\i)''$.
For every $i\in\Ind$ we have $s(\z_{i})\Omega=\z_{i}$, so that
$\tau(s(\z_{i}))=0$ but $s(\z_{i}) \stackrel{i \in \Ind}{\longarrownot\longrightarrow} 0$ strongly.
In the proof of \cite[Proposition 3.1]{Vaes_strict_out_act}
it is observed that for all $\om\in\lone{\G}$ and $\z\in\K$,
\[
(\om\tensor\i)\a(s(\z))=s((\om\tensor\i)(U^{*})\z).
\]
Therefore, by almost invariance of $\left(\z_{i}\right)_{i \in \Ind}$ under $U$, we have
\[
(\om\tensor\i)\a(s(\z_{i}))-s(\z_{i})=s((\om\tensor\i)(U^{*})\z_{i}-\z_{i})\stackrel{i \in \Ind}{\longrightarrow}0
\]
in norm for every normal state $\om$ of $\Linfty{\G}$. In conclusion, $\left(s(\z_{i})\right)_{i \in \Ind}$ is asymptotically invariant under $\a$ and is non-trivial, and thus $\a$ is not strongly ergodic.

It is left to prove that the action $\a$ is weakly mixing. By the
observation succeeding \prettyref{prop:free_Araki_Woods_factor_action},
$\mathcal{F}(U)$ is the implementing unitary of $\a$. Restricting
$\mathcal{F}(U)$ to $\Omega^{\perp}\subseteq\mathcal{F}(\K)$, we
get the operator $Y:=\bigoplus_{n=1}^{\infty}U^{\tprsmall n}$. Now
$Y\tpr Y^{c}=\bigoplus_{n=1}^{\infty}U\tpr(U^{\tprsmall(n-1)}\tpr Y^{c})$.
But $U$ is weakly mixing, which, since $\G$ is of Kac type, is
equivalent to $U\tpr Z$ being ergodic for every representation
$Z$ of $\G$. Consequently, $Y\tpr Y^{c}$ is ergodic, i.e., $Y$
is weakly mixing (again, $\G$ being of Kac type).\end{proof}

Note that to the best of our knowledge the equivalence between \prettyref{enu:Connes_Weiss_T} and \prettyref{enu:Connes_Weiss__wm_faw} above is new also for classical discrete groups.

\renewcommand{\theequation}{A.\arabic{equation}}

\renewcommand{\thesection}{A}

\renewcommand{\thetw}{A.\arabic{tw}}

\section*{Appendix}

\addtocounter{section}{1}
\setcounter{tw}{0}

In this appendix we discuss the notion of a morphism between locally compact quantum groups having \emph{dense image}, as proposed in Definition \ref{Def:dense_image}. Begin by the following easy observation (see \cite[Theorem 1.1]{Daws_Fima_Skalski_White_Haagerup_LCQG}): for locally compact Hausdorff spaces $X$ and $Y$, there is a one-to-one correspondence between continuous maps $\phi: X \to Y$ and morphisms $\Phi:\C_0(Y)\to \C_b(X)$ given by $\Phi(f) = f \circ \phi$, $f \in \C_0(Y)$; a map $\phi$ has dense image if and only if the associated $\Phi$ is injective. Considering this fact one might expect that for locally compact quantum groups $\QG$ and $\QH$, a morphism from $\QH$ to $\G$ should be viewed as having a dense image if the associated morphism $\pi \in \Mor(\C_0^u(\QG), \C_0^u(\QH))$ is injective. This is however not satisfactory, as there exists a lattice $\Gamma$ in a Lie group $G$ such that the natural `inclusion' morphism in $\Mor(\cst(\Gamma), \cst(G))$ is not injective (\cite{Bekka_Valette__lattices_semi_simple_Lie}) (and we would like to view the natural morphism from $\hh G$ to $\hh \Gamma$ as having dense image; this is indeed the case -- see below).

This motivates the apparently more complicated Definition \ref{Def:dense_image}; in this appendix we present its several equivalent characterisations.

\begin{tw}\label{thm:dense_image_char}
Let $\QG$, $\QH$ be locally compact quantum groups and consider a morphism from $\QH$ to $\QG$ represented by $\pi \in \Mor(\C_0^u(\QG), \C_0^u(\QH))$ intertwining the respective coproducts,
with associated bicharacter $V\in \M(\C_0(\QH)\ot \C_0(\hQG))$.  The following
are equivalent:
\begin{enumerate}
\item\label{sp:one} the morphism in question has a dense image in the sense of the Definition \ref{Def:dense_image}, i.e.\ the map $\alpha:\Lone{\hQG}\rightarrow \M(\C_0^u(\QH))$,
$\omega\mapsto \pi( (\id\otimes\omega)\Ww^\G )$, is injective;
\item\label{sp:two} the map $\beta:\Lone{\hQG}\rightarrow \M(\C_0(\QH))$,
$\omega\mapsto \Lambda_\QH(\alpha(\omega)) = (\id\otimes\omega)(V)$, is injective;
\item\label{sp:three} $\cA_V := \{ (\zeta\otimes\id)(V) : \zeta\in
\Lone{\QH} \}$ is dense in the weak$^*$ topology of $L^\infty(\hQG)$;
\item\label{sp:four} the dual morphism $\widehat\pi:\C_0^u(\hQH) \rightarrow
\M (\C_0^u(\hQG))$ has strictly dense range;
\item\label{sp:five} the map $\gamma:\C_0^u(\hQG)^* \rightarrow \M(\C_0(\QH))$,
$\mu \mapsto \Lambda_\QH \pi \big( (\id\otimes\mu)\WW^\G \big)$ is
injective;
\item\label{sp:six} the map $\Lambda_{\hQG}\widehat\pi:\C_0^u(\hQH) \rightarrow
\M (\C_0(\hQG))$ has strictly dense range.
\end{enumerate}
\end{tw}

Let $G,H$ be locally compact groups. In the commutative case, let a morphism from $H$ to $G$ be represented by $\pi \in \Mor(\Cz{G},\Cz{H})$ with associated continuous map $\phi : H \to G$. As explained above, if $\phi$ has dense image, then the morphism has dense image in our sense. Conversely, if $\phi$ does not have dense image, let $H' := \overline{\phi(H)}$, and view $\VN{H'}$ as a proper weak$^*$-closed subspace (indeed, a von Neumann subalgebra) of $\VN{G} = \Linfty{\hh G}$. By Hahn--Banach separation, there is a function $f \neq 0$ in the Fourier algebra $A(G) \cong \VN{G}_*$ (see \cite{Eymard__Fourier_alg}), canonically embedded in $\Cz{G}$, with $f(H') = \{0\}$. Hence $\pi(f) = f \circ \phi = 0$. Therefore, the morphism does not have dense image in our sense.

In the cocommutative case, let a morphism from $\hh{G}$ to $\hh{H}$ be represented by $\pi \in \Mor(\cst(H),\cst(G))$. The dual morphism $\hh{\pi} \in \Mor(\Cz{G}, \Cz{H})$ is associated with a continuous map from $H$ to $G$, which is injective if and only if $\hh{\pi}$ has strictly dense image (again, see \cite[Theorem 1.1]{Daws_Fima_Skalski_White_Haagerup_LCQG}), if and only if the original morphism has dense image in our sense by \prettyref{thm:dense_image_char}.

We postpone the proof of \prettyref{thm:dense_image_char} to present first some technical results.  Note that the condition \prettyref{sp:three} in the above Theorem implies that the notion of the dense image is compatible with the concept of the \emph{closure of the image of a quantum group morphism}, as defined recently in \cite{Kasprzak_Khosravi_Soltan__int_act_QGs} -- indeed, according to \prettyref{sp:three} a morphism between $\QH$ and $\QG$ has dense image in the sense of Definition \ref{Def:dense_image} if and only if the closure of its image in the sense of \cite{Kasprzak_Khosravi_Soltan__int_act_QGs} is equal to $\QG$.

Let us now turn to look at the strict topology on multiplier algebras.

\begin{lem}\label{lem:stdensity}
Let $\sA,\sB$ be \cst-algebras and let $\theta:\sA\rightarrow\sB$ be
a surjective $*$-homomorphism.  Writing $\overline\theta$ for the strict
extension of $\theta$, if $X\subseteq \M(\sA)$ is a strictly dense \cst-algebra,
then also $\overline\theta(X)$ is strictly dense in $\M(\sB)$.
\end{lem}
\begin{proof}
Easy consequence of the strict version of the Kaplansky density theorem.
\end{proof}

\begin{prop}\label{prop:uni_str_dense}
Let $\G$ be a locally compact quantum group, let $(x_i)_{i \in \Ind}$ be a bounded net
in $\M(\C_0^u(\G))$, and let $x \in \M(\C_0^u(\G))$ be such that $\Lambda_\G(x_i)
\rightarrow \Lambda_\G(x)$ in $L^\infty(\G)$ for the $*$-strong topology.
Fix $\omega_0 \in L^1(\G)$ and set $y_i = (\id\otimes\omega_0\circ\Lambda_\G)
\Delta_u^\G(x_i)\in \M(\C_0^u(\G))$ and $y = (\id\otimes\omega_0\circ\Lambda_\G)
\Delta_u^\G(x) \in \M(\C_0^u(\G))$.  Then $y_i \stackrel{{i \in \Ind}}{\longrightarrow} y$ in the
strict topology on $\M(\C_0^u(\G))$.
\end{prop}
\begin{proof}
Recall that by \cite[Proposition 6.2]{Kustermans__LCQG_universal},
\[ (\id\otimes\Lambda_\G)\Delta_u^\G(x) = \Ww^*(\one\otimes \Lambda_\G(x))\Ww
\qquad (x\in \M(\C_0^u(\G))). \]
We may suppose that $\omega_0 = \omega_{\xi_0,\eta_0}$ for some $\xi_0, \eta_0 \in L^2(\G)$.
Fix $\e>0$, fix $a\in \C_0^u(\G)$, and choose a norm one compact operator
$\theta_0\in \mc K(L^2(\G))$ with $\theta_0(\eta_0)=\eta_0$.  As $\Ww\in \M(\C_0^u(\G)\otimes\mc K(L^2(\G)))$
we can find $\sum_{j=1}^n b_j \otimes \theta_j \in \C_0^u(\G)\otimes\mc K(L^2(\G))$ with
\[ \Big\| \Ww(a\otimes\theta_0) - \sum_{j=1}^n b_j \otimes \theta_j \Big\|
\leq \e. \]
We may suppose that $\C_0^u(\G)$ acts faithfully on a Hilbert space $\H$.
Then, for any $\xi,\eta\in\H$ with $\|\xi\| \|\eta\|=1$, $i \in \Ind$, 
\begin{align*}
\la \xi , (y_i-y)a \eta \ra
&= \la \xi \otimes \xi_0, \Ww^*(\one\otimes\Lambda_\G(x_i-x))
\Ww(a\eta\otimes\theta_0\eta_0) \ra.
\end{align*}
Up to an error not greater than  $\e \|\xi_0\| \|\eta_0\| \|x_i-x\|$, this is
\begin{align*}
\sum_{j=1}^n \la \xi \otimes \xi_0, \Ww^*(\one\otimes\Lambda_\G(x_i-x))
(b_j\eta \otimes \theta_j\eta_0) \ra.
\end{align*}
If $i$ is sufficiently large, then for each $j=1,\ldots,n$ we have that
$\|\Lambda_\G(x_i-x)\theta_j\eta_0\| \leq \e n^{-1} \|b_j\|^{-1}$, and
so the absolute value of the sum is dominated by
\[  \sum_{j=1}^n \|\xi_0\| \| b_j\| \|\Lambda_\G(x_i-x)\theta_j\eta_0\|
\leq \|\xi_0\| \e. \]
Thus
\[ |\la \xi , (y_i-y)a \eta \ra| \leq \e \|\xi_0\| (\|\eta_0\|
\|x_i-x\| + 1) \]
and so $\|(y_i-y)a\|$ is small for sufficiently large $i$, as required.

An analogous argument, using that $\Lambda_\G(x_i^*) \rightarrow \Lambda_\G(x^*)$ strongly,
shows that $\|a(y_i-y)\|\stackrel{{i \in \Ind}}{\longrightarrow}0$ as well.  Hence $y_i\stackrel{{i \in \Ind}}{\longrightarrow}y$ in
the strict topology, as required.
\end{proof}

\begin{proof}[Proof of Theorem~\ref{thm:dense_image_char}]
The equivalence \ref{sp:two}$\iff$\ref{sp:three} follows as for any element $\omega \in \Lone{\hQG}$ we have that $(\id \ot \omega)(V)= 0$ if and only if $\omega|_{\cA_V}=0$. The implication
\ref{sp:two}$\implies$\ref{sp:one} is obvious.
The canonical embedding of $\Lone{\hQG}$ into $\C_0^u(\hh\G)^*$ is
the composition of the restriction map with the adjoint
of the reducing morphism $\Lambda_{\hh\G}:\C_0^u(\hh\G)
\rightarrow \C_0(\hh\G)$.  Thus $\beta$ is the restriction of $\gamma$
to $\Lone{\hQG}$ (see \eqref{eq:morphs_bichars}), and hence \ref{sp:five}$\implies$\ref{sp:two}. The implication \ref{sp:four}$\implies$\ref{sp:six} follows from Lemma~\ref{lem:stdensity}.

If \ref{sp:four} holds, then as
$\{ (\omega\circ\Lambda_\QH \ot \id)\WW^\QH : \omega\in \Lone{\QH} \}$
is dense in $\C_0^u(\hh\QH)$, and as $(\id \ot \widehat\pi)\WW^\QH =
(\pi \ot \id)\WW^\G$, we see that
\[ \{ \widehat\pi((\omega\circ\Lambda_\QH\otimes\id)\WW^\QH)
: \omega\in \Lone{\QH} \}
= \{ (\omega\circ\Lambda_\QH\circ\pi \ot \id)\WW^\G)
: \omega\in \Lone{\QH} \} \]
is strictly dense in $\M (\C_0^u(\hh\G))$, and thus  also weak$^*$-dense as a subspace of $\C_0^u(\hh\G)^{**}$.  So for non-zero
$\mu\in \C_0^u(\hh\G)^*$, there is $\omega\in \Lone{\QH}$ with
\[ 0 \not= \mu\left((\omega\circ\Lambda_\QH\circ\pi \ot \id)\WW^\G\right)
= \gamma(\mu)(\omega), \]
and so $\gamma(\mu)\not=0$.  Hence \ref{sp:four}$\implies$\ref{sp:five} (so by the earlier reasoning \ref{sp:four}$\implies$\ref{sp:one}). In a similar way we can prove that \ref{sp:six}$\implies$\ref{sp:two}, arguing via a `reduced' version of the map $\gamma$ appearing in \ref{sp:five}.

Suppose now that \ref{sp:three} does not hold, so there is a non-zero
$\omega_0\in \Lone{\hQG}$ with $(\zeta\ot \omega_0)(V)=0$ for all
$\zeta\in \Lone{\QH}$.  As $(W^\QH_{12})^* V_{23} W^\QH_{12}
= (\Delta_\QH \ot \id)(V) = V_{13} V_{23}$ we see that for
$\zeta\in \Lone{\QH}$ and $\zeta'\in \B(\Ltwo{\QH})_*$,
\begin{align*}
(\zeta \ot\zeta' \ot \omega_0) (V_{23}W^{\QH}_{12} V_{23}^*)
=(\zeta \ot\zeta' \ot \omega_0) (W^{\QH}_{12} V_{13})
= (\zeta \ot\omega_0)((x \ot \one)V)
= (\zeta x\ot\omega_0)(V)= 0, \end{align*}
where in the above computation $x = (\id \ot \zeta')(W^\QH) \in \C_0(\QH)$.

Now let $\mu\in \C_0^u(\QH)^*$ and set $\tilde{x} = (\mu \ot\id)(\Ww^\QH)
\in \M(\C_0(\hh\QH)) \subseteq \Linfty{\hh\QH}$.  Then we can find a net
$(\zeta_i)_{i \in \Ind}$ in $\Lone{\QH}$ such that $(\zeta_i\otimes\id)(W^\QH) \rightarrow
\tilde{x}$ in the weak$^*$ topology.  It follows that for any $\zeta'\in \B(\Ltwo{\QH})_*$,
\[ (\mu \ot\zeta' \ot \omega_0) (V_{23}\Ww^{\QH}_{12} V_{23}^*) =
 (\zeta' \ot \omega_0)(V(\tilde{x} \ot \one) V^*)= \lim_{i \in \Ind}  (\zeta_i \ot\zeta' \ot \omega_0) (V_{23}W^{\QH}_{12} V_{23}^*)=0.\]
As $\Ww^\QH$ is a unitary, this also shows that that for any $\mu\in \C_0^u(\QH)^*$ and $\zeta'\in \B(\Ltwo{\QH})_*$ we have
\[ (\mu \ot\zeta' \ot \omega_0) ((\Ww^{\QH}_{12})^* V_{23}\Ww^{\QH}_{12} V_{23}^*)
= 0.  \]

Let $U = (\pi \ot \id)(\Ww^\G)$ so that $V=(\Lambda_\QH \ot \id)(U)$, and so
\[(\Ww^{\QH}_{12})^* V_{23}\Ww^{\QH}_{12} = (\id\otimes\Lambda_\QH\otimes\id)
(\Delta_u^{\QH} \ot \id)(U) = (\id\otimes\Lambda_\QH\otimes\id)(U_{13}U_{23})
=  U_{13} V_{23}. \]
Thus, for any state $\zeta'\in \B(\Ltwo{\QH})_*$, and any $\mu\in \C_0^u(\QH)^*$,
\[ 0 =  (\mu \ot\zeta' \ot \omega_0)(U_{13})=  (\mu \ot \omega_0)(U) = \mu ((\id \ot \omega_0)(U)) =
\mu (\pi ((\id \ot \omega_0)(\Ww^{\G})) )= \mu (\alpha(\omega_0)).\]
Hence $\alpha(\omega_0)=0$ and so \ref{sp:one} does not hold.
Thus \ref{sp:one}$\implies$\ref{sp:three}.

Thus it remains to show \ref{sp:three}$\implies$\ref{sp:four}.
We assume that $\cA_V$ is weak$^*$-dense in $L^\infty(\hh\G)$, and aim to show
that the morphism $\widehat\pi$ has strictly dense range.
Let now
\[ \Vv := (\Lambda_\QH\pi \ot \id)(\WW^\G) = (\id\otimes\hh\pi)(\wW^\QH) \]
denote the `other-sided' (compared to $U$ above) lift
of the bicharacter $V$, so that $\Vv \in \M(\C_0(\QH) \ot \C_0^u(\hQG))$ and $(\id \ot \Lambda_{\hh\G})(\Vv)=V$.  As
$\{ (\zeta\ot\id)(\wW^\QH) : \zeta\in \Lone\QH \}$ is dense in $\C_0^u(\hh\QH)$,
we see that
\[ \cA_{\Vv}:=\{ (\zeta \ot \id)(\Vv) : \zeta\in \Lone{\QH} \}\subseteq \M(\C_0^u(\hQG)) \]
is dense in the image of $\hh\pi$.  We conclude that $\hh\pi$ has strictly dense range
if and only if $\cA_{\Vv}$ is strictly dense, and that the (norm) closure of
$\cA_{\Vv}$ is a C$^*$-subalgebra of $\M(\C^u_0(\hh\G))$.
Notice also that $\cA_V = \Lambda_{\hQG}(\cA_{\Vv})$.

Now, $(\id \ot \Delta^{\hQG}_u)(\Vv) = \Vv_{13} \Vv_{12}$,  so applying
$(\id\ot\id\ot\Lambda_{\hh\G})$ by \eqref{reducedimplementation} we see that
$ V_{13} \Vv_{12} = (\Ww^{\hh\G}_{23})^* V_{13} \Ww^{\hh\G}_{23}$, so also $ \Vv_{12} = V_{13}^* (\Ww^{\hh\G}_{23})^* V_{13} \Ww^{\hh\G}_{23}. $
It follows that
\[ \cA_{\Vv} = \{ (\zeta \ot \id\ot \mu)\big(
 V_{13}^* (\Ww^{\hh\G}_{23})^* V_{13} \Ww^{\hh\G}_{23} \big) :
\mu\in\B(L^2(\G))_*, \zeta\in \Lone{\QH} \}. \]
As $V$ is unitary, the closure of $\cA_{\Vv}$ equals the closed linear span of
\begin{align*} \{ (\zeta \ot \id\ot \mu)\big(
& (\Ww^{\hh\G}_{23})^* V_{13} \Ww^{\hh\G}_{23} \big) :
\mu\in\B(L^2(\G))_*, \zeta\in \Lone{\QH} \} \\
&= \{ (\id\ot \mu)\big(
(\Ww^{\hh\G})^* (\one\otimes x) \Ww^{\hh\G} \big) :
\mu\in\B(L^2(\G))_*, x\in\cA_V \},
\end{align*}
and further is equal to the closed linear span of
\[ \{(\id \ot (\omega\circ\Lambda_{\hh\G})) \Delta_u^{\hQG}(y) :
y\in M(\C_0^u(\hh\G)), \Lambda_{\hh\G}(y)\in\cA_V, \omega\in \Lone{\hQG} \},\]
again using that $(\Ww^{\hh\G})^*(1\otimes\Lambda_{\hh\G}(\cdot))\Ww^{\hh\G}
= (\id\otimes\Lambda_{\hh\G})\Delta_u^{\hh\G}(\cdot)$.

Pick $\omega_1 \in L^1(\G)$ and set $y = (\omega_1\ot\id)(\wW^\G)
\in \C_0^u(\hh\G)$.  As \ref{sp:three} holds, the norm closure of $\cA_V$, which is a $\cst$-algebra, is
weak$^*$-dense in $L^\infty(\hh\G)$.  By the Kaplansky density theorem,
we can find a bounded net $(x_i)_{i \in \Ind}$ in $\cA_V$ with $x_i\stackrel{i \in \Ind}{\longrightarrow} \Lambda_{\hh\G}(y)$
in the $*$-strong topology.  We can `lift' $(x_i)_{i \in \Ind}$ to find a bounded net
$(y_i)_{i \in \Ind}$ in $\M(\C_0^u(\hh\G))$ with $\Lambda_{\hh\G}(y_i)=x_i$ for each $i$.
For any $\omega_0\in L^1(\hh\G)$, by Proposition~\ref{prop:uni_str_dense},
we know that
\[ (\id\ot\omega_0\circ\Lambda_{\hh\G})\Delta_u^{\hh\G}(y_i)
\stackrel{i \in \Ind}{\longrightarrow} (\id\ot\omega_0\circ\Lambda_{\hh\G})\Delta_u^{\hh\G}(y) \]
in the strict topology on $\M(\C_0^u(\hh\G))$.

By the above, we know that $(\id\ot\omega_0\circ\Lambda_{\hh\G})\Delta_u^{\hh\G}(y_i)$
is a member of the closure of $\cA_{\Vv}$ for each $i\in \Ind$.  Furthermore,
\begin{align*}
(\id\ot\omega_0\circ\Lambda_{\hh\G})\Delta_u^{\hh\G}(y)
&= (\omega_1\ot\id\ot\omega_0\circ\Lambda_{\hh\G})
	(\id\otimes\Delta_u^{\hh\G})(\wW^\G) \\
&= (\omega_1\ot\id\ot\omega_0\circ\Lambda_{\hh\G})( \wW^\G_{13} \wW^\G_{12} ) \\
&= (\omega_1z \ot \id)(\wW^\G),
\end{align*}
say, where $z = (\id\ot\omega_0\circ\Lambda_{\hh\G})(\wW^\G) =
(\id\ot\omega_0)(W^\G) \in \C_0(\G)$.  Thus, as $\omega_0$ and $\omega_1$ vary,
we see that $(\id\ot\omega_0\circ\Lambda_{\hh\G})\Delta_u^{\hh\G}(y)$
takes values in a dense subset of $\C_0^u(\hh\G)$.

We conclude that the strict closure of $\cA_{\Vv}$ contains all of $\C_0^u(\hh\G)$,
and hence that $\cA_{\Vv}$ is strictly dense in $\M(\C_0^u(\hh\G))$, as required.
\end{proof}

\bibliographystyle{plain}

\end{document}